\documentclass[a4paper,11pt]{article}
\usepackage[T1]{fontenc}
\usepackage{lmodern}

\usepackage{amsthm,amsmath,amsfonts,amssymb,bbm,mathrsfs}
\usepackage{mathtools}
\usepackage{enumitem}
\usepackage{url}
\usepackage{dsfont} 
\usepackage{appendix}
\usepackage{amsthm}
\usepackage{color}
\usepackage{graphicx}

\usepackage[colorlinks=true, linkcolor=blue, urlcolor=black, citecolor=blue,pdfstartview=FitH]{hyperref}

\usepackage[english]{babel}

\usepackage{caption,tikz,subfigure}
\usetikzlibrary{shapes}
\usetikzlibrary{patterns}

\usepackage[top=2.7cm, bottom=2.7cm, left=2.5cm, right=2.5cm]{geometry}

\makeatletter

\@addtoreset{equation}{section}
\makeatother

\setlist[enumerate,1]{label=(\roman*), font = \normalfont} 

\let\originalleft\left
\let\originalright\right
\renewcommand{\left}{\mathopen{}\mathclose\bgroup\originalleft}
\renewcommand{\right}{\aftergroup\egroup\originalright}

\newlength{\bibitemsep}
\newlength{\bibparskip}
\let\oldthebibliography\thebibliography
\renewcommand\thebibliography[1]{\oldthebibliography{#1}
	\setlength{\parskip}{\bibitemsep}
	\setlength{\itemsep}{\bibparskip}}

\newcommand{\N}{\mathbb{N}}

\newcommand{\R}{\mathbb{R}}

\renewcommand{\P}{\mathbb{P}}
\newcommand{\E}{\mathbb{E}}

\newcommand{\1}{\mathbbm{1}}

\newcommand{\cA}{\mathcal{A}}
\newcommand{\cB}{\mathcal{B}}

\newcommand{\cN}{\mathcal{N}}

\newcommand{\cD}{\mathcal{D}}

\newcommand{\cC}{\mathcal{C}}

\newcommand{\cE}{\mathcal{E}}

\newcommand{\cM}{\mathcal{M}}

\newcommand{\cG}{\mathcal{G}}

\newcommand{\Ec}[1]{\mathbb{E} \left[#1\right]}
\newcommand{\Pp}[1]{\mathbb{P} \left(#1\right)}
\newcommand{\Ecsq}[2]{\mathbb{E} \left[#1\middle|#2\right]}

\newcommand{\hEc}[1]{\widetilde{\mathbb{E}}_t \left[#1\right]}
\newcommand{\hPp}[1]{\widetilde{\mathbb{P}}_t \left(#1\right)}
\newcommand{\hEcsq}[2]{\widetilde{\mathbb{E}}_t 
	\left[#1\mathrel{}\middle|\mathrel{}#2\right]}

\newcommand{\e}{\mathrm{e}}
\newcommand{\ep}{\varepsilon}
\newcommand{\es}{\mathbb{E}}
\newcommand{\de}{\coloneqq}

\newcommand{\dt}{\mathrm{d}t}
\newcommand{\du}{\mathrm{d}u}
\newcommand{\dx}{\mathrm{d}x}

\newcommand{\gam}{\Gamma}

\newcommand{\q}{Q(\beta,\beta')}

\newcommand{\diff}{\mathop{}\mathopen{}\mathrm{d}}
\DeclareMathOperator{\Var}{Var}
\DeclareMathOperator{\Cov}{Cov}

\newcommand{\abs}[1]{\left\lvert#1\right\rvert}

\newcommand{\norme}[1]{\left\lVert#1\right\rVert}

\theoremstyle{plain}
\newtheorem{thm}{Theorem}[section]
\newtheorem{prop}[thm]{Proposition}
\newtheorem{lem}[thm]{Lemma}
\newtheorem{cor}[thm]{Corollary}

\theoremstyle{definition}

\theoremstyle{remark}
\newtheorem{rem}[thm]{Remark}
\newtheorem{ex}[thm]{Example}

\newtheorem{oquest}[thm]{Open question}
\usepackage{anyfontsize}

\title{Branching Brownian motion versus Random Energy Model in the supercritical phase: overlap distribution and temperature susceptibility}
\author{
	Benjamin \textsc{Bonnefont}\thanks{Université de Genève, Section de Mathématiques, Rue du Conseil-Général 7-9,
		1205 Genève, Suisse. Email: \texttt{benjamin.bonnefont@unige.ch}. Partially supported by GdR Branchement.},	
	~Michel \textsc{Pain}\thanks{Institut Math\'ematiques de Toulouse (UMR 5219), Universit\'e de Toulouse, France, CNRS. Email: \texttt{michel.pain@math.univ-toulouse.fr}. Partially supported by GdR Branchement and by PEPS JCJC grant n°246644.}
	~and Olivier \textsc{Zindy}\thanks{Sorbonne Universit\'e, Sorbonne Paris Cit\'e, CNRS, Laboratoire de Probabilit\'es Statistique et Mod\'elisation, LPSM, F-75005 Paris, France. Email: \texttt{olivier.zindy@sorbonne-universite.fr}. Partially supported by GdR Branchement.}
}

\begin{document}

\maketitle

\begin{abstract}
	In comparison with Derrida's REM, we investigate the influence of the so-called decoration processes arising in the limiting extremal processes of numerous log-correlated Gaussian fields. In particular, we focus on the branching Brownian motion and two specific quantities from statistical physics in the vicinity of the critical temperature. The first one is the two-temperature overlap, whose behavior at criticality is smoothened by the decoration process---unlike the one-temperature overlap which is identical---and the second one is the temperature susceptibility, as introduced by Sales and Bouchaud, which is strictly larger in the presence of decorations and diverges, close to the critical temperature, at the same speed as for the REM but with a different multiplicative constant. We also study some general decorated cases in order to highlight the fact that the BBM has a critical behavior in some sense to be made precise.
\end{abstract}


{
	\hypersetup{linkcolor=black}
	\tableofcontents
}

\section{Introduction and results}

\medskip


\subsection{Motivations}

In order to shed some light on the mysteries of the Parisi theory for mean field spin glasses, Derrida introduced in the
80's the {\it random energy models} (REMs) \cite{Derrida1981}, where the Gaussian energy levels are assumed to be independent, and its generalizations, the {\it generalized random energy models} (GREMs) \cite{Derrida1985}, whose correlations are given by a tree structure of finite depth. One question of central interest in spin glass theory is to understand the structure of pure states, which translates into the analysis of the extremal process in the language of extreme value theory of stochastic processes. 

\medskip

These two tractable models have been extensively studied and allowed, in particular, to investigate the phenomenon of {\it replica symmetry breaking}. Indeed  the REM seems to be the foremost representative of a universality class: in spin glass terminology one may call this the {\it 1-step replica symmetry breaking} (1-RSB) class, or REM-class. More precisely, a spin glass model displays a 1-RSB if there exists some critical $\beta_c>0$ such that, asymptotically, the overlap between two points chosen independently according to the Gibbs measure at inverse temperature $\beta>0$ is concentrated at 0 for $\beta \leq \beta_c$, but is supported by 0 and 1, when $\beta > \beta_c$. This phenomenon is a consequence of the fact that the REM-class undergo what physicists refer to as the {\it REM-freezing transition}: there is a phase transition for the free energy at some $\beta_c>0$ meaning that, for $\beta \le \beta_c$, there is an exponentially large number of configurations, with energy level strictly less than the extremes, contributing to free energy and the Gibbs measure is roughly uniformly distributed among such configurations while, for $\beta > \beta_c$, the relevant configurations are the ones with the largest energies and the free energy becomes dominated by a relatively small set of configurations.
Another striking fact characterizing the REM-class is that, for $\beta>\beta_c$, the ordered weights of the pure states under the Gibbs measure at inverse temperature $\beta>0$ follow asymptotically a \textit{Poisson--Dirichlet distribution} of parameter $\beta_c/\beta$.
We refer to Bolthausen \cite{bolthausensznitman2002,Bolthausen2007} and Kistler \cite{Kistler2015} for surveys on the REMs, GREMs and connections to spin glass theory. Finally, let us mention that, despite the simplicity of the freezing transition, rather sophisticated models are known, or conjectured, to belong to the REM-class, such is the case for the extremes of the Riemann zeta-function along the critical line, see Arguin  \cite{arguin_2016} for a survey.

\medskip

Natural hierarchical models with an infinite number of levels are the branching Brownian motion (BBM) and the
branching random walk (BRW), see e.g.\@ the seminal paper by Derrida and Spohn \cite{derridaspohn88}, who introduced directed polymers
on trees (a BRW with i.i.d.\@ displacements) as an infinite hierarchical extension of the GREMs for spin glasses. Physicists
suggested that Gaussian BRW and BBM should belong to a universality class called {\it log-correlated Gaussian fields}, or {\it $\log$-REMs}, which is in some sense a ``subclass'' of the REM-class. These models are not necessarily hierarchical but admit correlations that decay approximately like the logarithm of the inverse of the distance between index points. We refer to the works by Carpentier and Le Doussal \cite{carpentier-ledoussal2001}, Fyodorov and Bouchaud \cite{FyodorovBouchaud2008a,FyodorovBouchaud2008b} and Fyodorov \& al \cite{FyodorovLeDoussalRosso2009} for connection between $\log$-REMs and spin glass theory. Furthermore, these processes play an essential role in Liouville quantum gravity as well as models of three-dimensional turbulence or finance, see the review on Gaussian Multiplicative Chaos by Rhodes and Vargas \cite{rhodesvargas14} for discussions.

\medskip

Another line of research heavily relies on relations between $\log$-REMs and
traveling wave equations of Fisher-Kolmogorov-Petrovsky-Piscounov (FKPP) type.
Such a relation is exact for a particular instance of $\log$-REM: BBM mentioned above. Therefore BBM was studied over the last 50 years as a subject of interest in its own right, with contributions by McKean \cite{mckean75}, Bramson \cite{bramson78,bramson83}, Lalley and Sellke \cite{lalleysellke87}, Chauvin and Rouault \cite{chauvinrouault88,chauvinrouault90}. 
From the probabilistic point of view, the full picture was recently obtained by  A\"id\'ekon \& al \cite{abbs2013} and Arguin \& al \cite{abk2013} separately, while Cortines \& al  \cite{cortineshartunglouidor2019,cortineshartunglouidor2021} obtained a third description of the decorations' distribution.  Indeed it is now known that, in the thermodynamic limit, the extremal process tends to a randomly shifted decorated Poisson point process (SDPPP), see \cite{SubagZeitouni2015} for a precise definition. 
Compared with the Poisson point process which describes the extremes of the (uncorrelated) REM, the decorations appearing here describe the internal structure of blocks of extreme values which share a near {\it ancestor}, and thus are highly correlated.
For instance, this confirms the observation made by Bovier and Kurkova \cite{bovierkurkova2004-2} that BBM is a particularly
interesting example, lying right at the borderline where correlations begin to influence the behavior of the extremes of the process. Finally one can say that BBM is the prototype of hierarchical $\log$-REMs and therefore will naturally be the model of interest in this paper.

\medskip

Let us also mention that another important and famous example of log-correlated Gaussian fields is the two-dimensional discrete Gaussian Free Field (GFF),which possesses a complicated (non-hierachical) structure of extreme values, but it turns out to be possible to compare it with that of the branching random walk. By comparison to analoguous results for branching
random walks, many deep results have been recently established by Bolthausen, Deuschel and Giacomin \cite{bdg2001}, Daviaud \cite{Daviaud2006}, Arguin and Zindy \cite{arguinzindy2014,arguinzindy2015}, Bolthausen, Deuschel and Zeitouni \cite{bdz2011}, Bramson and Zeitouni \cite{bramsonzeitouni2012}, Bramson, Ding and Zeitouni \cite{bdz2016-2}, Biskup and Louidor \cite{biskuplouidor2016,biskuplouidor2018}. We refer to the excellent notes by Biskup \cite{biskup2020} for more details.

\medskip

While from a probabilistic point of view, the difference between BBM and REM is perfectly known, it is not the same when quantities from statistical physics are considered. Indeed it is not clear when the decorations of the log-correlated models are felt at the level of the Gibbs measure. For example, motivated by the question of temperature chaos for spin glasses (see the introduction of \cite{PainZindy21} for more comments on this), Bonnefont  \cite{bonnefont22} recently proved for the BBM that the distribution of the overlap between two points sampled independently according to Gibbs measures at different temperatures is different than
the REM's one, while it is the same for both models if the two points are sampled at the same temperature.
This raises questions about the influence of these decorations. Motivated by the recent work of Derrida and Mottishaw \cite{DerridaMottishaw_2021}, we seek to compare both models by studying carefuly the overlap distribution in the neighborhood of the critical temperature. Another quantity of interest is the notion of temperature susceptibility introduced and studied for the REM (with exponentially distributed energies) by Sales and Bouchaud \cite{salesbouchaud01}.

\subsection{Definitions and some results}

\paragraph{A probabilistic point of view.}

On one side, the (binary) branching Brownian motion (BBM) is a branching Markov process defined on some general probability space $(\Omega,\mathcal F, \P)$ as follows. Initially,
there is one single particle at the origin which moves according to a standard Brownian motion during an exponentially distributed time of parameter $1/2$%
\footnote{This choice is made such that the critical inverse temperature introduced later in this article and denoted by $\beta_c$ equals 1.} and
then splits into two new particles. These new particles start the same process from their place of birth, behaving independently of
the others and the system goes on indefinitely. Let $L_t$ denote the set of alive particles at time $t \ge 0$ and $h_t(x)$ the position of the particle $x$ at time $t$.
The position of the highest particle has been studied by Bramson \cite{bramson78,bramson83} and Lalley and Sellke \cite{lalleysellke87}. A new step has recently been taken with the proof of the convergence of the extremal process  in the space of Radon measures on $\R$ endowed with the vague topology, to a randomly
shifted decorated Poisson point process. More precisely, A\"id\'ekon et al.\@ \cite{abbs2013} and Arguin et al.\@ \cite{abk2013} proved simultaneously that
\begin{equation}
\label{eq:extremalprocessBBM}
\sum_{x \in L_t} \delta_{h_t(x)- t + \frac{3}{2} \log t -\log (c_{\mathrm{d}} D_{\infty})} \xrightarrow[t\to\infty]{{\rm (d)}} \sum_{i,j} \delta_{\xi_i + d_{ik}}, 
\end{equation}
 for some constant $c_{\mathrm{d}}>0$ and where $D_{\infty}\de\lim_{t \to +\infty} \sum_{x \in L_t} (t-h_t(x)) \, \e^{h_t(x)-t} >0$ (almost-surely) is the limiting {\it derivative martingale}, $\sum_{i\geq 1} \delta_{\xi_i}$ is a Poisson point process on $\R$ with intensity measure $\e^{-x} \diff x$ independent of $(\sum_{k\geq 0} \delta_{d_{ik}})_{i \ge 1}$, which are are i.i.d.\@ copies of a point process $\cD$ on $(-\infty,0]$ which has a.s.\@ an atom at~0. 
The point process $\cD$ is called \textit{decoration process}\footnote{We will also consider the case of {\it general decoration processes}, see Section \ref{sec:change_of_measure} and use the same notation $(\sum_{k\geq 0} \delta_{d_{ik}})_{i \ge 1}$ for general decoration processes. For clarity, one shall make precise (between brackets), in every statement/result, when one treats the BBM case or the general decorated case.}.

\medskip

On the other side, the REM of interest is defined as follows in order to be comparable to the BBM introduced above: it consists of $n_t:=\lfloor \e^{t/2} \rfloor$ i.i.d. centered Gaussian random variables of variance $t$, denoted by $(g_t(k) \, ; \, 1 \le k \le n_t)$. It is well known that the extremal process for the REM satisfies the following convergence in the space of Radon measures on $\R$ endowed with the vague topology:
\begin{equation}
\label{eq:extremalprocessREM}
\sum_{1 \le k \le n_t} \delta_{g_t(k)-t + \frac{1}{2} \log t - \log c_0}  \xrightarrow[t\to\infty]{{\rm (d)}} \sum_{i} \delta_{\xi_i}, 
\end{equation}
for some numerical positive constant $c_0>0$ and where again  the $(\xi_i)_{i\geq 1}$ are the atoms of a Poisson point process on $\R$ with intensity measure $\e^{- x} \diff x$. We refer to Kistler \cite{Kistler2015} for a recent survey on the REM. 

\medskip

Looking first at the convergences in Equations (\ref{eq:extremalprocessBBM}) and (\ref{eq:extremalprocessREM}) allows to compare the BBM and the REM from a probabilistic point of view: both model's maxima share the same first order $t$ but the correlations for the BBM start to affect the second order, namely $- \frac{3}{2} \log t$ for the BBM is smaller than $ - \frac{1}{2} \log t$ for the REM. And finally Equations (\ref{eq:extremalprocessBBM}) and (\ref{eq:extremalprocessREM}) also complete the picture by telling that the limiting extremal process for the REM is a standard Poisson point process while the BBM's one is a randomly shifted decorated Poisson point process.  The additional ingredient for the BBM is mainly the decoration process, which describes the internal structure of blocks of extremal particles sharing a near ancestor and thus highly correlated.

\medskip

\paragraph{A statistical physics approach.}

In statistical physics, it is common to consider first the {\it partition function} $Z_{t,{\mathrm{d}}}(\beta)$ of the model, here the BBM ($\beta$ stands for the inverse-temperature):
\begin{equation*}
Z_{t,{\mathrm{d}}}(\beta)\de\sum_{x \in L_t} \e^{\beta h_t(x)}, \qquad \forall  \, \beta >0,
\end{equation*}
and the {\it free energy}
\begin{equation*}
f_{t,{\mathrm{d}}}(\beta)\de \frac{1}{t}  \, \log Z_{t,{\mathrm{d}}}(\beta),  \qquad \forall  \, \beta >0. 
\end{equation*}
The subscript ``d'' stands for ``decorated'' and is here to highlight the difference with the REM (for which similar quantities are written without the subscript).
It is well known, see \cite{derridaspohn88}, that the BBM exhibits a phase transition at the level of the free energy and that this latter is the same as for the REM, namely
\begin{equation*}
f_{\mathrm{d}}(\beta)\de \lim_{t \to \infty} f_{t,{\mathrm{d}}}(\beta)=
 \begin{cases}
  \frac{1 + \beta^2}{2}, & {\rm if}
\ \beta < \beta_c\de1, \\
  \beta, &  {\rm if} \ \beta \ge \beta_c,
 \end{cases}
 \qquad 
 \text{ a.s. and in $L^1$.}
\end{equation*}
In particular, the model undergoes {\it freezing} above $\beta_c$ in the sense that the quantity $f_{\mathrm{d}}(\beta)/\beta$ is constant. 
More importantly, one considers the {\it Gibbs measure}; this is the random probability measure ${\cal G}_{\beta,t,{\mathrm{d}}}$ on $L_t$ defined, at inverse temperature $\beta>0$, by
\[
{\cal G}_{\beta,t,{\mathrm{d}}}(x) \de \frac{\e^{\beta h_t(x)}}{Z_{t,{\mathrm{d}}}(\beta)}, \qquad \forall \, x \in L_t.
\]
By design, the Gibbs measure concentrates at low temperature, i.e. when $\beta>\beta_c$,  on the high points of the Gaussian field. With spin glasses in mind, one also considers the normalized covariance or {\it overlap} between two particles $x,y \in L_t$ by 
\begin{equation*}
\label{eqn:overlapBBM}
q_{t,{\mathrm{d}}}(x,y) \de \frac{1}{t} \, \E\left[h_t(x) h_t(y)\right] = \frac{1}{t} \sup\left\{s \ge 0 \, : \, x,y \text{ share a common ancestor in } L_s \right\} \, .
\end{equation*}
A fundamental object, that records the correlations of high points, is the {\it distribution function of the overlap} sampled from the Gibbs measure, i.e. $\cG_{\beta,t,{\mathrm{d}}} \otimes \cG_{\beta',t,{\mathrm{d}}} (q_{t,{\mathrm{d}}}(u,v) \geq a)$ for any $a \in (0,1)$, where $u$ (respectively $v$) is sampled according to $ \cG_{\beta,t,{\mathrm{d}}}$ (respectively $ \cG_{\beta',t,{\mathrm{d}}}$). Bonnefont \cite{bonnefont22} recently proved that, if $\beta \leq \beta_c$ or $\beta' \leq \beta_c$, then $ \cG_{\beta,t,{\mathrm{d}}} \otimes \cG_{\beta',t,{\mathrm{d}}} (q_{t,{\mathrm{d}}}(u,v) \geq a)$ tends to $0$ in $L^1$ for all $a \in (0,1)$, while if $\beta > \beta_c$ and $\beta'  >\beta_c$, then, for all $a \in (0,1)$, 
\begin{equation} \label{def:Q_d}
 \cG_{\beta,t,{\mathrm{d}}} \otimes \cG_{\beta',t,{\mathrm{d}}} (q_{t,{\mathrm{d}}}(u,v) \geq a)
\xrightarrow[t\to\infty]{{\rm (d)}} \frac{\sum_{i} \left( \e^{\beta \xi_i} \sum_{k} \e^{\beta d_{ik}} \right)
\left( \e^{\beta' \xi_i} {\sum_{k} \e^{\beta' d_{ik}}}\right)}{\left(\sum_{i,k} \e^{\beta \xi_i} \e^{\beta d_{ik}}
\right)\left(\sum_{i,k} \e^{\beta' \xi_i} \e^{\beta' d_{ik}} \right)} 
\eqqcolon Q_{\mathrm{d}}(\beta,\beta'),
\end{equation}
where the $(\xi_i)_i$ and the $(d_{ik})_{i,k}$ were introduced for the description of the limiting extremal process of the BBM, see Equation (\ref{eq:extremalprocessBBM}). In other words, this result proves the convergence of the pushforward of the measure $ \cG_{\beta,t,{\mathrm{d}}} \otimes \cG_{\beta',t,{\mathrm{d}}}$ on $L_t^2$ by the function $q_{t,{\mathrm{d}}}$, which is a random measure on $[0,1]$. The limit is either $\delta_0$ if $\beta \wedge \beta' \leq \beta_c$, or $(1-Q_{\mathrm{d}}(\beta,\beta')) \delta_0 + Q_{\mathrm{d}}(\beta,\beta') \delta_1$ otherwise.
Note that, when $\beta > \beta_c$ and $\beta'  >\beta_c$,  the random variables $(\e^{\beta \xi_i} \sum_{k} \e^{\beta d_{ik}} / \sum_j \e^{\beta \xi_j} \sum_{k} \e^{\beta d_{jk}})_{i \geq 1}$ are the asymptotic weights of the clusters under $\cG_{\beta,t,{\mathrm{d}}},$ such that $Q_{\mathrm{d}}(\beta,\beta')$ is simply the probability of choosing two points in the same cluster (when they are chosen proportionally to their Gibbs weights with inverse temperature $\beta$ and $\beta'$ respectively). Observe that, in particular, this result implies absence of temperature chaos for the BBM.

For the REM, the overlap is simply given by
\begin{equation*}
\label{eqn:overlapREM}
q_{t}(k,\ell)\de  \frac{1}{t} \, \E\left[g_t(k) g_t(\ell)\right] = \1_{\{k=\ell\}}, \qquad \forall \, 1 \le k,\ell \le n_t,
\end{equation*}
and the Gibbs measure at inverse temperature $\beta>0$ is defined by 
\[
{\cal G}_{\beta,t}(k) \de \frac{\e^{\beta g_t(k)}}{Z_{t}(\beta)}, \qquad \forall \, 1 \le k \le n_t,
\]
where $Z_{t}(\beta)\de\sum_{ 1 \le k \le n_t} \e^{\beta g_t(k)}$. Kurkova \cite{kurkova2003} proved that, if $\beta \leq \beta_c$ or $\beta' \leq \beta_c$, then $ \cG_{\beta,t} \otimes \cG_{\beta',t} (q_{t}(u,v) \geq a)$ tends to $0$ in $L^1$ for all $a \in (0,1)$, while if $\beta > \beta_c$ and $\beta'  >\beta_c$, then, for all $a \in (0,1)$, 
\[
 \cG_{\beta,t} \otimes \cG_{\beta',t} (q_{t}(u,v) \geq a)
\xrightarrow[t\to\infty]{{\rm (d)}} \frac{\sum_{i} \e^{\beta \xi_i}  \e^{\beta' \xi_i}}{\left(\sum_{i} \e^{\beta \xi_i}
\right)\left(\sum_{i} \e^{\beta' \xi_i}  \right)} 
\eqqcolon Q(\beta,\beta').
\]
It is well known that in the case $\beta'=\beta$, the random variables $Q_{\mathrm{d}}(\beta,\beta)$ and $Q(\beta,\beta)$ have the same distribution (see \cite[Equation (1.2)]{PainZindy21} for more details). Therefore a natural question one may ask is whether $Q_{\mathrm{d}}(\beta,\beta')$ and $Q(\beta,\beta')$ still have the same distribution when $\beta \neq \beta' > \beta_c$. Following the work by Pain and Zindy \cite{PainZindy21} for the two-dimensional discrete Gaussian free field, Bonnefont \cite{bonnefont22} proved, for the BBM, that
\begin{equation} \label{eq:two-temp-overlap-inequality}
 \E \left[Q_{\mathrm{d}}(\beta,\beta')\right]
<\E \left[Q(\beta,\beta')\right],
\end{equation}
meaning that the answer is negative. In this paper, our aim is to study and compare the functions $\beta \mapsto \E \left[Q_{\mathrm{d}}(\beta,\beta')\right]$ and $\beta \mapsto \E \left[Q(\beta,\beta')\right]$, when $\beta'> \beta_c$ is fixed. We will specially focus on their behavior when $\beta$ tends to $\beta_c^+$.

Before introducing the second quantity of interest let us define, for any $\beta >\beta_c, $ the partition functions associated with both limiting extremal processes  
\begin{equation} \label{def:Z_d}
Z(\beta) \de \sum_{i\geq1} \e^{\beta \xi_i}, \qquad 
Z_{\mathrm{d}}(\beta) \de \sum_{i\geq 1} \sum_{k\geq 0} \e^{\beta (\xi_i+ d_{ik})}.
\end{equation}
Then, let us consider the correlation coefficient between the free energies at two temperatures, introduced by Fisher and Huse \cite{FisherHuse1991} and given for both models by
\[
\mathbf{C}(\beta,\beta')\de \frac{\Cov(\log Z(\beta),\log Z(\beta'))}{\sigma(\log Z(\beta))\, \sigma(\log Z(\beta'))}, \qquad 
\mathbf{C}_{\mathrm{d}}(\beta,\beta')\de \frac{\Cov(\log Z_{\mathrm{d}}(\beta),\log Z_{\mathrm{d}}(\beta'))}{\sigma(\log Z_{\mathrm{d}}(\beta))\, \sigma(\log Z_{\mathrm{d}}(\beta'))},
\]
where $\sigma(X) \coloneqq \sqrt{\Var(X)}$.
The {\it susceptibility to temperature} we are interested in is defined by Sales and Bouchaud \cite{salesbouchaud01} as the coefficient $\kappa_{\mathrm{d}}(\beta)$ for the BBM and $\kappa(\beta)$ for the REM such that
\begin{align*}
\mathbf{C}(\beta,\beta+h) & = 1-\kappa(\beta)\, h^2+ o(h^2), \qquad  h\rightarrow 0,  \\
\mathbf{C}_{\mathrm{d}}(\beta,\beta+h) & = 1-\kappa_{\mathrm{d}}(\beta)\, h^2 + o(h^2), \qquad  h\rightarrow 0.
\end{align*}
See Corollary \ref{cor:formula_susc} for a justification that this coefficient exists.
Let us emphasize that Sales and Bouchaud \cite{salesbouchaud01} work with a slightly different convention, which results in their temperature susceptibility to equal $\kappa(\beta)\beta^2$ with our notation.

\medskip

\paragraph{Notation.} 
Throughout the paper, $C$ and $c$ denote a positive constant that does not depend on the parameters and can change from line to line. 
For $f \colon (1,\infty) \to \R$ and $g \colon (1,\infty) \to \R_+^*$, we say, as $\beta \downarrow 1$, 
that $f(\beta) = o(g(\beta))$ 
if $\lim_{\beta \downarrow 1} f(\beta)/g(\beta) = 0$, 
that $f(\beta) = O(g(\beta))$ 
if $\limsup_{\beta \downarrow 1} \abs{f(\beta)}/g(\beta) < \infty$, 
that $f(\beta) \sim g(\beta)$ 
if $\lim_{\beta \downarrow 1} f(\beta)/g(\beta) =1$
and that $f(\beta) \asymp g(\beta)$ 
if $0 < \liminf_{\beta \downarrow 1} f(\beta)/g(\beta) \leq \limsup_{\beta \downarrow 1} f(\beta)/g(\beta) < \infty$.

Let us also recall that, for readability, we use the same notation $(\sum_{k\geq 0} \delta_{d_{ik}})_{i \ge 1}$ for the decoration processes of the BBM and of the general decorated case; but, for clarity, it will be precised (between brackets), for every statement/result, when one treats the BBM case or the general decorated case.

\bigskip

\subsection{Main results}
\label{subsection:result_overlap}

\medskip

\paragraph{Near-critical two-temperature overlap.}
Our first results concern the behavior of the mean overlap $\E \left[Q(\beta,\beta')\right]$ or $\E \left[Q_{\mathrm{d}}(\beta,\beta')\right]$ at two supercritical inverse temperatures $\beta$ and $\beta'$.
Recall from Equation \eqref{eq:two-temp-overlap-inequality} that it has already been proved that $\E \left[Q_{\mathrm{d}}(\beta,\beta')\right] < \E \left[Q(\beta,\beta')\right]$.
Our goal here is to quantify this difference in the regime where $\beta'$ is fixed but $\beta$ approaches the critical inverse temperature $\beta_c=1$.

We start with the REM case.
It is not hard to see that the function $\beta \mapsto \E \left[Q(\beta,\beta')\right]$ has an infinite right-derivative at 1, see Remark \ref{rem:Q_right_diff}. We establish a sharper estimate in the following result.

\medskip

\begin{thm}[REM case]
	\label{thm:THM1_REM}
	For any $\beta' > \beta_c=1$, as $\beta \downarrow 1$,
	\[	
	\E \left[Q(\beta,\beta')\right]= (\beta- 1)\log\frac{1}{\beta- 1} + O(\beta-1).
	\]	
\end{thm}

\medskip

We compare this behavior to the one arising for the BBM. 
Surprisingly, the decorations have the effect of smoothing the function $\beta \mapsto \E \left[Q_{\mathrm{d}}(\beta,\beta')\right]$ at criticality: more precisely, $\E \left[Q_{\mathrm{d}}(\beta,\beta')\right]$ is of order $(\beta-1)$ as $\beta \downarrow 1$, as shown in the following result. 
It would be interesting to obtain an asymptotic equivalent instead, which would prove that $\beta \mapsto \E \left[Q_{\mathrm{d}}(\beta,\beta')\right]$ is right-differentiable at $\beta=1$, but this seems out of reach with the techniques used here.

\medskip

\begin{thm}[BBM case]
\label{thm:THM1_BBM}
	For any $\beta' > \beta_c= 1$, 
	\[
	0< \liminf_{\beta \downarrow 1} \, \frac{\E \left[Q_{\mathrm{d}}(\beta,\beta')\right]}{\beta-1}  \leq  \limsup_{\beta \downarrow 1} \, \frac{\E \left[Q_{\mathrm{d}}(\beta,\beta')\right]}{\beta-1} < +\infty.
	\]
\end{thm}

\medskip

In order to shed some light on these different behaviors, we study in Section \ref{sec:general_decorated_case} the case of general decoration processes $\cD = \sum_{k\geq0} \delta_{d_k}$.
We do not state these results precisely here, some of them having very technical assumptions, but rather give an informal description of some of the conclusions:
\begin{itemize}
	\item If the sum $\sum_{k\geq0} \e^{\beta d_k}$ has bounded $(1+\varepsilon)$-moment as $\beta \downarrow 1$, then the behavior of $\E \left[Q_{\mathrm{d}}(\beta,\beta')\right]$ is the same as for the REM in Theorem \ref{thm:THM1_REM}, see Corollary \ref{cor:overlap_small_deco}.
	\item If the $\gamma$-th moment of the sum $\sum_{k\geq0} \e^{\beta d_k}$ explodes like $(\beta-1)^{-\psi(\gamma)}$ as $\beta \downarrow 1$, then different behaviors appear depending on the value of $\psi'(1)-\psi(1)$: the larger this value---in other words, the more moments of $\sum_{k\geq0} \e^{\beta d_k}$ are governed by rare events---the faster $\E \left[Q_{\mathrm{d}}(\beta,\beta')\right]$ decreases as $\beta \downarrow 1$.
	More precisely, when $\psi'(1)-\psi(1)<1$, $\E \left[Q_{\mathrm{d}}(\beta,\beta')\right]$ is still of order $(\beta-1) \log \frac{1}{\beta-1}$ like in the REM case, but possibly with another multiplicative constant.
	At the critical value $\psi'(1)-\psi(1)=1$, the first order constant vanishes and $\E \left[Q_{\mathrm{d}}(\beta,\beta')\right]$ becomes of order $(\beta-1)$.
	Above the critical value, $\E \left[Q_{\mathrm{d}}(\beta,\beta')\right]$ decreases faster than $(\beta-1)^{1+\eta}$ for $\eta > 0$ in some explicit interval,
	see Corollary \ref{cor:overlap} for a precise statement.
\end{itemize}
It turns out that the BBM corresponds exactly to the critical case described in the second point. This fact is a consequence of some fine properties of the decoration of the BBM, proved by building upon results by Cortines, Hartung and Louidor \cite{cortineshartunglouidor2019,cortineshartunglouidor2021}, who showed that the mean of level sets of the decoration process is driven by rare events in which the decoration is much bigger than its typical behavior.
We believe that this should be the case of other models falling in the log-correlated fields universality class, such as general branching random walks or the 2D discrete Gaussian free field.

\medskip

We conclude this discussion by two open questions. 
Firstly, we wonder if the limit in Theorem~\ref{thm:THM1_BBM} actually exists, or in other words:

\medskip

\begin{oquest}[BBM case]
	For any $\beta'>1$, is the function $\beta \in (1,\infty) \mapsto \E[Q_{\mathrm{d}}(\beta,\beta')]$ right-differentiable at $1$?
\end{oquest}

\medskip

Our results compare two 1-RSB systems, but, as suggested by an anonymous referee, it would be interesting to compare systems with more levels of replica symmetry breaking.
More precisely, a $k$-GREM (that is with a tree structure of height $k+1$), where each vertex has $\e^{t/(2k)}$ children and with energy increments $\cN(0,\sigma_\ell^2 t)$ at level $\ell$, is a $k$-RSB system if $\sigma_1 > \dots > \sigma_k$, see \cite{bovierkurkova2004-1} for details.
It can be compared to a $k$-speed BBM, where particles move with variance $\sigma_\ell^2$ on $[(\ell-1)t,\ell t]$ for a given time horizon $t$, see \cite{bovierhartung2014} for a description of its extremal process when $k=2$.
With these definitions, the one-temperature overlap distribution is the same for both models in the limit $t\to\infty$, but the two-temperature overlap distribution should be different. 
It could be compared in an abstract manner in the spirit of \cite{PainZindy21,bonnefont22}, or more quantitatively in the near-critical regime as in this paper.

\medskip

\begin{oquest}
	Compare the two-temperature overlap distributions of the $k$-GREM and of the $k$-speed BBM.
\end{oquest}

\medskip

\paragraph{Temperature susceptibility.}
We now turn to results concerning the temperature susceptibility. 
We first compute it explicitly for the REM in the supercritical phase and study its asymptotic behavior close to the critical temperature, as well as in the low temperature regime.

\medskip

\begin{thm}[REM case]
	\label{thm:THM2_REM}
	For any $\beta > \beta_c = 1$, 
	\begin{align*}
		\kappa(\beta)&=\frac{1}{2}\left(\frac{1}{\beta^2-1}+\frac{6}{\pi^2\beta^3(\beta+1)}\left(\frac{\Gamma''\left(\frac{\beta-1}{\beta}\right)}{\Gamma\left(\frac{\beta-1}{\beta}\right)}-\left(\frac{\Gamma'\left(\frac{\beta-1}{\beta}\right)}{\Gamma\left(\frac{\beta-1}{\beta}\right)}\right)^2\right)-\frac{\beta^2}{(\beta^2-1)^2}\right) \, ,
	\end{align*}
	and
	\[
	\kappa(\beta)\, \underset{\beta \downarrow 1}{\sim}\, \left(\frac{3}{2\pi^2}-\frac{1}{8}\right)\frac{1}{(\beta-1)^2}\, ,\qquad\qquad
	\kappa(\beta)\underset{\beta\rightarrow+\infty}{\sim}\left(\frac{6}{\pi^2}\zeta(3)-\frac{1}{2}\right)\frac{1}{\beta^5}\, .
	\]
\end{thm}

\medskip

\begin{rem}
	This question has already been studied in the physics literature by Sales and Bouchaud \cite{salesbouchaud01} in the case of the REM with exponentially distributed energies. 
	The formula they obtain for $\kappa(\beta)$ (see their appendix) does not seem to match the one above, but we are not able to find the source of this discrepancy. In particular, the exponents appearing as $\beta \downarrow 1$ and $\beta \to \infty$ are different.
\end{rem}

\medskip

Then, we study the temperature susceptibility of the BBM. 
The following results shows that it is strictly larger than the one of the REM and gives the behavior of the temperature susceptibility close to the critical temperature, showing that it diverges at the same speed as for the REM but with a different multiplicative constant. We do not investigate the behavior at low temperature.

\medskip

\begin{thm}[BBM case]
	\label{thm:THM2_DECORATED}
	For any $\beta > \beta_c =1$, one has $\kappa_{\mathrm{d}}(\beta) > \kappa(\beta)$. Moreover, we have
	\[
	\frac{3}{2\pi^2}-\frac{1}{8}
	< \liminf_{\beta \downarrow 1} \, (\beta-1)^2 \, \kappa_{\mathrm{d}}(\beta) 
	\le \limsup_{\beta \downarrow 1} \, (\beta-1)^2 \,\kappa_{\mathrm{d}}(\beta) 
	\le \frac{3}{\pi^2}-\frac{1}{8}.
	\]
\end{thm}

\medskip

As for the two-temperature overlap, we also investigate the general decorated case. Proposition \ref{prop:susceptibility_decorated} shows that $\kappa_{\mathrm{d}}(\beta)$ can be written as the sum of $\kappa(\beta)$ and a nonnegative term, which is positive for most non-deterministic decoration processes. 
Showing that this additional term is growing like $(\beta-1)^{-2}$ as $\beta \downarrow 1$ for the BBM, and in particular the lower bound, is the main task in the proof of Theorem \ref{thm:THM2_DECORATED}.
In general, this term can be negligible w.r.t.\@ $(\beta-1)^{-2}$ for ``small enough'' decorations, and it can be much larger than $(\beta-1)^{-2}$ for ``big enough'' decorations, see Example \ref{ex:deco2}. So again, the BBM seems to belong to a critical window, even if here this critical window seems to include more models than for the two-temperature overlap, see e.g.\@ the first family of decorations considered in Example \ref{ex:deco2}.

\medskip

Finally we state two open questions similar to the ones for the two-temperature overlap. 

\medskip

\begin{oquest}[BBM case]
	Does the limit of $(\beta-1)^2 \, \kappa_{\mathrm{d}}(\beta)$ as $\beta \downarrow 1$ exists?
\end{oquest}

\medskip
	
In the context of susceptibility, the number of levels of replica symmetry breaking does not really matter. However, one can expect that the susceptibility is also explicit for the GREM and it would be interesting to compare it with the one for the REM.
	
\medskip

\begin{oquest}
	Find the susceptibility for the GREM.
\end{oquest}

\bigskip

\subsection{Organization of the paper}

\medskip

Section \ref{section:prelimREM} includes preliminary results on the partition function of the REM, on a change of measures used to study the decorated case and on the decoration process of the BBM. 
Section~\ref{ov2t} is dedicated to the two-temperature overlap, including proofs of Theorems \ref{thm:THM1_REM} and \ref{thm:THM1_BBM}, as well as results in more general decorated cases in Section \ref{sec:general_decorated_case}.
In Section \ref{susc}, we investigate the temperature susceptibility, proving Theorems \ref{thm:THM2_REM} and  \ref{thm:THM2_DECORATED}. 
Appendix \ref{sec:app_BBM} contains the proof of some results needed about the decoration process of the BBM throughout the paper.

\section{Preliminaries}
\label{section:prelimREM}

\medskip

\subsection{Properties of \texorpdfstring{$Z(\beta)$}{Z}}

In this section, we collect some basic facts on the partition function of the REM in the supercritical phase $\beta>1$ that will be useful in the following sections. We work directly with the limiting extremal process in the spirit of Ruelle's reformulation of the REM \cite{ruelle87} (see also \cite{DerridaMottishaw_2021} by Derrida and Mottishaw where this model is called the Poisson REM) : the partition function is given by $Z(\beta)=\sum \e^{\beta \xi_k}$ where $(\xi_k)_{k\geq1}$ are the atoms ranked in non-increasing order of a PPP($\e^{-x}\dx$). The fact that this is the limit of the properly renormalized partition function of the finite size REM is proved by Bovier, Kurkova and Löwe \cite[Theorem~1.5]{BovierKurkovaLowe02}.

It will be convenient to use the notation $\eta_k \de \e^{-\xi_k}$ because these are the atoms of a homogeneous Poisson point process on $\R_+$  and allows us to use the law of large numbers. 
Hence $Z(\beta)$, can be rewritten
\begin{equation} \label{eq:Z_in_terms_of_eta}
Z(\beta) = \sum_{k\geq 1} \eta_k^{-\beta}, \qquad \forall \, \beta >1.
\end{equation}
We first show that $Z(\beta)$ has a stable distribution. We use the following convention: the stable distribution ${\rm Stable}_\alpha\left(c,s,\mu \right)$\footnote{Here we do not use $\beta$ for the skewness parameter because the letter is used for the inverse temperature.} has a characteristic function given by
\[
t\mapsto\exp\left\{it\mu-|ct|^\alpha\left(1-is\,\mathrm{sgn}(t)\Phi\right)\right\},
\quad \text{where} \quad 
\Phi = \left\{
\begin{array}{ll}
\tan\left(\frac{\pi\alpha}{2}\right), & \mbox{if } \alpha \neq 1, \\
-\frac{2}{\pi}\log |t|, & \mbox{if } \alpha=1.
\end{array}
\right.
\]

\medskip

\begin{lem}[REM case]	\label{lem:Z_stable} 
	The Laplace transform of $Z(\beta)$ is given, for any $t \geq 0$, by
	\begin{align*}
		\Ec{\e^{-tZ(\beta)}}
		= \exp\left\{-\Gamma\left(\frac{\beta-1}{\beta}\right) t^{\frac{1}{\beta}}\right\}.
	\end{align*}
	In other words, the distribution of $Z(\beta)$ is ${\rm Stable}_{1/\beta} \left( \big(\Gamma(1-\frac{1}{\beta})\cos\frac{\pi}{2\beta}\big)^\beta,1,0 \right)$.
\end{lem}

\medskip

\begin{proof}
	The formula for Laplace transforms of Poisson point processes applied to Equation \eqref{eq:Z_in_terms_of_eta} gives
	\begin{align*}
	\Ec{\e^{-tZ(\beta)}} 
	&= \exp\left\{-\int_{0}^{\infty} (1-\e^{-tx^{-\beta}}) \diff x\right\}
	= \exp\left\{-\Gamma\left(\frac{\beta-1}{\beta}\right) t^{\frac{1}{\beta}}\right\}\,,
	\end{align*}
	which is the Laplace transform of the desired stable law.
\end{proof}

\medskip

\begin{rem}
	The fact that the partition function has a stable distribution appears as a consequence of \cite[Theorem 2.3]{BenarousBogachevMolchanov05} by Ben Arous, Bogachev and Molchanov, where the authors study the limiting distributions of exponential sums of i.i.d. random variables.
\end{rem}

\medskip

\begin{lem}[REM case]
\label{esz}
	For any $\beta >1$ and any $\alpha>-\frac{1}{\beta}$, 
	\[
	\Ec{ Z(\beta)^{-\alpha} }
	=\frac{\Gamma(\alpha\beta+1)}{\Gamma(\alpha+1)\Gamma(\frac{\beta-1}{\beta})^{\alpha\beta}} \, .
	\]
\end{lem}

\medskip

\begin{proof}
	For $\alpha>0$, this is proved by taking the expectation of 
	\[
	Z(\beta)^{-\alpha} = \frac{1}{\Gamma(\alpha)}\int_{0}^{\infty}t^{\alpha-1}\e^{-tZ(\beta)}\diff t,
	\]
	and then applying Lemma \ref{lem:Z_stable}.
	For $-1/\beta<\alpha<0$, we use instead the following representation
	 \[
	 Z(\beta)^{-\alpha} = \frac{1}{\Gamma(1+\alpha)}\int_{0}^{\infty}t^{\alpha}Z(\beta) \, \e^{-tZ(\beta)}\diff t,
	 \]
	 and note that $\E[Z(\beta) \e^{-tZ(\beta)}]$ is obtained by differentiating $\E[\e^{-tZ(\beta)}]$ with respect to $t$.
\end{proof}

\medskip

The formula established in the previous lemma yields to the following moment estimates in the regime $\beta\downarrow1$, which are used throughout the paper.

\medskip

\begin{cor}[REM case]
 Let $\alpha > -1$. Then, as $\beta\downarrow1$,
	\begin{align}
	\label{a1}
	\Ec{ Z(\beta)^{-\alpha} } & = (\beta-1)^\alpha \left(1+O\left((\beta-1) \log \frac{1}{\beta-1} \right) \right) \, ,\\
	\label{a2} 
	\Ec{ Z(\beta)^{-\alpha/\beta} } & = (\beta-1)^\alpha \left(1+O(\beta-1) \right) \, ,\\
	\label{a5}
	\Ec{ Z(\beta)^{-1/\beta} \log \left(Z(\beta)^{1/\beta}\right)} & = (\beta-1) \log \frac{1}{\beta-1} +O(\beta-1).
	\end{align}
\end{cor}

\medskip

\begin{proof}
	The first two expansions are obtained by expanding the Gamma function in the formula given by Lemma \ref{esz}.
	For the third one, note that, for any $\beta > 1$ and $\alpha >-1/\beta$,  $\frac{\diff}{\diff \alpha} \Ec{ Z(\beta)^{-\alpha} } = - \alpha \Ec{ Z(\beta)^{-\alpha} \log Z(\beta)}$.
\end{proof}

\medskip

The expansions in the previous corollary are governed by the fact that $Z(\beta)$ concentrates around $1/(\beta-1)$ as $\beta \downarrow 1$. More precisely, we have surprisingly the following almost sure expansion.

\medskip

\begin{lem}[REM case]	\label{lem:Z_as}
	When $\beta \downarrow 1$, we have
	\[
	Z(\beta) = \frac{1}{\beta -1} + W + o(1),\quad \text{almost surely,}
	\]
	where $W \sim {\rm Stable}_1\left(\frac{\pi}{2},1,1-\gamma\right)$
	and $\gamma$ is the Euler constant.
\end{lem}

\medskip

\begin{proof}
	With the Riemann $\zeta$ function defined by $\zeta(\beta) = \sum_{k\geq1} \frac{1}{k ^\beta}$, one has
	\[
	Z(\beta)-\zeta(\beta) = \sum_{k\geq1} \frac{k^\beta-\eta_k^\beta}{k^\beta\eta_k^\beta}.
	\]
	If $\beta\leq2$, one has
	\[
	 \frac{\left|k^\beta-\eta_k^\beta\right|}{k^\beta\eta_k^\beta}\leq \frac{\beta\left(\eta_k^{\beta-1}+k ^{\beta-1}\right)\left|\eta_k-k\right|}{k^\beta\eta_k^\beta} \leq 5 \frac{\left|\eta_k-k\right|}{k^2},
	\]
	for $k$ large enough. This last term is summable since $\eta_k= k+O(k^{\frac{3}{4}})$ a.s. 
	Hence, the dominated convergence theorem (for series) implies that $Z(\beta)-\zeta(\beta)$ has an a.s. limit when $\beta$ tends to $1$. 
	The almost sure expansion follows from the fact that $\zeta(\beta) = \frac{1}{\beta-1}+\gamma+o(1)$ when $\beta\downarrow1$. Then, computing the characteristic function enables to identify the limiting law: if $t>0$, one has
	\begin{align*}
		\Ec{\e^{itZ(\beta)}}&=\exp\left\{-\Gamma\left(\frac{\beta-1}{\beta}\right)\cos\frac{\pi}{2\beta}\,t^{\frac{1}{\beta}}\left(1-i\tan\frac{\pi}{2\beta}\right)\right\}\\
		&=\exp\left\{-\Gamma\left(\frac{\beta-1}{\beta}\right)t^{\frac{1}{\beta}}\e^{-i\frac{\pi}{2\beta}}\right\}\\
		&=\exp\left\{it\left(\frac{1}{\beta-1}-\log(t)+1-\gamma+i\frac{\pi}{2}\right)+O(\beta-1)\right\},
	\end{align*}
	as $\beta\downarrow1$, therefore
		\begin{align*}
		\Ec{\e^{it\left(Z(\beta)-\frac{1}{\beta-1}\right)}}&\underset{\beta\downarrow1}{\longrightarrow}
		\exp\left\{it(1-\gamma)-\frac{\pi}{2}t\left(1+i\frac{2}{\pi}\log t\right)\right\},
	\end{align*}
	which is the expected characteristic function. The case $t<0$ is similar.
\end{proof}

\medskip

\begin{rem} \label{rem:1-stable}
	These 1-stable fluctuations are reminiscent of some recent results on the BBM: such fluctuations appear for the critical Gibbs measure \cite{MaiPai2019,MaiPai2021}, as well as for the limiting extremal process, when re-centered around a low position \cite{mytnik2021fisherkpp}. The result above shows that 1-stable fluctuations already appear in the much simpler context of the extremal process of the REM.
	Moreover, it has a consequence for the BBM stated in Remark \ref{rem:1-stable_2}.
\end{rem}

\bigskip

\subsection{Change of measure in the general decorated case}
\label{sec:change_of_measure}

\medskip

In the paper, we sometimes consider the \textit{general decorated case}.
This means that we work with a limiting extremal process which is of the form 
\[
	\sum_{i \geq 1}\sum_{k \geq 0} \delta_{\xi_i+d_{ik}},
\]
where the $(\xi_i)_{i\geq 1}$ are the atoms of a PPP($\e^{-x} \diff x$) independent of $(\sum_{k\geq 0} \delta_{d_{ik}})_{i \ge 1}$, which are i.i.d.\@ copies of a point process $\cD = \sum_{k\geq 0} \delta_{d_k}$ on $(-\infty,0]$ which has a.s.\@ an atom at~0,
but we do not assume that $\cD$ is the decoration process of the BBM.
In particular, for $\beta, \beta' > \beta_c=1$ the limiting partition function $Z_{\mathrm{d}}(\beta)$ is defined as in Equation \eqref{def:Z_d} and the limiting mass at 1 of the two-temperature overlap $Q_{\mathrm{d}}(\beta,\beta')$ is defined as in Equation \eqref{def:Q_d}.

Moreover, let us introduce the partition functions associated to the decoration processes\footnote{Note that we do not add a subscript ``d'' in this notation for the sake of readability in the proofs where the quantity $S_\beta$ is ubiquitous. There is no confusion possible with the REM for which the quantity $S_\beta$ does not appear (it is equal to 1 in that case).}:
\[
	S_\beta \coloneqq \sum_{k\geq 0} \e^{\beta d_k} 
	\qquad \text{and} \qquad 
	S_{\beta,i} \coloneqq \sum_{k\geq 0} \e^{\beta d_{ik}}, \qquad \forall \, i \geq 1.
\]
For $\beta > 1$, assuming $\Ec{S_\beta^{1/\beta}} < \infty$, we introduce a new probability measure $\P_\beta$ such that
\begin{itemize}
	\item the distribution of $\cD$ under $\P_\beta$ is characterized by 
	\begin{equation} \label{eq:def_P_beta}
	\E_\beta\left[F(\cD)\right]
	= \frac{\E\left[S_\beta^{1/\beta} F(\cD)\right]}{\E\left[S_\beta^{1/\beta}\right]},
	\end{equation}
	for any measurable bounded function $F$ from the space of Radon measures on $\R$ to $\R$;
	\item under $\P_\beta$, 
	the $(\xi_i)_{i\geq 1}$ are still the atoms of a Poisson point process on $\R$ with intensity measure $\e^{-x} \diff x$ independent of $(\sum_{k\geq 0} \delta_{d_{ik}})_{i \ge 1}$, which are i.i.d.\@ copies of $\cD$ under $\P_\beta$.
\end{itemize}
With this definition in hand, the following fact holds. 

\medskip

\begin{lem}[General decorated case]
 \label{lem:change_of_measure}
	Let $\beta > 1$. Assume $\Ec{S_\beta^{1/\beta}} < \infty$. 
	Then, with $c_\beta \coloneqq \log \Ec{S_\beta^{1/\beta}}$, 
	\[
	\left(\, \sum_{i\geq 0} \delta_{\xi_i + \frac{1}{\beta} \log S_{\beta,i}}, \bigg( \sum_{k\geq 0} \delta_{d_{ik}} \bigg)_{i \ge 1} \,\right)
	\;\mathrm{ under }\; \P 
	\;\; \overset{\mathrm{(d)}}{=}\;\; 
	\left(\, \sum_{i\geq 0} \delta_{\xi_i + c_\beta}, \bigg( \sum_{k\geq 0} \delta_{d_{ik}} \bigg)_{i \ge 1} \,\right)
	\;\mathrm{ under }\; \P_\beta\,.
	\]
\end{lem}

\medskip

\begin{proof}
	This can be obtained via a direct Laplace transform calculation, or as a consequence of \cite[Lemma 2.1]{panchenkotalagrand2007-1} applied with 
	$u_i = \e^{\beta \xi_i}$, 
	$m = 1/\beta$, 
	$X_i = S_{\beta,i}$
	and $Y_i = \sum_{k\geq 0} \delta_{d_{ik}}$.
\end{proof}

\medskip

As a consequence, we get the following simple distribution for the decorated partition function
\[
	Z_{\mathrm{d}}(\beta) 
	\de \sum_{i \geq 1} \sum_{k \geq 0} \e^{\beta (\xi_i+ d_{ik})} 
	= \sum_{i \geq 1} \e^{\beta \xi_i} S_{\beta,i}.
\]
However, note that one cannot relate the joint distribution of $(Z_{\mathrm{d}}(\beta),Z_{\mathrm{d}}(\beta'))$ with the one of $(Z(\beta),Z(\beta'))$ in such a way (otherwise the two-temperature overlap distribution would not depend on the decoration).

\medskip

\begin{cor}[General decorated case]
\label{cor:change_of_measure}
	Let $\beta > 1$. Assume $\Ec{S_\beta^{1/\beta}} < \infty$. Then, under $\P$,
	\[
	Z_{\mathrm{d}}(\beta)
	\overset{(\mathrm{d})}{=}
	\Ec{S_\beta^{1/\beta}}^{\beta} Z(\beta).
	\]
\end{cor}

\medskip

\begin{proof}
	By Lemma \ref{lem:change_of_measure}, $Z_{\mathrm{d}}(\beta)$ under $\P$ has the same distribution as $\e^{\beta c_\beta} Z(\beta)$ under $\P_\beta$. But $Z(\beta)$ has the same law under $\P$ and under $\P_\beta$, so the result follows.
\end{proof}

\medskip

\begin{rem} \label{rem:deco_part_fc}
	When $\E[S_\beta^{1/\beta}] = \infty$, one can show that $Z_{\mathrm{d}}(\beta) = \infty$ a.s.
	This follows for example from the previous result applied to $S_\beta \wedge M$ and then letting $M \to \infty$. Therefore, the assumption that $\E[S_\beta^{1/\beta}] < \infty$ for any $\beta > 1$ is a very minimalist one when working with the decorated partition function.
\end{rem}

\medskip

\subsection{The decoration of the BBM}
\label{subsec:decoration_BBM}

\medskip

\paragraph{A description of the decoration.}

In this section, we recall some results on the law of the decoration point process
\[
	\cD = \sum_{k \geq 0} \delta_{d_k} \, ,
\]
appearing in the limit of the extremal process of the BBM.
As mentioned in the introduction (see Equation \eqref{eq:extremalprocessBBM}), convergence of this extremal process has been established in \cite{abbs2013} and \cite{abk2013}, which give two different descriptions of the law of $\cD$.
However, we use here a third description obtained in \cite[Lemma 5.1]{cortineshartunglouidor2019}, as well as several other results shown in this paper and its sequel \cite{cortineshartunglouidor2021}.
Note that the authors work with a BBM with branching rate 1 (instead of $1/2$ for us here), but that both processes have the same distribution up to a time-space scaling.
In particular, one can check that the decoration process $\cC$ appearing in this case (branching rate 1) can be related to $\cD$ as follows:
\[
	\cC = \sum_{k \geq 0} \delta_{d_k/\sqrt{2}} \, .
\]
For the sake of clarity, we give below the description of the law of $\cC$ (instead of $\cD$) so that we can work exactly in the same setting as \cite{cortineshartunglouidor2019} when we adapt some of their proofs.

We consider a BBM with a spine defined under a new probability measure $\widetilde{\P}$ as follows. It starts with one particle at 0 at time 0 which is part of the spine. Particles along the spine branch at rate 2 and move according to a standard Brownian motion. When they branch into two particles, one of them, chosen uniformly at random, is part of the spine and the other one starts a standard BBM with branching rate 1.
We denote by $L_t$ the set of particles alive at time~$t$, by $X_t$ the particle at time~$t$ which is part of the spine, and by $h_t(x)$ for $x \in L_t$ the position of particle~$x$ at time~$t$.

Moreover, we set
\[
m_t \coloneqq \sqrt{2} t - \frac{3}{2\sqrt{2}} \log t \, ,
\]
which is the position of the maximum of the BBM (with branching rate 1) at time $t$ up to $O(1)$ fluctuations,
and, for $0 < r \leq t$ and $x \in L_t$,
\[
\cC_{t,r}(x) \coloneqq \sum_{y \in L_t \, : \, d(x,y) < r} \delta_{h_t(y) - h_t(x)} \, ,
\]
where $d(x,y) := \inf \{s \geq 0 :x \text{ and } y \text{ share a common ancestor in } L_{t-s} \}$.
For brevity, we introduce the new probability measure
\begin{equation} \label{eq:def_widetilde_P_t}
\hPp{ \cdot }
\coloneqq \widetilde{\P} \left( \, \cdot \, \mathrel{}\middle|\mathrel{} 
h_t(X_t) = \max_{x\in L_t} h_t(x) = m_t \right),
\end{equation}
as well as 
\[
\cC_{t,r_t}^* \coloneqq \cC_{t,r_t}(X_t).
\]
Then, for any positive function $t \mapsto r_t$ such that both $r_t$ and $t-r_t$ tend to $\infty$ as $t \to \infty$,  \cite[Lemma 5.1]{cortineshartunglouidor2019} with $u=0$ establishes that, for the vague convergence on the set of Radon measures on $\R$, 
\begin{equation} \label{eq:description_deco}
\cC_{t,r_t}^* \text{ under } \widetilde{\P}_t
\quad \xrightarrow[t\to\infty]{\text{(d)}} \quad 
\cC \text{ under } \P.
\end{equation}
In their papers \cite{cortineshartunglouidor2019,cortineshartunglouidor2021}, Cortines, Hartung and Louidor develop tools to study $\cC_{t,r_t}^*$ under $\widetilde{\P}_t$ and therefore obtain results on the distribution of $\cC$.

\paragraph{Level sets of the decoration.}

Some of the main results in \cite{cortineshartunglouidor2019} concern the level sets of the decoration point process itself. Recall that $\cC$ is supported on $(-\infty,0]$ and has a.s.\@ an atom at~0.
In \cite[Proposition 1.5]{cortineshartunglouidor2019}, they prove the existence of constants $C_\star,C > 0$ such that
\begin{equation} \label{eq:equiv_C_x}
\Ec{\cD([-\sqrt{2}x,0])} = \Ec{\cC([-x,0])} \sim C_\star \,\e^{\sqrt{2} x}, \qquad \text{as } x \to \infty,
\end{equation}
and, for any $x \geq 0$,
\begin{equation} \label{eq:2nd_moment_C_x}
\Ec{\cD([-\sqrt{2}x,0])^2} = \Ec{\cC([-x,0])^2} \leq C(x+1) \,\e^{2\sqrt{2} x}.
\end{equation}
Note also that, as a consequence of Equation \eqref{eq:equiv_C_x}, there exists $C>0$ such that, for any $x \geq 0$,
\begin{equation} \label{eq:bound_C_x}
\Ec{\cD([-\sqrt{2}x,0])} = \Ec{\cC([-x,0])} \leq C\,\e^{\sqrt{2} x}.
\end{equation}
As a consequence, Cortines, Hartung and Louidor deduce a law of large numbers for large level sets of the whole limiting extremal process.
This law of large numbers, as well as the 1-stable fluctuations appearing at the next order, have been subsequently obtained via PDE techniques by Mytnik, Roquejoffre and Ryzhik \cite{mytnik2021fisherkpp}.

\medskip

\begin{rem} \label{rem:level_sets}
	The second moment bound, see Equation \eqref{eq:2nd_moment_C_x}, is actually of the right order, as proved by Cortines, Hartung and Louidor in a second paper \cite[Proposition 1.1]{cortineshartunglouidor2021}, which means that the first and second moments of $\cC([-x,0])$ are dominated by an unlikely event, as $x \to \infty$.
	This event already appeared in the proof of Equation \eqref{eq:equiv_C_x} in \cite{cortineshartunglouidor2019}, when working with $\cC_{t,r_t}^*([-x,0])$ under $\widetilde{\P}_t$, before taking the limit $t \to \infty$. It consists in the fact that the spine is going sufficiently high at a time $t-s$ with $s$ of order $x^2$, more precisely, it can be defined, for some small $\eta > 0$ and large $M>0$, as
	\begin{equation} \label{eq:event}
		\left\{ \max_{s \in [\eta x^2,\eta^{-1} x^2]} \left( h_{t-s}(X_{t-s}) -m_t + m_s \right) \in [-M,M] \right\},
	\end{equation}
	or in the same way but with $[-M,\infty)$ instead of $[-M,M]$%
	\footnote{
		To see that both choices are roughly equivalent, note that, if $h_{t-s}(X_{t-s}) \simeq m_t - m_s + y$, then the particles branching from the spine around time $s$ typically (under $\widetilde{\P}$) have descendants at time $t$ close to $m_t+y$.
		Hence, the conditioning in the definition of $\widetilde{\P}_t$ implies that $h_{t-s}(X_{t-s})$ cannot be much larger than $m_t - m_s$.
	}.
	This event has a probability of order $1/x$ and, on this event, $\cC_{t,r_t}^*([-x,0])$ is typically of order $x \e^{\sqrt{2} x}$.
	Moreover, it is the dominating event in the first and second moments of $\cC_{t,r_t}^*([-x,0])$ and this implies, up to multiplicative constants, that
	\[
		\Ec{\cC_{t,r_t}^*([-x,0])} \simeq \e^{\sqrt{2} x}
		\qquad \text{and} \qquad 
		\Ec{\cC_{t,r_t}^*([-x,0])^2} \simeq x \e^{2\sqrt{2} x},
	\]
	in accordance with Equations \eqref{eq:equiv_C_x} and \eqref{eq:2nd_moment_C_x}. See \cite[Section 1.3]{cortineshartunglouidor2021} for a more detailed version of this heuristic picture, which plays a crucial role in the arguments used in this paper for the proofs of Theorems \ref{thm:THM1_BBM} and \ref{thm:THM2_DECORATED}.
	We also mention here that the question of the typical size of $\cC([-x,0])$ has been investigated in the physics literature \cite{MuellerMunier2020,LeMuellerMunier2022}, where it is conjectured that it differs from $\e^{\sqrt{2} x}$ by a stretched-exponentially small factor $\e^{-\xi x^{2/3}}$ where $\xi$ is a positive random variable. 
	This result has been proved after the first version of this paper in \cite{HarLouWu2024}.
\end{rem}

\medskip

From these estimates on level sets of the decoration, we can deduce a first moment estimate and a second moment bound for $S_\beta = \sum_{k\geq 0} \e^{\beta d_k}$ stated in the following proposition.

\medskip

\begin{prop}[BBM case]
\label{prop:moments_points_extremaux}
	Let $C_\star > 0$ be the constant appearing in Equation \eqref{eq:equiv_C_x}.
	As $\beta \downarrow 1$, we have
	\begin{equation} \label{eq:moment1_Sbeta}
	\Ec{S_\beta} \sim \frac{C_\star}{\beta-1},
	\end{equation}
	and there exists a constant $C > 0$ such that, for any $\beta \in (1,2]$,
	\begin{equation} \label{eq:moment2_Sbeta}
	\Ec{S_\beta^2} \leq \frac{C}{(\beta-1)^3}.
	\end{equation}
\end{prop}

\medskip

\begin{proof}
	We write $S_\beta = \int_0^{\infty} \cD([-x,0]) \beta\e^{-\beta x}\diff x$.
	Therefore, by Fubini's theorem and Equation \eqref{eq:equiv_C_x},
	\begin{align*}
	\Ec{S_\beta}
	= \int_{0}^{\infty} \left(C_\star \e^{x}+o(\e^x)\right) \beta\e^{-\beta x}\dx 
	\underset{\beta\downarrow 1}{\sim} \frac{C_\star}{\beta -1}.
	\end{align*}
	Similarly, applying Fubini's theorem and then Cauchy-Schwarz inequality together with Equation \eqref{eq:2nd_moment_C_x} yields
	\begin{align*}
	\Ec{S_\beta^2}
	& = \int_{0}^{\infty} \int_{0}^{\infty} \es\left[\cD([-x,0]) \cD([-y,0])\right]\, \beta^2\e^{-\beta (x+y)}\dx\mathrm dy\\
	&\leq C\int_{0}^{\infty} \int_{0}^{\infty} \sqrt{(x+1)(y+1)}\, \e^{x+y}\, \beta^2\e^{-\beta (x+y)}\dx\mathrm dy\\
	&= C \left(\int_{0}^{\infty} \sqrt{x+1}\, \e^{x}\, \beta\e^{-\beta x}\dx\right)^2
	\leq \frac{C}{(\beta-1)^3},
	\end{align*}
	which concludes the proof.
\end{proof}

\medskip

\begin{rem} \label{rem:heuristic_picture_BBM}
	As for the level sets of the decoration $\cC$ (see Remark \ref{rem:level_sets}), the moments of $S_\beta$ are governed by a rare event when $\beta \downarrow 1$.
	Heuristically, everything behaves as if $S_\beta$ was of order $1$ except on an event of small probability $(\beta-1)$ on which it is of order $1/(\beta-1)^2$.
	This event can be described when working on $\cC^*_{t,r_t}$ (before taking the limit $t \to \infty$ in the description of the decoration \eqref{eq:description_deco}) and is given by the event in Equation \eqref{eq:event} with $x = 1/(\beta-1)$.
	This heuristic picture is consistent with Proposition \ref{prop:moments_points_extremaux} and also with the further results in Lemma \ref{lem:small_moments} and
	Corollary \ref{cor:encadrement_moment_S_beta}.
\end{rem}

\medskip

\begin{cor}[BBM case]
\label{cor:moment_1/beta}
	Let $C_\star > 0$ be the constant appearing in Equation \eqref{eq:equiv_C_x}.
	As $\beta \downarrow 1$, we have
	\[
	\Ec{S_\beta^{1/\beta}}
	\sim \Ec{S_\beta^{2-1/\beta}}
	\sim \frac{C_\star}{\beta-1}.
	\]
\end{cor}

\medskip

\begin{proof}
	This is a consequence of the following inequalities: for any real random variables $X \geq 1$ and $Y \geq 0$ and any $\ep \in (0,1)$,
	\begin{equation} 
	\Ec{Y}^{1+\ep} \Ec{XY}^{-\ep}
	\leq \Ec{X^{-\ep} Y} 
	\leq \Ec{Y} 
	\qquad \text{and} \qquad 
	\Ec{Y}
	\leq \Ec{X^\ep Y} 
	\leq \Ec{Y}^{1-\ep} \Ec{XY}^\ep, \label{eq:ineq_X_Y}
	\end{equation}
	which follows on one side simply from the fact that $X \geq 1$ and on the other side from H\"older's inequality.
	These inequalities with $X = Y = S_\beta$ and $\ep = 1 - 1/\beta$ together with Proposition \ref{prop:moments_points_extremaux} yield the desired result.
\end{proof}

\medskip

\begin{rem} \label{rem:1-stable_2}
	Recall from Corollary \ref{cor:change_of_measure} that $Z_{\mathrm{d}}(\beta)$ has the same distribution as $\E[S_\beta^{1/\beta}]^{\beta} Z(\beta)$.
	On the other hand, we have seen in Lemma \ref{lem:Z_as} that $Z(\beta) = \frac{1}{\beta -1} + W + o(1)$ almost surely, as $\beta \downarrow 1$, where $W \sim {\rm Stable}_1 \left(\frac{\pi}{2},1,1-\gamma\right)$.
	Combining this with Corollary \ref{cor:moment_1/beta} above, we get the following fluctuation result for $Z_{\mathrm{d}}(\beta)$:
	\[
	(\beta-1) \left( Z_{\mathrm{d}}(\beta) - \frac{\Ec{S_\beta^{1/\beta}}^{\beta}}{\beta-1} \right)
	\xrightarrow[\beta \downarrow 1]{(\mathrm{d})} C_\star  \, W.
	\]
	Note that the deterministic centering term used here for $Z_{\mathrm{d}}(\beta)$ is not explicit: it is asymptotically equivalent to $C_\star/(\beta-1)^2$ but could involve other non-negligible terms.
	As mentioned in Remark~\ref{rem:1-stable}, these 1-stable fluctuations appear for other quantities describing the front of the BBM.
	The content of this remark has been suggested by Xinxin Chen.
\end{rem}

\medskip

To conclude this section we state two new results on level sets of the decoration. The first one is a bound on cross-moments of level sets, which can be of independent interest. Its proof is postponed to Section \ref{sec:cross_moment}.

\medskip

\begin{prop}[BBM case]
 \label{prop:crossed}
	There exists $C > 0$ such that, for any $v \geq v' \geq 0$,
	\[
	\Ec{\cC([-v,0])\cC([-v',0])}
	\leq C (v'+1) \e^{\sqrt{2}(v+v')}.
	\]
\end{prop}

\medskip

The last result of this section is a uniform bound on the first and second moment of the level sets at finite $t$, which follows from the proofs of Equations \eqref{eq:equiv_C_x} and \eqref{eq:2nd_moment_C_x} in \cite{cortineshartunglouidor2019}, as explained in Section~\ref{sec:unif_bound}.

\medskip

\begin{lem}[BBM case]
\label{lem:moment_C_t_r_t}
	There exists $C > 0$ such that for any $v \geq 0$ and $t \geq 1$,
	\begin{align*} 
	\hEc{\cC_{t,r_t}^*([-v,0])} & \leq C \,\e^{\sqrt{2}v}, \\
	\hEc{\cC_{t,r_t}^*([-v,0])^2} & \leq C(v+1) \,\e^{2\sqrt{2}v}.
	\end{align*}
\end{lem}

\bigskip

\section{Overlap distribution at two temperatures}\label{ov2t}

\medskip
\subsection{Results in the general decorated case}
\label{sec:general_decorated_case}

\medskip

We study here the influence of the decorations on the behavior of $\E[Q_{\mathrm{d}}(\beta,\beta')]$, when $\beta'>\beta_c=1$ is fixed and $\beta \downarrow 1$ and show that they can change the leading order drastically. 
We work in the general decorated case (see Section \ref{sec:change_of_measure}) in order to highlight the fact that the BBM has a critical behavior in some sense to be made precise below.

We assume that 
\begin{equation} \label{eq:hyp_de_base}
	\Ec{S_\beta^{1/\beta}} < \infty, \qquad \forall \, \beta > 1.
\end{equation}
Recall from Corollary \ref{cor:change_of_measure} and Remark \ref{rem:deco_part_fc} that it is a necessary and sufficient condition to have $Z_{\mathrm{d}}(\beta) < \infty$ almost surely.
In particular, this ensures that $Q_{\mathrm{d}}(\beta,\beta')$, introduced in Equation \eqref{def:Q_d} and which can be written as
\begin{equation} \label{def:Q_d_2}
	Q_{\mathrm{d}}(\beta,\beta')
	= \frac{\sum_{i} \e^{(\beta+\beta') \xi_i} S_{\beta,i} S_{\beta',i}}
	{\left(\sum_{i} \e^{\beta \xi_i} S_{\beta,i} \right) 
		\left(\sum_{i} \e^{\beta' \xi_i} S_{\beta',i} \right)},
\end{equation}
is well-defined, and henceforth $\E \left[Q_{\mathrm{d}}(\beta,\beta')\right]$ as well.

In the following results, assumptions are stated in terms of the following random variables
\[
	T_\beta \coloneqq \frac{S_{\beta}^{1/\beta}}{\Ec{S_\beta^{1/\beta}}}, 
	\qquad \forall \, \beta > 1.
\]
Note that $\E[T_\beta] = 1$. 
The fact that the assumptions of the following results are stated in terms of $T_\beta$ shows that the behavior of $\E \left[Q_{\mathrm{d}}(\beta,\beta')\right]$ as $\beta \downarrow 1$ is not governed by the growth rate of $\E[S_\beta^{1/\beta}]$ as $\beta \downarrow 1$, but rather by how much $S_\beta^{1/\beta}$ is fluctuating around its mean.
In the forthcoming assumptions, these fluctuations are controlled in terms of how fast $\E[T_\beta^{1+\varepsilon}]$ explodes or $\E[T_\beta^{1-\varepsilon}]$ vanishes.

The first result concerns the case of weakly fluctuating decorations in which $\E[T_\beta^{1+\varepsilon}]$ does not explode too fast.

\medskip

\begin{thm}[General decorated case]
\label{thm:overlap_deco_1}
	Assume \eqref{eq:hyp_de_base} and that the decoration process~$\cD$ satisfies the following properties:
	\begin{enumerate}
		\item \label{it:ass_alpha} There exists $\alpha \in [0,1]$ such that
		$\Ec{T_\beta \log T_\beta } = (\alpha+o(1)) \log \frac{1}{\beta-1}$, as $\beta \downarrow 1$.
		\item \label{it:ass_1+eps} There exists $\varepsilon > 0$ such that 
		$\E[T_\beta^{1+\varepsilon}] = O((\beta-1)^{-\varepsilon})$, as $\beta \downarrow 1$.
		\item \label{it:ass_S_beta'} For any $\beta' > 1$, $\E[T_\beta \log T_{\beta'}] = O(1)$, as $\beta \downarrow 1$.
	\end{enumerate}
	Then, for any $\beta' > 1$, as $\beta \downarrow 1$,
	\[
	\E \left[Q_{\mathrm{d}}(\beta,\beta')\right] 
	= (1-\alpha+o(1)) (\beta-1) \log\frac{1}{\beta-1}.
	\]
	If moreover, we assume
	\begin{enumerate}[label=(\roman*')]
		\item \label{it:ass_alpha'} There exists $\alpha \in [0,1]$ such that
		$\Ec{T_\beta \log T_\beta } = \alpha \log \frac{1}{\beta-1} + O(1)$, as $\beta \downarrow 1$,
	\end{enumerate}
	then, for any $\beta' > 1$, as $\beta \downarrow 1$,
	\[
	\E \left[Q_{\mathrm{d}}(\beta,\beta')\right] 
	= (1-\alpha) (\beta-1) \log\frac{1}{\beta-1}+ O(\beta-1).
	\]
\end{thm}

\medskip

\begin{cor}[General decorated case]
\label{cor:overlap_small_deco}
	Assume there exists $\varepsilon > 0$ such that $\E[S_\beta^{1+\varepsilon}] = O(1)$, as  $\beta \downarrow 1$. Then, for any $\beta' > 1$, as  $\beta \downarrow 1$,
	\[
	\E \left[Q_{\mathrm{d}}(\beta,\beta')\right] 
	= (\beta-1) \log\frac{1}{\beta-1}+ O(\beta-1).
	\]
\end{cor}

\begin{proof}
	Noting that $T_\beta \leq S_\beta^{1/\beta} \leq S_\beta$, $\log T_\beta \leq C \, T_\beta^\varepsilon$ and $S_{\beta'} \leq S_\beta$ as soon as $\beta \leq \beta'$,
	Assumptions \ref{it:ass_alpha'}-\ref{it:ass_1+eps}-\ref{it:ass_S_beta'} of Theorem \ref{thm:overlap_deco_1} are satisfied with $\alpha=0$ and the result follows.
\end{proof}

\medskip

In the case $\alpha=1$, which we call the critical case, the previous theorem yields 
an upper bound for $\E \left[Q_{\mathrm{d}}(\beta,\beta')\right]$, but does not identify the main order. The following result proves 
a lower bound.

\medskip

\begin{thm}[General decorated case]
\label{thm:overlap_deco_crit}
	Assume \eqref{eq:hyp_de_base} and that there exists $\varepsilon > 0$ such that $\E[T_\beta^{1+\varepsilon}] = O((\beta-1)^{-\varepsilon})$, as $\beta \downarrow 1$.
	Then, for any $\beta' > 1$,
	\[
	\liminf_{\beta \downarrow 1} \frac{\E \left[Q_{\mathrm{d}}(\beta,\beta')\right]}{\beta-1} > 0.
	\]
\end{thm}

\medskip

Finally, we consider strongly fluctuating cases where $\E \left[Q_{\mathrm{d}}(\beta,\beta')\right]$ can vanish faster than $(\beta-1)$, as  $\beta \downarrow 1$. In that case, the function $\beta \mapsto \E \left[Q_{\mathrm{d}}(\beta,\beta')\right]$ has a zero derivative at $\beta =1$.

\medskip

\begin{thm}[General decorated case]
\label{thm:overlap_deco_2}
	Assume \eqref{eq:hyp_de_base} and that there exists $\theta \in (0,1)$ and $\eta \geq 0$ such that,
	\[
		\Ec{ T_{\beta}^{1-\theta} T_{\beta'}^{\theta}}
		= O \left( (\beta-1)^{\theta+\eta}\right), \qquad \beta \downarrow 1.
	\]
	Then, for any $\beta' > 1$, 
	\[
	\E \left[Q_{\mathrm{d}}(\beta,\beta')\right] 
	= O\left((\beta-1)^{1+\eta} \right), \qquad \beta \downarrow 1.
	\]
\end{thm}

\medskip

In the three theorems above, assumptions were stated exactly as we need them in the proof. 
In the following corollary, we work instead under a more readable set of assumptions without seeking any optimality.
Recall that, for $f \colon (1,\infty) \to \R$ and $g \colon (1,\infty) \to \R_+^*$, we write $f(\beta) \asymp g(\beta)$ 
if $0 < \liminf_{\beta \downarrow 1} f(\beta)/g(\beta) \leq \limsup_{\beta \downarrow 1} f(\beta)/g(\beta) < \infty$.

\medskip

\begin{cor}[General decorated case]
 \label{cor:overlap}
	Assume there exist $\gamma_- < 1 < \gamma_+$ and a function $\psi \colon (\gamma_-, \gamma_+) \to [0,\infty)$ such that, for any $\gamma \in (\gamma_-, \gamma_+)$,
	\begin{equation} \label{eq:ass_asymp}
	\Ec{S_\beta^\gamma} \asymp (\beta-1)^{-\psi(\gamma)}, \qquad \beta \downarrow 1.
	\end{equation}
	Assume moreover that $\psi$ is differentiable at 1.
	Then, the following holds, with all asymptotic notation meant to hold as $\beta \downarrow 1$.
	\begin{enumerate}
		\item\label{it:sub_critical} If $\psi'(1) < \psi(1) + 1$ and, for any $\beta' > 1$, $\E[S_\beta \log S_{\beta'}] = O(\E[S_\beta])$, then, setting $\alpha\coloneqq \psi'(1)-\psi(1) \in [0,1)$, we have, for any $\beta' > 1$,
		\[
		\E \left[Q_{\mathrm{d}}(\beta,\beta')\right]
		\sim (1-\alpha) (\beta-1) \log\frac{1}{\beta-1}, \qquad \beta \downarrow 1.
		\]
		\item\label{it:critical} If $\psi'(1) = \psi(1) + 1$, $\psi$ is linear on a neighborhood of 1 and, for any $\beta' > 1$, $\E[S_\beta \log S_{\beta'}] = O(\E[S_\beta])$, then, for any $\beta' > 1$, 
		\[
		\E \left[Q_{\mathrm{d}}(\beta,\beta')\right] \asymp (\beta-1), \qquad \beta \downarrow 1.
		\]
		\item\label{it:super_critical} If $\psi'(1) > \psi(1) + 1$ and, for any $\gamma \in (\gamma_-,1)$ and $\beta' > 1$, $\E[S_\beta^\gamma S_{\beta'}^{1-\gamma}] = O(\E[S_\beta^\gamma])$, then
		\[
		\eta_0 \coloneqq \sup_{\gamma \in (\gamma_-,1)} 
		\left[ \psi(1)\gamma-\psi(\gamma)-1+\gamma \right] > 0,
		\]
		and, for any $\beta' > 1$, 
		\[
		\E \left[Q_{\mathrm{d}}(\beta,\beta')\right] \leq (\beta-1)^{1+\eta_0+o(1)}.
		\]
	\end{enumerate}
\end{cor}

\medskip

\begin{proof}
	Before distinguishing cases, we first make some general remarks. 
	Firstly, for any $\gamma \in (\gamma_-, \gamma_+)$, we have
	\begin{equation} \label{eq:ass_asymp_new}
	\Ec{S_\beta^{\gamma/\beta}} \sim \Ec{S_\beta^{\gamma}} \asymp (\beta-1)^{-\psi(\gamma)}, \qquad \beta \downarrow 1,
	\end{equation}
	as a consequence of the first part of Equation \eqref{eq:ineq_X_Y} with $X=S_\beta^{\gamma_0-\gamma}$, $Y= S_\beta^{\gamma}$ and $\varepsilon = \gamma (\beta-1)/(\beta(\gamma_0-\gamma))$ for some $\gamma_0 \in (\gamma,\gamma_+)$.
	Secondly, recalling the definition of $\P_\beta$ in Equation \eqref{eq:def_P_beta}, we have, for any $h > 0$,
	\[
	\E[T_\beta \log S_\beta^{1/\beta}]
	= \frac{1}{h} \frac{\E[S_\beta^{1/\beta} \log S_\beta^{h/\beta}]}{\E[S_\beta^{1/\beta}]} 
	= \frac{1}{h} \E_\beta[\log S_\beta^{h/\beta}]
	\leq \frac{1}{h} \log\E_\beta[S_\beta^{h/\beta}]
	= \frac{1}{h} \log\frac{\E[S_\beta^{(1+h)/\beta}]}{\E[S_\beta^{1/\beta}]}.
	\]
	Similarly, we have
	\[
	\E[T_\beta \log S_\beta^{1/\beta}]
	= -\frac{1}{h} \E_\beta[\log S_\beta^{-h/\beta}]
	\geq - \frac{1}{h} \log\E_\beta[S_\beta^{-h/\beta}]
	= -\frac{1}{h} \log\frac{\E[S_\beta^{(1-h)/\beta}]}{\E[S_\beta^{1/\beta}]}.
	\]
	By Assumption \eqref{eq:ass_asymp}, we deduce the following inequalities, for any fixed $h > 0$ such that $\gamma_- < 1-h$ and $1+h < \gamma_+$, 
	\begin{equation} \label{eq:encadrement_log}
	\frac{\psi(1)-\psi(1-h)}{h} \log \frac{1}{\beta-1} + O(1)
	\leq \E[T_\beta \log S_\beta^{1/\beta}]
	\leq \frac{\psi(1+h)-\psi(1)}{h} \log \frac{1}{\beta-1} + O(1).
	\end{equation}
	
\medskip

\noindent We now treat the different cases separately.

\medskip
	
	\underline{\textit{Part} \ref{it:sub_critical}}. First note that $\psi'(1) \geq \psi(1)$, as a consequence of the fact that $\E[S_\beta^\gamma] \geq \E[S_\beta]^\gamma$ for any $\gamma>1$. This implies $\alpha \geq 0$.
	We now check that the assumptions of Theorem \ref{thm:overlap_deco_1} are satisfied.
	First note that it follows from Equation \eqref{eq:encadrement_log} by letting $h \to 0$ that
	$\E[T_\beta \log S_\beta^{1/\beta}] = (\psi'(1)+o(1)) \log \frac{1}{\beta-1}$. Together with Equation \eqref{eq:ass_asymp_new}, this shows Assumption \ref{it:ass_alpha}.
	Then, Equation \eqref{eq:ass_asymp_new} implies, for any $\varepsilon \in (0,\gamma_+-1)$,
	\begin{equation} \label{eq:check_T_1+eps}
	\E[T_\beta^{1+\varepsilon}] 
	\asymp (\beta-1)^{-\psi(1+\varepsilon)+(1+\varepsilon)\psi(1)}, \qquad \beta \downarrow 1,
	\end{equation}
	so Assumption \ref{it:ass_1+eps} follows from $\psi'(1) < \psi(1) + 1$.
	Finally, Assumption \ref{it:ass_S_beta'} follows from $\E[S_\beta \log S_{\beta'}] = O(\E[S_\beta])$, together with Equation \eqref{eq:ass_asymp_new} again.
	Therefore, the result follows from Theorem \ref{thm:overlap_deco_1}.

\medskip

	\underline{\textit{Part} \ref{it:critical}}. We check again that the assumptions of Theorem \ref{thm:overlap_deco_1} are satisfied.
	Since $\psi$ is linear in a neighborhood of 1, we have $\psi(1\pm h) = \psi(1) \pm h \psi'(1)$ for $h$ small enough and Equation \eqref{eq:encadrement_log} implies 
	$\E[T_\beta \log S_\beta^{1/\beta}] = \psi'(1) \log \frac{1}{\beta-1} + O(1)$.
	Combining this with Equation \eqref{eq:ass_asymp_new}, we get 
	Assumption~\ref{it:ass_alpha'} with $\alpha = \psi'(1)-\psi(1) = 0$.
	Assumption \ref{it:ass_1+eps} follows from Equation \eqref{eq:check_T_1+eps} and the fact that $\psi(1+\varepsilon)-(1+\varepsilon)\psi(1)= \varepsilon(\psi'(1)-\psi(1)) = \varepsilon$, and Assumption \ref{it:ass_S_beta'} is obtained as before. Hence, Theorem \ref{thm:overlap_deco_1} implies $\E \left[Q_{\mathrm{d}}(\beta,\beta')\right] = O(\beta-1)$. But under these assumptions Theorem \ref{thm:overlap_deco_crit} can also be applied and yield the desired lower bound.

\medskip

	\underline{\textit{Part} \ref{it:super_critical}}. The fact that $\eta_0 > 0$ is a consequence of $\psi'(1) > \psi(1)+1$. 
	Now, for some fixed $\eta \in (0,\eta_0)$, there exists $\gamma \in (\gamma_-,1)$ such that $\psi(1)\gamma-\psi(\gamma)-1+\gamma \geq \eta$.
	Then, with $\theta = 1-\gamma$, we have
	\[
	\Ec{ T_{\beta}^{1-\theta} T_{\beta'}^{\theta}}
	= \frac{\Ec{S_{\beta}^{\gamma/\beta} S_{\beta'}^{(1-\gamma)/\beta}}}
	{\Ec{S_{\beta}^{1/\beta}}^\gamma \Ec{S_{\beta'}^{1/\beta}}^{1-\gamma}}
	\leq \frac{\Ec{S_{\beta}^{\gamma} S_{\beta'}^{1-\gamma}}}
	{\Ec{S_{\beta}^{1/\beta}}^\gamma}
	= O \left( \frac{\Ec{S_{\beta}^{\gamma}}}{\Ec{S_{\beta}^{1/\beta}}^\gamma} \right), \qquad \beta \downarrow 1,
	\]
	using $S_{\beta'} \geq 1$ in the inequality and then the assumption of Part \ref{it:super_critical}. Applying Equations \eqref{eq:ass_asymp} and \eqref{eq:ass_asymp_new}, we get
	\[
	\Ec{ T_{\beta}^{1-\theta} T_{\beta'}^{\theta}}
	= O \left( (\beta-1)^{-\psi(\gamma)+\gamma\psi(1)}\right)
	= O \left( (\beta-1)^{\theta+\eta}\right), \qquad \beta \downarrow 1,
	\]
	by our choice of $\gamma$. Hence, we can apply Theorem \ref{thm:overlap_deco_2} and get $\E \left[Q_{\mathrm{d}}(\beta,\beta')\right] = O((\beta-1)^{1+\eta})$, which proves the result.
\end{proof}

\bigskip

\begin{ex}\label{ex:deco}
	We introduce a family of decoration processes $(\cD^{a,b})_{a,b>0}$ and study the behavior of $\E \left[Q_{\mathrm{d}}(\beta,\beta')\right]$ for these cases.
	For $a,b > 0$, let $X_a \geq 1$ be a random variable with law defined by
	\[
	\P(X_a \geq x) = x^{-a}, \quad \text{for } x \geq 1,
	\]
	and let $\cD^{a,b}$ be a point process such that, conditionally on $X_a$, $\cD^{a,b}$ is the sum of a Dirac mass at 0 and a Poisson point process with intensity $\abs{x}^{b-1} \e^{-x} \1_{x \in [-X_a,0]} \diff x$. This decoration is a toy model of the decoration of the BBM when $a=1$ and $b=2$, see Remarks \ref{rem:level_sets} and \ref{rem:heuristic_picture_BBM}. We are going to check that, for any $\beta' > 1$, as $\beta \downarrow 1$,
	\begin{equation} \label{eq:example}
	\E \left[Q_{\mathrm{d}}(\beta,\beta')\right] 
	\begin{cases}
	\sim (\beta-1) \log\frac{1}{\beta-1}, & \text{if } b<a, \smallskip \\
	\sim (1-a) (\beta-1) \log\frac{1}{\beta-1}, & \text{if } b>a \text{ and } a < 1, \smallskip \\
	\asymp (\beta-1), & \text{if } b>a = 1, \smallskip \\
	= O((\beta-1)^{1+(a-1)(1-\frac{a}{b}) + o(1)}), & \text{if } b>a>1.
	\end{cases}
	\end{equation}
	Note that one does not treat the case $a=b$, which do not fits in the theorems of this section. In the last case, by taking for example $b=2a$ large enough, one can have $\E \left[Q_{\mathrm{d}}(\beta,\beta')\right]$ vanishing faster than any power of $(\beta-1)$.
	
	Let $(d_k)_{k\geq0}$ be the atoms of $\cD^{a,b}$ ranked in decreasing order (with the convention $d_k = -\infty$ when there are no atoms left). One has
	\begin{align*}
		S_\beta 
		= \sum_{k\geq0} \e^{\beta d_k}
		& = 1 + \underbrace{\int_0^{X_a} x^{b-1} \e^{-(\beta-1)x} \dx}_{\eqqcolon R_\beta} + \underbrace{\sum_{k\geq1} \e^{\beta d_k} -\int_0^{X_a} x^{b-1} \e^{-(\beta-1)x} \dx}_{\eqqcolon V_\beta} .
	\end{align*}
	One can check that $\E[V_\beta^2]$ remains bounded when $\beta\downarrow1$, so we focus on the behavior of $R_\beta$.
	Let us write
	\begin{align*}
		R_\beta = \frac{1}{(\beta-1)^b} \int_{0}^{(\beta-1)X_a} u^{b-1}\e^{-u}\du = \frac{1}{h^b} G(hX_a),
	\end{align*}
	where $h \coloneqq \beta-1$ and $G(x)\de\int_{0}^{x} u^{b-1}\e^{-u}\du$. Then, if $\gamma>0$, one has
	\begin{align*}
		\es [R_\beta^\gamma] 
		= \frac{1}{h^{b\gamma}}\int_{1}^\infty G(hx)^\gamma \frac{a \dx}{x^{a+1}} 
		=\frac{a}{h^{b\gamma-a}}\int_{h}^{\infty} \frac{G(t)^\gamma}{t^{a+1}}\dt.
	\end{align*}
	Using the fact that $G$ is bounded and $G(t)\sim t^b/b$ when $t\rightarrow0$, one gets, up to constants $C=C(a,b,\gamma)$,
	\[
	\es [R_\beta^\gamma] \underset{\beta\downarrow1}{\sim}
	\begin{cases}
		C h^{a-b\gamma}, & \mbox{if } b\gamma>a, \\ 
		C\log\frac{1}{h}, & \mbox{if } b\gamma=a,\\
		C, & \mbox{if } b\gamma<a.
	\end{cases}
	\]
	Let us write $\norme{Y}_\gamma \coloneqq \Ec{\abs{Y}^\gamma}^{1/\gamma}$, for any $\gamma >0$. One has
	\begin{align*}
		\norme{R_\beta}_\gamma-\norme{1+V_\beta}_\gamma
		\leq \norme{S_\beta}_\gamma
		\leq \norme{R_\beta}_\gamma+ \norme{1+V_\beta}_\gamma,
	\end{align*}
	which extends the previous results to $S_\beta$ when $\gamma\leq2$ using that $\E[V_\beta^2] = O(1)$: 
	\[
	\es [S_\beta^\gamma] \underset{\beta\downarrow1}{\asymp} 
	\begin{cases}
		h^{a-b\gamma}, & \mbox{if } b\gamma>a, \\ 
		\log\frac{1}{h}, & \mbox{if } b\gamma=a,\\
		1, & \mbox{if } b\gamma<a.
	\end{cases}
	\]
	Note that, if $b<a$, then there exists $\gamma>1$ such that $\es[S_\beta^\gamma] = O(1)$, so the first part of \eqref{eq:example} follows from Corollary \ref{cor:overlap_small_deco}.
	
	So we now focus on the case $b>a$. Then, we have $\es [S_\beta^\gamma] \asymp (\beta-1)^{-\psi(\gamma)}$ with $\psi(\gamma) = b\gamma-a$ for $\gamma \in (a/b,2]$.
	We want to apply Corollary \ref{cor:overlap} noting that $\psi'(1) - \psi(1) = a$. 
	For this we first prove that, for any $\beta'>1$ and $\gamma \in (0,1]$, as $\beta \downarrow 1$,
	\begin{equation} \label{eq:ex_check}
		\Ec{S_\beta^\gamma S_{\beta'}} = O\left( \Ec{S_\beta^\gamma} \right).
	\end{equation}
	To see this, we first use subadditivity of $x \mapsto x^\gamma$ to get
	\begin{align*}
	\Ec{S_\beta^\gamma S_{\beta'}}
	\leq \Ec{\left(1+R_\beta^\gamma+\abs{V_\beta}^\gamma\right) \left(1+R_{\beta'}+\abs{V_{\beta'}} \right)}
	\leq C \left(
	1+\Ec{R_\beta^\gamma}+ \Ec{R_\beta^\gamma \cdot \abs{V_{\beta'}}}
	\right),
	\end{align*}
	where $C = C(a,b,\beta')$ and we used $\E[V_\beta^2] \leq C$ and $R_{\beta'} \leq C$.
	Then, we bound
	\[
	\Ec{R_\beta^\gamma \cdot \abs{V_{\beta'}}} 
	= \Ec{R_\beta^\gamma\,\Ecsq{|V_{\beta'}|}{X_a}} 
	\leq \Ec{R_\beta^\gamma\,\Ecsq{V_{\beta'}^2}{X_a}^{1/2}} 
	\leq C \Ec{R_\beta^\gamma}.
	\]
	Since $\Ec{R_\beta^\gamma} \asymp \Ec{S_\beta^\gamma}$, this implies Equation \eqref{eq:ex_check}.
	In particular, this proves $\E[S_\beta \log S_{\beta'}] \leq \E[S_\beta S_{\beta'}] = O(\E[S_\beta])$, so we can apply Corollary \ref{cor:overlap}.\ref{it:sub_critical} if $a<1$ and Corollary \ref{cor:overlap}.\ref{it:critical} if $a=1$ to get the second and third parts of Equation \eqref{eq:example}.
	On the other hand, Equation \eqref{eq:ex_check} implies $\E[S_\beta^\gamma S_{\beta'}^{1-\gamma}] = O(\E[S_\beta^\gamma])$ for $\gamma \in (0,1]$, so we can apply Corollary \ref{cor:overlap}.\ref{it:sub_critical} if $a>1$, with 
	\[
	\eta_0 = \sup_{\gamma \in (a/b,1)} 
	\left[ (b-a)\gamma-(b\gamma-a)-1+\gamma \right]
	= (a-1)\left(1-\frac{a}{b}\right),
	\]
	which yields the fourth part of Equation \eqref{eq:example}.
\end{ex}

\medskip

\subsection{Some tools for the proofs}
\label{sec:general_tools}

\medskip

The starting point for the proof of the results stated in the previous section is the following expression for $\E \left[Q_{\mathrm{d}}(\beta,\beta')\right]$.

\medskip

\begin{lem}[General decorated case]
 \label{lem:expression_F_deco}
	Let $\beta,\beta' > 1$. Define
	\[
	R = R(\beta,\beta') \coloneqq 
	\biggl( \frac{Z_{\mathrm{d}}(\beta)}{S_\beta} \biggr)^{1/\beta}
	\left( \frac{S_{\beta'}} {Z_{\mathrm{d}}(\beta')}\right)^{1/\beta'}.
	\]
	Then, we have 
	\[
	\E \left[Q_{\mathrm{d}}(\beta,\beta')\right]
	= \Ec{ \left( \frac{S_{\beta'}}{Z_{\mathrm{d}}(\beta')} \right)^{1/\beta'} \int_0^\infty \frac{\diff x}{(1+(Rx)^\beta)(1+ x^{\beta'})}}.
	\]
\end{lem}

\medskip

\begin{proof}
	Starting from Equation \eqref{def:Q_d_2} and recalling that $\eta_k = \e^{-\xi_k}$, we get
	\begin{align*}
	\E \left[Q_{\mathrm{d}}(\beta,\beta')\right]
	& = \Ec{ \sum_{k\geq 1} \eta_k^{-(\beta+\beta')} S_{\beta,k} S_{\beta',k}
		\frac{1}{\Bigl( \eta_k^{-\beta}S_{\beta,k} + \sum_{j\neq k} \eta_j^{-\beta} S_{\beta,j} \Bigr) 
			\Bigl( \eta_k^{-\beta'} S_{\beta',k} + \sum_{j\neq k}\eta_j^{-\beta'} S_{\beta',j} \Bigr)} }\\
	& = \int_{(\R_+)^3} \Ec{x^{-(\beta+\beta')} ss' \frac{1}{(x^{-\beta}s+Z_{\mathrm{d}}(\beta))(x^{-\beta'}s'+Z_{\mathrm{d}}(\beta'))} }
	\diff x \diff \nu(s,s'),
	\end{align*}
	where $\nu$ denotes the law of $(S_\beta,S_{\beta'})$ and one applied Palm formula (see Proposition \ref{prop:palmformula} stated below) to the Poisson point process $\sum_{i} \delta_{(\eta_i,S_{\beta,i},S_{\beta',i})}$ on $(\R_+)^3$, which has intensity $\diff x \otimes \diff \nu(s,s')$.
	Using Fubini's theorem together with the fact that $(S_\beta,S_{\beta'})$ is independent of $(Z_{\mathrm{d}}(\beta),Z_{\mathrm{d}}(\beta'))$ yields 
	\[
	\E \left[Q_{\mathrm{d}}(\beta,\beta')\right]
	= \Ec{ \int_0^\infty \frac{\diff x}{(1+x^\beta Z_{\mathrm{d}}(\beta)/S_\beta)(1+ x^{\beta'} Z_{\mathrm{d}}(\beta')/S_{\beta'})}}.
	\]
	Then the result follows from an obvious change of variable.
\end{proof}

\medskip

\begin{prop}[Palm formula, see Theorem 4.1 in \cite{lastpenrose2017}]
\label{prop:palmformula}
	Let $\Pi$ be a PPP($\mu$) where $\mu$ is a non-zero $\sigma$-finite positive measure on $\R$. 
	Let $\cM$ denote the set of Radon measures on $\R$ and $f \colon \R \times \cM \to \R$ be a positive mesurable function. 
	Then, we have
	\[
	\Ec{\sum_{X \in \Pi} f(X,\Pi\setminus\{X\}) } 
	= \int_\R  \mathbb E [f(x,\Pi)] \mu(\diff x) \, .
	\]	
\end{prop}

\medskip

We now study the integral appearing in Lemma \ref{lem:expression_F_deco} in a deterministic fashion.

\medskip

\begin{lem} \label{lem:integral}
	Let $\beta' \geq \beta > 1$. For $r \geq 0$, we set
	\[
	I(r) \coloneqq \int_0^\infty \frac{\diff x}{(1+(rx)^{\beta})(1+x^{\beta'})}. 
	\]
	\begin{enumerate}
		\item\label{it:bound_by_a_cst} For any $r \geq 0$, $0 \leq I(r) \leq C(\beta')$, where $C(\beta')$ denotes a constant depending only on $\beta'$.
		\item\label{it:estimate_r_>_1} For any $r \geq 1$,
		\[
		\abs{I(r) - \frac{\log r}{r}}
		\leq \frac{4}{r} + (\beta-1)\frac{\log^2 r}{2r}. 
		\]
		\item\label{it:estimate_r_>_0} For any $\delta \in (0,1)$ there exists $C(\beta',\delta)$, depending only on $\beta'$ and $\delta$ such that, for any $r > 0$,
		\[
		\abs{I(r) - \frac{\log r}{r}}
		\leq C(\beta',\delta) \left( \frac{1}{r} + \frac{(\beta-1)}{r^{1-\delta}} 
		+ \frac{1}{r^{1+\delta}}\right). 
		\]
		\item\label{it:lower_bound} For any $r \geq 0$, $I(r) \geq \frac{1}{4} (1 \wedge \frac{1}{r})$.
		\item\label{it:bound_strongly_fluctuating} For any $\delta \in (0,1)$, there exists $C(\beta',\delta)$ depending only on $\beta'$ and $\delta$ such that, for any $r > 0$, $I(r) \leq C(\beta',\delta) r^{\delta-1}$.
	\end{enumerate}
\end{lem}

\medskip

\begin{proof} \textit{Part} \ref{it:bound_by_a_cst}. The fact that $I(r)$ is nonnegative is trivial and the upper bound follows from the inequality $I(r) \leq \int_0^\infty \frac{\diff x}{1+x^{\beta'}}$.

\medskip
	
	\textit{Part} \ref{it:estimate_r_>_1}.
	We first split $I(r)$ into three pieces, keeping the central part as the main one: 
	\[
	\abs{I(r)- \int_{1/r}^{1} \frac{\diff x}{(1+(rx)^{\beta})(1+x^{\beta'})} }
	\leq \int_{0}^{1/r} \diff x
	+ \int_{1}^\infty \frac{\diff x}{r^\beta x^{\beta+\beta'}} \leq \frac{2}{r},
	\]
	using $1/r^{\beta} \leq 1/r$ (recall that $r$ is assumed to be larger than $1$).
	We now focus on the integral from $1/r$ to $1$.
	Using that for any $u,u'>0$, 
	\[
	\abs{\frac{1}{(1+u)(1+u')} - \frac{1}{u}} 
	= \frac{1+u'+uu'}{u(1+u)(1+u')} 
	\leq \frac{1}{u^2} + \frac{u'}{u},
	\]
	we get
	\[
	\abs{\int_{1/r}^{1} \frac{\diff x}{(1+(rx)^{\beta})(1+x^{\beta'})}
		- \int_{1/r}^{1} \frac{\diff x}{(rx)^{\beta}}} 
	\leq \int_{1/r}^{\infty} \frac{\diff x}{(rx)^{2\beta}}
	+ \int_0^{1} \frac{x^{\beta'} \diff x}{(rx)^{\beta}} 
	\leq \frac{2}{r}.
	\]
	Finally, note that 
	\begin{align*}
	\int_{1/r}^{1} \frac{\diff x}{(rx)^{\beta}}
	= \frac{r^{\beta-1}-1}{r^\beta(\beta-1)} 
	= \frac{1-\e^{-(\beta-1) \log r}}{r(\beta-1)},
	\end{align*}
	which yields Part \ref{it:estimate_r_>_1} using $\abs{1-\e^{-t}-t} \leq t^2/2$.
	
\medskip
	
	\textit{Part} \ref{it:estimate_r_>_0}. For $r \geq 1$, we apply Part \ref{it:estimate_r_>_1} and use that $\log^2 r \leq C(\delta) r^{\delta}$. 
	For $r \in (0,1)$, we write
	\[
	\abs{I(r) - \frac{\log r}{r}}
	\leq I(r) + \frac{\abs{\log r}}{r}
	\leq C(\beta') + \frac{C(\delta)}{r^{1+\delta}},
	\]
	using Part \ref{it:bound_by_a_cst}, and then, using $1 \leq 1/r$ for the first term on the right-hand side proves Part \ref{it:estimate_r_>_0}.
	
\medskip
	
	\textit{Part} \ref{it:lower_bound}. 
	For $x \in [0,1 \wedge \frac{1}{r}]$, we have $(1+(rx)^{\beta})(1+x^{\beta'}) \leq 4$.
	So keeping only this part of the integral yields Part \ref{it:lower_bound}. 
	
\medskip
	
	\textit{Part} \ref{it:bound_strongly_fluctuating}.
	This follows from Part \ref{it:bound_by_a_cst} if $r \in (0,1]$, and from Part \ref{it:estimate_r_>_1} if $r \geq 1$.
\end{proof}

\subsection{Proof of the general theorems}

In this section, we prove Theorems \ref{thm:overlap_deco_1}, \ref{thm:overlap_deco_crit} and \ref{thm:overlap_deco_2}. 

\begin{proof}[Proof of Theorem \ref{thm:overlap_deco_1}]
	Applying Lemma \ref{lem:expression_F_deco} and Lemma \ref{lem:integral}.\ref{it:estimate_r_>_0} we get, for some fixed $\delta \in (0,1)$ to be chosen small enough later,
	\begin{align}
	\E \left[Q_{\mathrm{d}}(\beta,\beta')\right]
	& = \Ec{ \left( \frac{S_{\beta'}}{Z_{\mathrm{d}}(\beta')} \right)^{1/\beta'} \frac{\log R}{R}}
	+ O \left( \Ec{ \left( \frac{S_{\beta'}}{Z_{\mathrm{d}}(\beta')} \right)^{1/\beta'} \left( \frac{1}{R} + \frac{(\beta-1)}{R^{1-\delta}} 
		+ \frac{1}{R^{1+\delta}}\right)} \right) \nonumber \\
	& \eqqcolon E_0 + O(E_1 + E_2 + E_3). \label{eq:F_d_2}
	\end{align}
	We estimate these four terms successively, in increasing order of difficulty.
	
\medskip
	
	\underline{\textit{Term $E_1$}}. Using the definition of $R$ and then Corollary \ref{cor:change_of_measure} together with the independence of $S_{\beta}$ and $Z_{\mathrm{d}}(\beta)$, we get
	\begin{equation} \label{eq:E_1}
	E_1
	= \Ec{ \left( \frac{S_{\beta}}{Z_{\mathrm{d}}(\beta)} \right)^{1/\beta} }
	= \Ec{ Z(\beta)^{-1/\beta} }
	\sim \beta-1, \qquad \beta \downarrow 1,
	\end{equation}
	by Equation \eqref{a2}.

\medskip

	\underline{\textit{Term $E_3$}}. We use the definition of $R$ and the independence of $(S_{\beta},S_{\beta'})$ and $(Z_{\mathrm{d}}(\beta),Z_{\mathrm{d}}(\beta'))$, and then $S_{\beta'} \geq 1$ and Cauchy--Schwarz inequality to get
	\begin{align*}
	E_3
	& = \Ec{ \frac{S_{\beta}^{(1+\delta)/\beta} }{S_{\beta'}^{\delta/\beta'}} }
	\Ec{ \frac{Z_{\mathrm{d}}(\beta')^{\delta/\beta'}}{Z_{\mathrm{d}}(\beta)^{(1+\delta)/\beta}}	}
	\leq \Ec{S_{\beta}^{(1+\delta)/\beta}}
	\Ec{Z_{\mathrm{d}}(\beta')^{2\delta/\beta'}}^{1/2}
	\Ec{\frac{1}{Z_{\mathrm{d}}(\beta)^{2(1+\delta)/\beta}}	}^{1/2}.
	\end{align*}
	Using Corollary \ref{cor:change_of_measure}, Lemma \ref{esz} and $\E[S_{\beta'}^{1/\beta'}] < \infty$, we get $\E[Z_{\mathrm{d}}(\beta')^{2\delta/\beta'}] = O(1)$ as soon as $\delta <1/2$.
	Therefore, by Corollary \ref{cor:change_of_measure} and Equation \eqref{a2},
	\begin{equation} \label{eq:E_3}
	E_3
	= O\left( \Ec{S_{\beta}^{(1+\delta)/\beta}} \cdot 
	\frac{(\beta-1)^{1+\delta}}{\E[S_\beta^{1/\beta}]^{1+\delta}} \right)
	= O(\beta-1), \qquad \beta \downarrow 1,
	\end{equation}
	by Assumption \ref{it:ass_1+eps} as soon as $\delta \leq \varepsilon$. 
	Indeed, if Assumption \ref{it:ass_1+eps} holds for some $\varepsilon > 0$ then, for any $\delta \in (0,\varepsilon]$ it holds with $\delta$ replacing $\varepsilon$ by Hölder's inequality%
	\footnote{More precisely this follows from the inequality $\E[T_\beta^{1+\delta}] \leq \E[T_\beta^{1+\varepsilon}]^{\delta/\varepsilon} \E[T_\beta]^{(\varepsilon-\delta)/\varepsilon}
	=\E[T_\beta^{1+\varepsilon}]^{\delta/\varepsilon}$.}.

\medskip

	\underline{\textit{Term $E_2$}}.
	We proceed as for $E_3$, using Cauchy--Schwarz inequality to bound the second expectation below:
	\begin{align*}
	E_2
	& = (\beta-1) \Ec{ S_{\beta}^{(1-\delta)/\beta} S_{\beta'}^{\delta/\beta'}}
	\Ec{ \frac{Z_{\mathrm{d}}(\beta')^{-\delta/\beta'}}{Z_{\mathrm{d}}(\beta)^{(1-\delta)/\beta}}	}
	= O\left( \Ec{ S_{\beta}^{(1-\delta)/\beta} S_{\beta'}^{\delta/\beta'}} \cdot \frac{(\beta-1)^{2-\delta}}{\E[S_\beta^{1/\beta}]^{1-\delta}} \right).
	\end{align*}
	Using $\E[X^{1-\delta} Y^\delta] \leq \E[X]^{1-\delta} \E[Y]^\delta$ with $X =S_\beta^{1/\beta}$ and $Y = S_{\beta'}^{1/\beta'}$ yields $E_2 = O(\beta-1)$, as $\beta \downarrow 1$.

\medskip

	\underline{\textit{Term $E_0$}}. We split this term into three parts:
	\begin{align*}
	E_0 
	& = \Ec{ \left( \frac{S_{\beta}}{Z_{\mathrm{d}}(\beta)} \right)^{1/\beta}
		\left( \log \frac{Z_{\mathrm{d}}(\beta)^{1/\beta}}{S_{\beta}^{1/\beta}}
		+ \frac{1}{\beta'} \log S_{\beta'}
		- \frac{1}{\beta'} \log Z_{\mathrm{d}}(\beta') \right)}
	\eqqcolon E_{00} + E_{01} - E_{02},
	\end{align*}
	where $E_{00}$ is the main term.
	By Corollary~\ref{cor:change_of_measure}, we get
	\[
	E_{00} 
	= \Ec{\frac{\log Z(\beta)^{1/\beta}}{Z(\beta)^{1/\beta}} }
	- \Ec{ Z(\beta)^{-1/\beta} } \Ec{\frac{S_{\beta}^{1/\beta}}{\E[S_\beta^{1/\beta}]} 
		\log \frac{S_{\beta}^{1/\beta}}{\E[S_\beta^{1/\beta}]} }.
	\]
	Then, using Equations \eqref{a5} for the first term and Equation \eqref{a2} together with Assumption \ref{it:ass_alpha} or \ref{it:ass_alpha'} for the second one, we get
	\[
	E_{00} = \begin{cases}
	(1-\alpha+o(1)) (\beta-1) \log\frac{1}{\beta-1} & \text{under Assumption \ref{it:ass_alpha}}, \\
	(1-\alpha) (\beta-1) \log\frac{1}{\beta-1}+ O(\beta-1) & \text{under Assumption \ref{it:ass_alpha'}}.
	\end{cases}
	\]
	On the other hand, we have $E_{01} = O(\beta-1)$ by Corollary~\ref{cor:change_of_measure} together with Equation \eqref{a2} and Assumption \ref{it:ass_S_beta'} (note that it implies $\E[T_\beta \log S_{\beta'}] = O(1)$).
	Finally, we have
	\begin{align*}
	E_{02} 
	& = \frac{1}{\beta'} \Ec{S_{\beta}^{1/\beta}} 
	\Ec{\frac{\log Z_{\mathrm{d}}(\beta')}{Z_{\mathrm{d}}(\beta)^{1/\beta}}}
	= O \left( \Ec{\log^2 \left(Z(\beta') \Ec{S_{\beta'}^{1/\beta'}}\right)}^{1/2} \Ec{Z(\beta)^{-2/\beta}}^{1/2} \right),
	\end{align*}
	using Cauchy--Schwarz inequality and Corollary~\ref{cor:change_of_measure}.
	The first expectation on the right-hand side is a $O(1)$ so Equation \eqref{a2} implies $E_{02} = O(\beta-1)$.
	This concludes the proof.
\end{proof}

\medskip

\begin{proof}[Proof of Theorem \ref{thm:overlap_deco_crit}]
	Applying Lemma \ref{lem:expression_F_deco} and Lemma \ref{lem:integral}.\ref{it:lower_bound}, we get 
	\begin{equation} \label{eq:F_d_5}
	\E \left[Q_{\mathrm{d}}(\beta,\beta')\right]
	\geq \frac{1}{4} \Ec{ \left( \frac{S_{\beta'}}{Z_{\mathrm{d}}(\beta')} \right)^{1/\beta'} \left( 1 \wedge \frac{1}{R} \right)} 
	= \frac{1}{4} \Ec{ \left( \frac{S_{\beta'}}{Z_{\mathrm{d}}(\beta')} \right)^{1/\beta'} \wedge 
	\left( \frac{S_{\beta}}{Z_{\mathrm{d}}(\beta)} \right)^{1/\beta}}.
	\end{equation}
	Moreover, for any $a,b,\delta >0$ and $M > 1$, we have
	\[
	a \wedge b
	\geq \frac{(Ma)\wedge b}{M} 
	\geq \frac{b}{M} \1_{\{b \leq Ma\}}
	= \frac{1}{M} ( b - b \1_{\{b > Ma\}} )
	\geq \frac{1}{M} \left(  b - \frac{b^{1+\delta}}{(Ma)^\delta} \right).
	\]
	Applying this to \eqref{eq:F_d_5} yields, for any $\delta >0$ and $M > 1$,
	\begin{equation} \label{eq:F_d_6}
	\E \left[Q_{\mathrm{d}}(\beta,\beta')\right]
	\geq \frac{1}{4M} \left( 
	\Ec{\left( \frac{S_{\beta}}{Z_{\mathrm{d}}(\beta)} \right)^{1/\beta}}
	- \frac{1}{M^\delta} \Ec{ \left( \frac{Z_{\mathrm{d}}(\beta')}{S_{\beta'}} \right)^{\delta/\beta'}
		\left( \frac{S_{\beta}}{Z_{\mathrm{d}}(\beta)} \right)^{(1+\delta)/\beta}}
	\right).
	\end{equation}
	The first expectation in Equation \eqref{eq:F_d_6} equals the term $E_1$ appearing in the proof of Theorem \ref{thm:overlap_deco_1}, and hence is asymptotically equivalent to $\beta-1$, see Equation \eqref{eq:E_1}. The second expectation equals the term $E_3$ and so is a $O(\beta-1)$ if $\delta \leq \frac{1}{2} \wedge \varepsilon$ with $\varepsilon$ given by the assumption of the theorem. Choosing $M$ large enough, the first expectation dominates and the result follows.
\end{proof}

\medskip

\begin{proof}[Proof of Theorem \ref{thm:overlap_deco_2}]
	Let $\theta \in (0,1)$ be given by the assumption of the theorem. 
	Applying Lemma \ref{lem:expression_F_deco} and Lemma \ref{lem:integral}.\ref{it:estimate_r_>_0} with $\delta = \theta$, we get 
	\begin{equation} \label{eq:F_d_4}
	\E \left[Q_{\mathrm{d}}(\beta,\beta')\right]
	= O \left( \Ec{ \left( \frac{S_{\beta'}}{Z_{\mathrm{d}}(\beta')} \right)^{1/\beta'} \frac{1}{R^{1-\theta}}} \right)
	= O \left( \Ec{ S_{\beta}^{(1-\theta)/\beta} S_{\beta'}^{\theta/\beta'}}
	\Ec{ \frac{Z_{\mathrm{d}}(\beta')^{-\theta/\beta'}}{Z_{\mathrm{d}}(\beta)^{(1-\theta)/\beta}}	} \right),
	\end{equation}
	using the definition of $R$ and the independence of $(S_{\beta},S_{\beta'})$ and $(Z_{\mathrm{d}}(\beta),Z_{\mathrm{d}}(\beta'))$. By Cauchy--Schwarz inequality and Corollary \ref{cor:change_of_measure}, 
	\begin{align*}
	\Ec{ \frac{Z_{\mathrm{d}}(\beta')^{-\theta/\beta'}}{Z_{\mathrm{d}}(\beta)^{(1-\theta)/\beta}}	}
	\leq \frac{\Ec{Z(\beta')^{-2\theta/\beta'}}^{1/2}}{\Ec{S_{\beta'}^{1/\beta'}}^{\theta}}
	\frac{\Ec{Z(\beta)^{-2(1-\theta)/\beta}}^{1/2}}{\Ec{S_\beta^{1/\beta}}^{1-\theta}}
	= O \left( \frac{(\beta-1)^{1-\theta}}{\Ec{S_{\beta'}^{1/\beta'}}^{\theta} \Ec{S_\beta^{1/\beta}}^{1-\theta}} \right),
	\end{align*}
	by Equation \eqref{a2}. Coming back to Equation \eqref{eq:F_d_4} and applying the assumption of the theorem concludes the proof.
\end{proof}

\subsection{The REM case}

First note that Theorem \ref{thm:THM1_REM} is simply a particular case of Corollary \ref{cor:overlap_small_deco}.

\begin{proof}[Proof of Theorem \ref{thm:THM1_REM}]
	For the REM, $S_\beta = 1$ so the result follows from Corollary \ref{cor:overlap_small_deco}.
\end{proof}

\medskip

In this subsection, we add some comments and further results concerning the REM case.

\medskip

\begin{rem} \label{rem:Q_right_diff}
	It follows from Lemma \ref{lem:Z_as} that a.s., as $\beta \downarrow 1$,
	\[
	Q(\beta,\beta') = \frac{Z(\beta+\beta')}{Z(\beta)Z(\beta')} \sim \frac{Z(1+\beta')}{Z(\beta')} (\beta-1),
	\]
	proving that the function $Q(\,\cdot\,,\beta')$ is almost surely right differentiable at $1$. 
	On the other hand, Theorem \ref{thm:THM1_REM} implies that $\E[Q(\,\cdot\,,\beta')]$ has an infinite right-derivative at $1$. 
	There is however a simple way to show this, without aiming for the first order of $\E \left[Q(\beta,\beta')\right]$ as $\beta \downarrow 1$: indeed, using Fatou's lemma, 
	\[
	\liminf\limits_{\beta \downarrow 1}\frac{\E \left[Q(\beta,\beta')\right]}{\beta -1}
	\geq \Ec{\liminf\limits_{\beta \downarrow 1}\frac{\q}{\beta -1}}
	=\Ec{\frac{Z(1+\beta')}{Z(\beta')}},
	\]
	and it is not hard to see that this last expectation is infinite using Palm formula.
\end{rem}

\medskip

We have seen in the previous remark that, as $\beta \downarrow 1$, $Q(\beta,\beta')$ is of order $\beta-1$, but Theorem~\ref{thm:THM1_REM} shows that $\E[Q(\beta,\beta')]$ is much larger.
Hence, the expectation is dominated by an unlikely event, in which the highest point in the extremal process is exceptionally high, more precisely at a position of order $\log \frac{1}{\beta-1}$, but not higher than $\log \frac{1}{\beta-1}$, as proved in the following proposition.

\medskip

\begin{prop}[REM case]
	Let $\varepsilon > 0$ and
	\[
	E_\beta \coloneqq \left\{ \xi_1 \in 
	\left[\varepsilon \log \frac{1}{\beta-1}, \log \frac{1}{\beta-1} \right]
	\right\}.
	\]
	Then, as $\beta \downarrow 1$,
	\[
	\Ec{Q(\beta,\beta') \1_{E_\beta^c}}
	\leq \varepsilon (\beta-1) \log \frac{1}{\beta-1} + O(\beta-1).
	\]
\end{prop}

\medskip

\begin{proof}
	On the one hand, noting that $Q(\beta,\beta') \leq 1$, we have
	\[
	\Ec{Q(\beta,\beta') \1_{\{\xi_1 > \log \frac{1}{\beta-1}\}}}
	\leq \Pp{\xi_1 > \log \frac{1}{\beta-1}}
	= 1 - \exp \left\{ - \int_{\log \frac{1}{\beta-1}}^\infty \e^{-x} \diff x \right\} = O(\beta-1).
	\]
	On the other hand, recalling that $\eta_k = \e^{-\xi_k}$ for all $k \geq 1$,
	\begin{align*}
	\Ec{Q(\beta,\beta') \1_{\{\xi_1 < \varepsilon \log \frac{1}{\beta-1}\}}}
	& \leq \Ec{ \sum_{k\geq 1} \eta_k^{-(\beta+\beta')} \1_{\{\eta_k > (\beta-1)^\varepsilon\}}
		\frac{1}{\Bigl( \eta_k^{-\beta}+\sum_{j\neq k}\eta_j^{-\beta} \Bigr) 
			\Bigl( \eta_k^{-\beta'}+\sum_{j\neq k}\eta_j^{-\beta'} \Bigr)} } \\
	& = \Ec{\int_{(\beta-1)^\varepsilon}^{\infty} 
		\frac{\diff x}{(1+x^{\beta}Z(\beta))(1+x^{\beta'}Z(\beta'))}},
	\end{align*}
	by Palm's formula as in the proof of Lemma \ref{lem:expression_F_deco}.
	This is at most
	\begin{align*}
	\Ec{\int_{(\beta-1)^\varepsilon}^{1} \frac{\diff x}{x^{\beta}Z(\beta)} 
		+ \int_1^\infty
		\frac{\diff x}{x^{\beta+\beta'}Z(\beta)Z(\beta')}}
	& \leq \Ec{Z(\beta)^{-1}} \varepsilon \log \frac{1}{\beta-1}
	+ \Ec{Z(\beta)^{-1}Z(\beta')^{-1}} \\
	& = \varepsilon (\beta-1) \log \frac{1}{\beta-1} + O(\beta-1),
	\end{align*}
	using Cauchy--Schwarz inequality for the second term and then 
	Lemma \ref{esz} and  Equation \eqref{a1} to estimate the two expectations.
\end{proof}

\subsection{The BBM case}
\label{sec:overlap_BBM}

In this section, we work in the BBM case and prove Theorem~\ref{thm:THM1_BBM} as an application of Theorems~\ref{thm:overlap_deco_1} and~\ref{thm:overlap_deco_crit}. In order to check the assumptions of these theorems, we first prove the following upper bounds for small moments of $S_\beta$. Only the case $\gamma \in (1/2,1)$ is useful for our purposes, but the other bounds follow from the same proof, so we include them as well. 
In particular, by the monotone convergence theorem, this shows that the random variable $S_1 = \sum_{k\geq 0} \e^{d_k}$ is in $L^\gamma$ for $\gamma \in (0,1/2)$ and therefore is finite a.s.\@ for the BBM.

\medskip

\begin{lem}[BBM case] \label{lem:small_moments}
	For any $\gamma \in (0,1)$, there exists $C = C(\gamma) > 0$ such that, for any $\beta \in (1,2]$,
	\[
	\Ec{S_\beta^\gamma} \leq
	\begin{cases}
	C(\beta-1)^{1-2\gamma}, & \text{if } \gamma \in (1/2,1), \\
	C \log \frac{1}{\beta-1}, & \text{if } \gamma = 1/2, \\
	C, & \text{if } \gamma \in (0,1/2).
	\end{cases} 
	\]
\end{lem}

\medskip

\begin{proof}
	We postpone most of the work to Lemma \ref{lem:small_moments_at_t}. With this result in hand, setting $f_\beta(x) = \e^{\beta \sqrt{2} x}$, we have $S_\beta = \cC(f_\beta)$ and therefore it is enough to show
	\begin{equation} \label{eq:goal_gamma}
	\Ec{\cC(f_\beta)^\gamma}
	\leq \limsup_{t \to \infty}	\hEc{\cC_{t,r_t}^*(f_\beta)^\gamma}.
	\end{equation}
	To see this, we consider, for $K > 0$, a continuous function $\chi_K$ such that $\1_{[-K,K]} \leq \chi_K \leq \1_{[-(K+1),K+1]}$.
	On the one hand, it follows from the vague convergence stated in Equation \eqref{eq:description_deco} that $\cC_{t,r_t}^*(\chi_K f_\beta)$ under $\widetilde{\P}_t$ converges in distribution to $\cC(\chi_K f_\beta)$ under $\P$.
	On the other hand, we have $\cC_{t,r_t}^*(\chi_K f_\beta)^\gamma
	\leq \e^{\gamma \beta \sqrt{2} (K+1)} \cC_{t,r_t}^*([-(K+1),0])^\gamma$, which is bounded in $L^2$ by Lemma \ref{lem:moment_C_t_r_t}. So $\cC_{t,r_t}^*(\chi_K f_\beta)^\gamma$ is uniformly integrable in $t$ and we get
	\[
	\Ec{\cC(\chi_K f_\beta)^\gamma}
	= \lim_{t \to \infty} 
	\hEc{\cC_{t,r_t}^*(\chi_K f_\beta)^\gamma}
	\leq \limsup_{t \to \infty}	\hEc{\cC_{t,r_t}^*(f_\beta)^\gamma}.
	\]
	Applying the monotone convergence theorem to let $K \to \infty$ on the left-hand side yields Equation \eqref{eq:goal_gamma} and concludes the proof.
\end{proof}

\medskip

We note in the following corollary that combining the estimates from Proposition \ref{prop:moments_points_extremaux} and Lemma \ref{lem:small_moments}, we get the order of $\E[S_\beta^\gamma]$ for any $\gamma \in (1/2,2]$.

\medskip
 
\begin{cor}[BBM case]
 \label{cor:encadrement_moment_S_beta}
	For any $\gamma \in (1/2,2]$, there exists $0 < c < C$ such that, for any $\beta \in (1,2]$,
	\[
	c (\beta-1)^{1-2\gamma}
	\leq \Ec{S_\beta^\gamma} 
	\leq C (\beta-1)^{1-2\gamma}.
	\]
\end{cor}

\medskip

\begin{proof}
	The result holds for $\gamma=1$ by Proposition \ref{prop:moments_points_extremaux}. The upper bound holds for $\gamma \in (1/2,1)$ by Lemma \ref{lem:small_moments} and for $\gamma = 2$  by Proposition \ref{prop:moments_points_extremaux}.
	For $\gamma \in (1,2)$, the upper bound then follows from Hölder's inequality:
	\[
		\Ec{S_\beta^\gamma} \leq \Ec{S_\beta}^{2-\gamma} \E[S_\beta^2]^{\gamma-1}.
	\]
	It remains to prove the lower bound for $\gamma \in (1/2,2] \setminus \{1\}$.
	If $\gamma \in (1/2,1)$, it follows from this consequence of Hölder's inequality
	\[
		\Ec{S_\beta^\gamma} 
		\geq \frac{\E[S_\beta]^{2-\gamma}}{\E[S_\beta^2]^{1-\gamma}},
	\]
	together with the lower bound on $\Ec{S_\beta}$ and the upper bound on $\E[S_\beta^2]$.
	If $\gamma \in (1,2]$, it follows from this consequence of Hölder's inequality, for some $\gamma_0 \in (1/2,1)$ ($\gamma_0=3/4$ for example),
	\[
	\Ec{S_\beta^\gamma} 
	\geq \frac{\Ec{S_\beta}^{(\gamma-\gamma_0)/(1-\gamma_0)}}
	{\E[S_\beta^{\gamma_0}]^{(\gamma-1)/(1-\gamma_0)}} \, ,
	\]
	together with the lower bound on $\Ec{S_\beta}$ and the upper bound on $\E[S_\beta^{\gamma_0}]$.
\end{proof}

\medskip

\begin{proof}[Proof of Theorem \ref{thm:THM1_BBM}]
	We prove that the decoration of the BBM satisfies the assumptions of Corollary \ref{cor:overlap}.\ref{it:critical}.
	Assumption \eqref{eq:ass_asymp} has been checked in Corollary \ref{cor:encadrement_moment_S_beta}, with $\psi(\gamma) = 2\gamma-1$ for $\gamma \in (1/2,2)$.
	In particular, we have $\psi'(1) = 2 = \psi(1) + 1$ and $\psi$ is linear on a right-neighborhood of 1.
	It remains to check that, for any $\beta' > 1$, as $\beta \downarrow 1$, $\E[S_\beta \log S_{\beta'}] = O(\E[S_\beta])$.
	To prove this, note that 
	\begin{align*}
	\E[S_\beta \log S_{\beta'}] \leq \E[S_\beta S_{\beta'}]
	&= \beta \beta' \int_{0}^{\infty} \int_{0}^{\infty} \Ec{ \cD([-x,0])\cD([-y,0])}
	\e^{-\beta x}\e^{-\beta' y} \diff x \diff y \\
	&\leq C\beta \beta'\int_{0}^{\infty} \int_{0}^{\infty}
	((x\wedge y)+1)\e^{x+y}\e^{-\beta x}\e^{-\beta' y} \diff x \diff y,
	\end{align*}
	where we used Proposition \ref{prop:crossed}. Integrating first w.r.t.\@ $y$ and then w.r.t.\@ $x$, it follows that $\E[S_\beta \log S_{\beta'}] = O(\frac{1}{\beta-1})$, which is enough since $\E[S_\beta] \sim C_\star/(\beta-1)$ by Proposition \ref{prop:moments_points_extremaux}.
\end{proof}

\bigskip

\section{Temperature susceptibility}\label{susc}

\subsection{A first formula for \texorpdfstring{$\kappa_{\mathrm{d}}$}{kappad}}
\label{subsec:existence_kappa}

\medskip

We work here in the general decorated case (see Section \ref{sec:change_of_measure}) to prove a first formula for the susceptibility, which is useful for both the REM and the BBM.
Recall that we consider 
\begin{equation} \label{eq:def_C}
\mathbf{C}_{\mathrm{d}}(\beta,\beta+h)
=\frac{\Cov\left(\log Z_{\mathrm{d}}(\beta),\log  Z_{\mathrm{d}}(\beta+h)\right)}{\sigma(\log Z_{\mathrm{d}}(\beta))\,\sigma(\log Z_{\mathrm{d}}(\beta+h))}
\end{equation}
and define the susceptibility in temperature, if it exists, as the coefficient $\kappa_{\mathrm{d}}(\beta)$ such that
\begin{equation} \label{eq:def_susc}
\mathbf{C}_{\mathrm{d}}(\beta,\beta+h) = 1-\kappa_{\mathrm{d}}(\beta)\,  h^2 + o(h^2), \qquad  h\rightarrow 0.
\end{equation}
The main technical tool to prove existence of $\kappa_{\mathrm{d}}(\beta)$ is the following lemma.

\medskip

\begin{lem}[General decorated case]
\label{lem:derC}
	Assume that $\es \left[S_\beta\right]<\infty$, for every $\beta>1$. Then, for any $\beta'>1$, the functions
	\[
	\beta \in (1,\infty) \mapsto f_1(\beta)\de\es\left[\log Z_{\mathrm{d}}(\beta') \log  Z_{\mathrm{d}}(\beta)\right],
	\]
	\[
	\beta \in (1,\infty) \mapsto f_2(\beta)\de\es\left[ \log^2  Z_{\mathrm{d}}(\beta)\right],
	\]
	\[
	\beta \in (1,\infty) \mapsto f_3(\beta)\de\es\left[\log Z_{\mathrm{d}}(\beta) \right],
	\]
	are of class $C^2$ and their first and second derivatives are given by taking the derivative under $\es$.
\end{lem}

\medskip

\begin{proof}
	Let us write
	\[
	Z_{\mathrm{d}}(\beta) = \e^{\beta \xi_1}\sum_{k\geq 1} \e^{\beta E_k},
	\]
	where $(E_k)_{k\geq 1}$ is the decreasing reordering of $\left\{\xi_i+d_{i,j}-\xi_1\,;\,i,j\right\}$. Note that $E_1=0$ and rewrite  $f_1$ as
	\begin{equation*}
	f_1(\beta)= \es\left[\log Z_{\mathrm{d}}(\beta')\log\left(\e^{\beta \xi_1}\sum \e^{\beta E_k}\right)\right]
	=\beta\,\es\left[\xi_1 \log Z_{\mathrm{d}}(\beta')\right]+\es\left[\log Z_{\mathrm{d}}(\beta')\log\sum \e^{\beta E_k}\right].
	\end{equation*}
	Thus $f_1$ is $C^1$ with
	\begin{equation*}
	f_1'(\beta)
	=\es\left[\xi_1 \log Z_{\mathrm{d}}(\beta')\right]+\es\left[\log Z_{\mathrm{d}}(\beta')\frac{\sum E_k\,\e^{\beta E_k}}{\sum \e^{\beta E_k}}\right],
	\end{equation*}
	since, for any given $\beta_1>1$, one has, for all $\beta>\beta_1$, 
	\begin{equation*}
	\abs{\log Z_{\mathrm{d}}(\beta')\,\frac{\sum E_k\,\e^{\beta E_k}}{\sum \e^{\beta E_k}} }
	\leq \abs{\log Z_{\mathrm{d}}(\beta')} \frac{\sum \, \abs{E_k}\,\e^{\beta_1 E_k}}{\sum \e^{\beta_1 E_k}}\, ,
	\end{equation*}
	and those two last terms have finite second moments. The first one because $\log Z_{\mathrm{d}}(\beta')$ has the same law as $\log Z(\beta)$ up to translation and the second one since
	\begin{equation*}
		\left(\frac{\sum |E_k|\,\e^{\beta_1 E_k}}{\sum \e^{\beta_1 E_k}}\right) ^2 \leq \frac{\sum E_k^2\,\e^{\beta_1 E_k}}{\sum \e^{\beta_1 E_k}}\leq \sum E_k^2\,\e^{\beta_1 E_k} \leq C\sum \e^{\beta_2 E_k},
	\end{equation*}
	for some $\beta_2\in(1,\beta_1)$ and $C>0$.
	For the second derivative, one has
	\begin{equation*}
	f_1''(\beta)
	=\es\left[\log Z_{\mathrm{d}}(\beta')\left(\frac{\sum E_k^2\,\e^{\beta E_k}}{\sum \e^{\beta E_k}}- \left(\frac{\sum E_k\,\e^{\beta E_k}}{\sum \e^{\beta E_k}}\right)^2\right)\right],
	\end{equation*}
	since, for any given $\beta_1>1$, one has, for all $\beta>\beta_1$, 
	\begin{equation*}
	\abs{\log Z_{\mathrm{d}}(\beta')\left(\frac{\sum E_k^2\,\e^{\beta E_k}}{\sum \e^{\beta E_k}}- \left(\frac{\sum E_k\,\e^{\beta E_k}}{\sum \e^{\beta E_k}}\right)^2\right)}
	\leq \abs{\log Z_{\mathrm{d}}(\beta')} 
	\frac{\sum E_k^2\,\e^{\beta_1 E_k}}{\sum \e^{\beta_1 E_k}}\,,
	\end{equation*}
	which has again finite expectation with the same reasoning as above. One can then check that the first and second derivatives are the same than the ones obtained by derivating under $\es$. The computations are very similar for $f_2$ and $f_3$, the only extra argument one needs for $f_2$ is 
	\[
	\abs{\log \sum \e^{\beta E_k}}  \,\frac{\sum E_k^2\,\e^{\beta E_k}}{\sum \e^{\beta E_k}} 
	\leq \abs{\log \sum \e^{\beta_1 E_k}} \, \frac{\sum E_k^2\,\e^{\beta_1 E_k}}{\sum \e^{\beta_1 E_k}},
	\] 
	for $\beta>\beta_1>1$ using the fact $\beta\mapsto \log \sum \e^{\beta E_k}$ is decreasing and positive.
\end{proof}

\medskip

We can now deduce the main result of this section.

\medskip

\begin{cor}[General decorated case]
 \label{cor:formula_susc}
	Assume that $\es \left[S_\beta\right]<\infty$, for every $\beta>1$. Then, for any $\beta>1$, $\kappa_{\mathrm{d}}(\beta)$ is well-defined and given by
	\[
	\kappa_{\mathrm{d}}(\beta) = \frac{1}{2} \left( \frac{\Var\, \left(\frac{Z_{\mathrm{d}}'(\beta)}{Z_{\mathrm{d}}(\beta)}\right)}{\Var \,(\log Z_{\mathrm{d}}(\beta))}
	- \left(\frac{\Cov\left(\log Z_{\mathrm{d}}(\beta),\frac{Z_{\mathrm{d}}'(\beta)}{Z_{\mathrm{d}}(\beta)}\right)}{\Var\, (\log Z_{\mathrm{d}}(\beta))}\right)^2 \right).
	\]
\end{cor}

\medskip

\begin{proof}
	We expand $\mathbf{C}_{\mathrm{d}}(\beta,\beta+h)$ as $h \to 0$, starting from Equation \eqref{eq:def_C} and using Lemma \ref{lem:derC}, writing $Z_{\mathrm{d}}$ instead of $Z_{\mathrm{d}}(\beta)$ for brevity,
	\begin{align*}
	\mathbf{C}_{\mathrm{d}}(\beta,\beta+h)
	&=\frac{\Cov\left(\log Z_{\mathrm{d}},\log Z_{\mathrm{d}}+\frac{Z_{\mathrm{d}}'}{Z_{\mathrm{d}}}\,h+\left(\frac{Z_{\mathrm{d}}''}{Z_{\mathrm{d}}}-\left(\frac{Z_{\mathrm{d}}'}{Z_{\mathrm{d}}}\right)^2\right)\,h^2\right)}{\sigma\left(\log Z_{\mathrm{d}}\right)\,\sigma\left(\log Z_{\mathrm{d}}+\frac{Z_{\mathrm{d}}'}{Z_{\mathrm{d}}}\,h+\left(\frac{Z_{\mathrm{d}}''}{Z_{\mathrm{d}}}-\left(\frac{Z_{\mathrm{d}}'}{Z_{\mathrm{d}}}\right)^2\right)\,h^2\right)}+o\left(h^2\right)\\
	&=1-\frac{1}{2}\left(\frac{\Var \frac{Z_{\mathrm{d}}'}{Z_{\mathrm{d}}} }{\Var  \log Z_{\mathrm{d}}}-\left(\frac{\Cov\left(\log Z_{\mathrm{d}},\frac{Z_{\mathrm{d}}'}{Z_{\mathrm{d}}}\right)}{\Var \log Z_{\mathrm{d}}}\right)^2\right)h^2+o\left(h^2\right) \, ,
	\end{align*}
	which yields the result.
\end{proof}

\medskip

\subsection{The REM case}

\medskip

In the case of the REM, the quantities appearing in the expression for the susceptibility given in Corollary \ref{cor:formula_susc} are explicit.
To see this, we start by establishing the following formulae. 
Note that the formula for $\es\left[\log Z(\beta) \right]$ already appears in the initial paper by Derrida \cite{Derrida1981}.

\medskip

\begin{lem}[REM case]
	For any $\beta>1$,
	\begin{align}
	\es\left[\log Z(\beta) \right]
	& = \gamma(\beta-1)+\beta\log\Gamma\left(1-\frac{1}{\beta}\right),
	\label{eq:E_log_Z} \\
	\Var (\log Z(\beta))
	& = \frac{\pi^2}{6}\left(\beta^2-1\right).
	\label{eq:Var_log_Z} 
	\end{align}
\end{lem}

\medskip

\begin{proof}
	Let $\beta > 1$ be fixed.
	The Laplace transform of $\log Z(\beta)$ has been computed in Lemma~\ref{esz} and is finite in a neighborhood of 0.
	Hence, differentiating it at 0 yields the formula for $\es[\log Z(\beta)]$, and differentiating once more, we get
	\[
	\es\left[\log^2 Z(\beta)\right] 
	= \frac{\pi^2}{6}\left(\beta^2-1\right) + \gamma ^2 (\beta - 1)^2 + \beta \log\Gamma\left(1-\frac{1}{\beta}\right) \left(2 \gamma (\beta - 1) + \beta \log\Gamma\left(1-\frac{1}{\beta}\right)\right),
	\]
	which yields the formula for the variance.
\end{proof}

\medskip

The following formula is a bit more tricky to obtain.

\medskip

\begin{lem}[REM case]
\label{var}
	For $\beta>1$, we have
	\[
	\Var\,\left(\frac{Z'(\beta)}{Z(\beta)}\right)
	=\frac{\pi^2}{6}+\frac{\beta-1}{\beta^3}\left(
	\frac{\Gamma''\left(\frac{\beta-1}{\beta}\right)}{\Gamma\left(\frac{\beta-1}{\beta}\right)}
	-\left(\frac{\Gamma'\left(\frac{\beta-1}{\beta}\right)}{\Gamma\left(\frac{\beta-1}{\beta}\right)}\right)^2
	\right)\, .
	\]
\end{lem}

\medskip

\begin{proof}
	To lighten notations, let us denote $\Gamma=\Gamma(\frac{\beta-1}{\beta})$ and similarly $\Gamma'=\Gamma'(\frac{\beta-1}{\beta})$, $\Gamma''=\Gamma''(\frac{\beta-1}{\beta})$.
	Differentiating once Equation \eqref{eq:E_log_Z} gives
	\begin{equation} \label{eq:var1}
	\es\left[\frac{Z'(\beta)}{Z(\beta)}\right] = \log\gam+\frac{1}{\beta}\frac{\gam'}{\gam}+\gamma\, ,
	\end{equation}
	and differentiating once more yields
	\begin{equation} \label{eq:var2}
	\es\left[\frac{Z''(\beta)}{Z(\beta)}-\left(\frac{Z'(\beta)}{Z(\beta)}\right)^2\right] = \frac{1}{\beta ^3}\left(\frac{\gam''}{\gam}-\left(\frac{\gam'}{\gam}\right)^2\right)\, .
	\end{equation}
	Therefore, we now aim at computing $\es[Z''(\beta)/Z(\beta)]$.	
	Applying Palm formula (see Proposition~\ref{prop:palmformula}), we get
	\begin{align*}
	\Ec{\frac{Z''(\beta)}{Z(\beta)} }
	&=\int_{0}^{\infty}\log^2(x) \, \Ec{\frac{1}{1+x^\beta Z(\beta)}} \diff x \\
	&=\int_{0}^{\infty} \diff x\,\log^2(x) \int_{0}^{\infty}  \,\e^{-t}\,\es \left[\e^{-tx^\beta Z(\beta)}\right] \diff t \\
	&=\int_{0}^{\infty} \diff t \e^{-t} \int_{0}^{\infty} \log^2(x) \e^{-\gam \, t^{1/\beta}x} \diff x\\
	&=\frac{1}{\gam}\int_{0}^{\infty} \diff t \,t^{-1/\beta}\,\e^{-t} \int_{0}^{\infty} \left(\log u-\log t^{1/\beta}\,\gam \right)^2\,\e^{-u} \diff u \\
	&=\frac{1}{\gam}\int_{0}^{\infty}\diff t \,t^{-1/\beta}\,\,\e^{-t}
	\left(\Gamma''(1)-2\,t^{1/\beta}\,\Gamma'(1)\log\gam \, +t^{1/\beta}\,\log^2\gam \right) \\
	&=\gamma^2+\frac{\pi^2}{6}+2\gamma\log\gam+\log^2\gam+\frac{2\gamma}{\beta}\frac{\gam'}{\gam}+\frac{2}{\beta}\frac{\gam'}{\gam}\log\gam +\frac{1}{\beta^2}\frac{\gam''}{\gam}\, .
	\end{align*}
	Combining this with Equation \eqref{eq:var2} gives
	\begin{align*}
	\es\left[\left(\frac{Z'(\beta)}{Z(\beta)}\right)^2\right]
	&=\gamma^2+\frac{\pi^2}{6}+2\gamma   \log\gam+\log^2\gam+\frac{2\gamma}{\beta}\,\frac{\gam'}{\gam}+\frac{2}{\beta}\frac{\gam'}{\gam}\log\gam+\frac{\beta-1}{\beta^3}\frac{\gam''}{\gam}+\frac{1}{\beta ^3}\left(\frac{\gam'}{\gam}\right)^2 \, ,
	\end{align*}
	which together with Equation \eqref{eq:var1} yields the result.
\end{proof}

\medskip

We can now conclude the section by proving Theorem \ref{thm:THM2_REM} concerning the susceptibility for the REM.

\medskip

\begin{proof}[Proof of Theorem \ref{thm:THM2_REM}]
	We start from the formula given by Corollary \ref{cor:formula_susc}.
	Differentiating Equation \eqref{eq:Var_log_Z} shows that
	\[
	\Cov\left(\log Z(\beta),\frac{Z'(\beta)}{Z(\beta)}\right) 
	= \frac{1}{2} \frac{\diff}{\diff \beta} \Var (\log Z(\beta))
	= \frac{\pi^2}{6}\beta.
	\]
	Together, with Equation \eqref{eq:Var_log_Z} and Lemma \ref{var}, this yields
	\begin{align*}
		\kappa(\beta)
		&=\frac{1}{2}\left(\frac{1}{\beta^2-1}+\frac{6}{\pi^2\beta^3(\beta+1)}\left(\frac{\Gamma''\left(\frac{\beta-1}{\beta}\right)}{\Gamma\left(\frac{\beta-1}{\beta}\right)}-\left(\frac{\Gamma'\left(\frac{\beta-1}{\beta}\right)}{\Gamma\left(\frac{\beta-1}{\beta}\right)}\right)^2\right)-\frac{\beta^2}{(\beta^2-1)^2}\right)\, ,
	\end{align*}
	and using the asymptotic expansion of $\Gamma$ at $0$ and at $1$ together with extra computations yield the desired asymptotic equivalents of $\kappa(\beta)$ as $\beta \downarrow1$ and $\beta \to \infty$.
\end{proof}

\medskip

\subsection{Comparison wih the REM in the general decorated case}

\medskip

In this section, we come back to the general decorated case.
Our aim is to prove the following proposition which compares the temperature susceptibility $\kappa_{\mathrm{d}}(\beta)$ with the one of the REM.

\medskip

\begin{prop}[General decorated case]
\label{prop:susceptibility_decorated}
	Assume that $\es \left[S_\beta\right]<\infty$, for every $\beta>1$.
	Then, for any $\beta > 1$,
	\[
	\kappa_{\mathrm{d}}(\beta) = \kappa(\beta) 
	+ \frac{3}{\pi^2\beta(\beta+1)}
	\Var_\beta \left( \frac{1}{\beta}\log S_\beta - \frac{S'_\beta}{S_\beta} \right),
	\]
	where $\Var_\beta$ denotes the variance under $\P_\beta$, which was introduced in Subsection \ref{sec:change_of_measure}, and $S_\beta' =\sum_{k\geq 0} d_k \e^{\beta d_k}$.
\end{prop}

\medskip

\begin{proof}
We start from the expression given by Corollary \ref{cor:formula_susc}. 
By Corollary \ref{cor:change_of_measure}, we have 
\[
\Var \, \left(\log Z_{\mathrm{d}}(\beta) \right)
= \Var \, \left( \log \Big( \Ec{S_\beta^{1/\beta}}^{\beta} Z(\beta) \Big) \right)
= \Var \, (\log Z(\beta) )
= \frac{\pi^2}{6}(\beta^2-1)
\]
and, by differentiating with respect to $\beta$, 
\[
\Cov \left(\log Z_{\mathrm{d}}(\beta),\frac{Z_{\mathrm{d}}'(\beta)}{Z_{\mathrm{d}}(\beta)}\right) 
= \frac{1}{2} \frac{\diff}{\diff \beta} \Var \,(\log Z_{\mathrm{d}}(\beta))
= \frac{1}{2} \frac{\diff}{\diff \beta} \Var \,(\log Z(\beta))
= \Cov \left(\log Z(\beta),\frac{Z'(\beta)}{Z(\beta)}\right).
\]
Therefore, we get
\begin{equation} \label{eq:intermediate}
\kappa_{\mathrm{d}}(\beta) = \kappa(\beta) 
+ \frac{3}{\pi^2\left(\beta^2-1\right)} 
\left( \Var\,\left( \frac{Z_{\mathrm{d}}'(\beta)}{Z_{\mathrm{d}}(\beta)} \right)
- \Var\, \left( \frac{Z'(\beta)}{Z(\beta)}\right) \right).
\end{equation}
We shall now focus on $\Var\,\left(  \frac{Z_{\mathrm{d}}'(\beta)}{Z_{\mathrm{d}}(\beta)}\right)$.
Notice first that
\[
Z_{\mathrm{d}}'(\beta) = \sum_{i,k} \left(\xi_i+ d_{ik}\right) \e^{\beta(\xi_i + d_{ik})}.
\]
Therefore, it follows from Lemma \ref{lem:change_of_measure} that
\[
\frac{Z_{\mathrm{d}}'(\beta)}{Z_{\mathrm{d}}(\beta)} \; \text{ under } \P
\; \overset{({\rm d})}{=} \;
\frac{\sum_{i,k} \left(\xi_i + c_\beta - \frac{1}{\beta} \log S_{\beta,i} + d_{ik}\right) \e^{\beta(\xi_i- \frac{1}{\beta} \log S_{\beta,i} + d_{ik})}}
{\sum_i \e^{\beta \xi_i}} \; \text{ under } \P_\beta.
\]
Moreover, the quantity on the right-hand side of the last equation equals
\[
c_\beta 
+ \frac{\sum_i \xi_i \e^{\beta\xi_i}}{\sum_i \e^{\beta \xi_i}} 
+ \frac{\sum_i \e^{\beta\xi_i} (S_{\beta,i}'/S_{\beta,i} - \frac{1}{\beta} \log S_{\beta,i})}{\sum_{i} \e^{\beta \xi_i}}\, .
\]
Hence, using the law of total variance, we get
\begin{align*}
\Var \, \left( \frac{Z_{\mathrm{d}}'(\beta)}{Z_{\mathrm{d}}(\beta)} \right)
& = \Var_\beta \left( \E_\beta \left[ 
\frac{\sum_i \xi_i \e^{\beta\xi_i}}{\sum_i \e^{\beta \xi_i}} 
+ \frac{\sum_i \e^{\beta\xi_i} (S_{\beta,i}'/S_{\beta,i} - \frac{1}{\beta} \log S_{\beta,i})}{\sum_{i} \e^{\beta \xi_i}} 
\mathrel{}\middle|\mathrel{} \xi \right]
\right) \\
& \quad {}
+ \E_\beta \left[ \Var_\beta \left( 
\frac{\sum_i \xi_i \e^{\beta\xi_i}}{\sum_i \e^{\beta \xi_i}} 
+ \frac{\sum_i \e^{\beta\xi_i} (S_{\beta,i}'/S_{\beta,i} - \frac{1}{\beta} \log S_{\beta,i})}{\sum_{i} \e^{\beta \xi_i}} 
\mathrel{}\middle|\mathrel{} \xi \right)
\right].
\end{align*}
Then, using the fact that the  $(S_{\beta,i}, S_{\beta,i}')_{i \ge 0}$ are i.i.d.\@ and independent of $\xi$, we get
\begin{align*}
\Var \, \left(  \frac{Z_{\mathrm{d}}'(\beta)}{Z_{\mathrm{d}}(\beta)} \right)
& = \Var_\beta \left( 
\frac{\sum_i \xi_i \e^{\beta\xi_i}}{\sum_i \e^{\beta \xi_i}}
\right) 
+ \E_\beta \left[ 
\frac{\sum_i \e^{2\beta\xi_i}}{(\sum_i \e^{\beta \xi_i})^2} \right]
\Var_\beta \left( \frac{S_\beta'}{S_\beta} - \frac{1}{\beta} \log S_\beta \right) \\
& = \Var \, \left(  \frac{Z'(\beta)}{Z(\beta)} \right)
+ \left( 1 - \frac{1}{\beta} \right)
\Var_\beta \left( \frac{S_\beta'}{S_\beta} - \frac{1}{\beta} \log S_\beta \right).
\end{align*}
Coming back to Equation \eqref{eq:intermediate}, this concludes the proof.
\end{proof}

\bigskip

\begin{ex}\label{ex:deco2}
	We investigate here the behavior of the temperature susceptibility for the decoration processes $\cD = \delta_0 + f(X) \delta_{-X}$, where $X$ is a positive r.v. with density $p$ and $f \colon (0,\infty) \to \N^*$ a measurable function. We specify these parameters afterwards to obtain different behaviors of the susceptibility as $\beta \downarrow 1$.
	
	By Proposition \ref{prop:susceptibility_decorated}, one has
	\begin{equation} \label{eq:suscep}
	\kappa_{\mathrm{d}}(\beta) = \kappa(\beta) 
	+ \frac{3}{\pi^2\beta(\beta+1)}
	\Var_\beta \left( \frac{1}{\beta}\log S_\beta - \frac{S'_\beta}{S_\beta} \right).
	\end{equation}
	In the cases which follow, our strategy to find the behavior of $\kappa_{\mathrm{d}}(\beta) - \kappa(\beta)$ is to estimate precisely
	\begin{equation} \label{eq:formula_in_terms_of_expectations}
	\Var_\beta \left( \frac{S'_\beta}{S_\beta} \right) 
	= \frac{1}{\Ec{S_\beta^{1/\beta}}} \cdot
	\Ec{\frac{(S_\beta')^2}{S_\beta^{2-1/\beta}}} 
	- \frac{1}{\Ec{S_\beta^{1/\beta}}^2} \cdot
	\Ec{\frac{S_\beta'}{S_\beta^{1-1/\beta}}}^2,
	\end{equation}
	where we used the definition of $\P_\beta$, and then to control $\Var_\beta(\log S_\beta)$ to show that its contribution to the variance in Equation \eqref{eq:suscep} is negligible.
	Hence, we now aim at estimating the expectations appearing on the right-hand side of Equation \eqref{eq:formula_in_terms_of_expectations}. Note that $S_\beta = 1 + f(X) e^{-\beta X}$ and $S_\beta' = -X f(X) e^{-\beta X}$.
	If $\E[X^2] < \infty$, we get the following estimates
	\begin{align}
	\Ec{S_\beta^{1/\beta}} 
	& = \int_0^\infty \left(1+f(x) \e^{-\beta x}\right)^{1/\beta} p(x) \diff x
	= \int_0^\infty f(x)^{1/\beta} \e^{-x} p(x) \diff x + O(1), 
	\nonumber \\
	\Ec{\frac{S_\beta'}{S_\beta^{1-1/\beta}}}
	& = \int_0^\infty \frac{-xf(x) \e^{-\beta x}}{\left(1+f(x) \e^{-\beta x}\right)^{1-1/\beta}} p(x) \diff x
	= - \int_0^\infty x f(x)^{1/\beta} \e^{-x} p(x) \diff x + O(1), 
	\nonumber \\
	\Ec{\frac{(S_\beta')^2}{S_\beta^{2-1/\beta}}}  
	& = \int_0^\infty \frac{(xf(x) \e^{-\beta x})^2}{\left(1+f(x) \e^{-\beta x}\right)^{2-1/\beta}} p(x) \diff x
	= \int_0^\infty x^2 f(x)^{1/\beta} \e^{-x} p(x) \diff x + O(1), 
	\label{eq:various_expectations}
	\end{align}
	using respectively 
	$0 \leq (1+a)^{1/\beta} - a^{1/\beta} \leq 1$,
	$0 \leq a^{1/\beta-1} - (1+a)^{1/\beta-1} \leq (1-1/\beta) a^{-1}$ and
	$0 \leq a^{1/\beta-2} - (1+a)^{1/\beta-2} \leq (2-1/\beta) a^{-2}$
	for any $a> 0$ and $\beta > 1$.
	
\medskip
	
	\textit{First family of examples.} We specify $p(x) = (x+1)^{-4}$ and $f(x) = \lceil (x+1)^b \e^{x} \rceil$ with $b \in \R$.
	Then, noting that the ceiling function can be removed up to another $O(1)$ term, 
	one gets, for any $k\geq 0$, as $\beta \downarrow 1$,
	\begin{align*}
	\int_0^\infty x^k f(x)^{1/\beta} \e^{-x} p(x) \diff x 
	& = \int_0^\infty x^k (x+1)^{(b/\beta)-4} \e^{-(\beta-1)x/\beta} \diff x + O(1) \\
	& = \begin{cases}
	O(1), & \text{if } k+b < 3, \smallskip \\
	\log \dfrac{1}{\beta-1} + O(1), 
	& \text{if } k+b = 3, \smallskip \\
	\dfrac{\Gamma(b+k-3)}{(2(\beta-1))^{b+k-3}} (1+o(1)),
	& \text{if } k+b> 3.
	\end{cases} 
	\end{align*}
	Therefore, combining this with Equations \eqref{eq:formula_in_terms_of_expectations} and \eqref{eq:various_expectations}, we get (also note that $\E[S_\beta^{1/\beta}] \geq 1$), as $\beta \downarrow 1$,
	\begin{equation*}
	\Var_\beta \left( \frac{S'_\beta}{S_\beta} \right)  \,
	\begin{cases}
	= O(1), & \text{if } b < 1, \smallskip \\
	\asymp \log \dfrac{1}{\beta-1}, & \text{if } b = 1, \smallskip \\
	\sim \dfrac{\Gamma(b-1)}{(2(\beta-1))^{b-1}}, & \text{if } b \in (1,3), \smallskip \\
	\sim \dfrac{1}{4(\beta-1)^2 \log \frac{1}{\beta-1}}, 
	& \text{if } b = 3, \smallskip \\
	\sim \dfrac{(b-3)}{4(\beta-1)^2},
	& \text{if } b>3.
	\end{cases} 
	\end{equation*}
	We now write $\Var_\beta(\log S_\beta) \leq \E_\beta[\log^2 S_\beta]$ and hence bound, proceeding initially as in Equation \eqref{eq:various_expectations},
	\begin{align*}
	\Ec{S_\beta^{1/\beta} \log^2 S_\beta} 
	& = \int_0^\infty \log^2 \left(1+f(x) \e^{-\beta x}\right) f(x)^{1/\beta} \e^{-x} p(x) \diff x + O(1) \\
	& \leq \int_0^\infty \log^2 \left( 1+(1+x)^b \right) (x+1)^{(b/\beta)-4} \e^{-(\beta-1)x/\beta} \diff x + O(1) \\
	& = \begin{cases}
	O(1), & \text{if } b < 3, \smallskip \\
	O \left( \Ec{S_\beta^{1/\beta}} \cdot \log^2 \dfrac{1}{\beta-1} \right),
	& \text{if } b \geq 3.
	\end{cases} 
	\end{align*}
	It follows that $\Var_\beta(\log S_\beta)$ is either a $O(1)$ if $b<3$, or a $o(\Var_\beta (S_\beta'/S_\beta))$ otherwise.
	Controlling $\Cov_\beta( \log S_\beta,S_\beta'/S_\beta)$ by Cauchy-Schwarz inequality and coming back to Equation \eqref{eq:suscep}, we get
	\begin{equation*} 
	\kappa_{\mathrm{d}}(\beta) - \kappa(\beta)  \,
	\begin{cases}
	= O(1), & \text{if } b < 1, \smallskip \\
	\asymp \log \dfrac{1}{\beta-1}, & \text{if } b = 1, \smallskip \\
	\sim \dfrac{3\Gamma(b-1)}{2^b \pi^2(\beta-1)^{b-1}}, & \text{if } b \in (1,3), \smallskip \\
	\sim \dfrac{3}{8\pi^2(\beta-1)^2 \log \frac{1}{\beta-1}}, 
	& \text{if } b = 3, \smallskip \\
	\sim \dfrac{3(b-3)}{8\pi^2(\beta-1)^2},
	& \text{if } b>3.
	\end{cases} 
	\end{equation*}
	Note that, for this first family of examples, the difference $\kappa_{\mathrm{d}}(\beta) - \kappa(\beta)$ grows at most like $(\beta-1)^{-2}$, which is the order of magnitude of $\kappa(\beta)$ as $\beta \downarrow 1$. More precisely, if $b \leq 3$, then $\kappa_{\mathrm{d}}(\beta) \sim \kappa(\beta)$, and if $b>3$, then both susceptibilities grow at the same speed but with different multiplicative constants.

	In the case $b>3$, observe that the expectations in Equation \eqref{eq:various_expectations} are all dominated by the event where $X \asymp (\beta-1)^{-1}$, so, on this event, $S_\beta'/S_\beta \asymp (\beta-1)^{-1}$ which yields eventually a variance $\Var_\beta(S_\beta'/S_\beta)$ of order $(\beta-1)^{-2}$.
	This observation is the idea motivating the next family of examples: in order to construct cases where $\kappa_{\mathrm{d}}(\beta)$ grows much faster than $\kappa(\beta)$, we want the expectations in Equation \eqref{eq:various_expectations} to be dominated by the event where $X \asymp (\beta-1)^{-\alpha}$ for some $\alpha > 1$.

\medskip
	
	\textit{Second family of examples.}
	We now take $p(x) = \e^{-x^\gamma}$ and $f(x) = \lceil \e^{x+2x^\gamma} \rceil$ for some $\gamma \in (0,1)$.
	First note that, getting rid of the ceiling function and replacing $x$ by $(\beta/(2-\beta))^{1/\gamma} x$, we have
	\begin{equation} \label{eq:begin_2nd_family}
	\int_0^\infty x^k f(x)^{1/\beta} \e^{-x} p(x) \diff x 
	= \left( \frac{\beta}{2-\beta} \right)^{(k+1)/\gamma} \int_0^\infty x^k \e^{x^\gamma-\varepsilon x} \diff x + O(1),
	\end{equation}
	with $\varepsilon \coloneqq (2-\beta)^{-1/\gamma} \beta^{(1-\gamma)/\gamma} (\beta-1)$.
	One can check that this last integral is dominated, as $\varepsilon \to 0$, by the part $x \sim (\gamma/\varepsilon)^{1/(1-\gamma)}$, which means that the expectations in Equation \eqref{eq:various_expectations} are dominated by the event where $X \sim (\gamma/(\beta-1))^{1/(1-\gamma)}$. This implies that the two terms on the right-hand side of Equation \eqref{eq:formula_in_terms_of_expectations} are of order $(\beta-1)^{-2/(1-\gamma)}$, however the first order terms cancel out and $\Var_\beta(S_\beta'/S_\beta)$ is of a smaller order.
	Hence, we have to push the calculation to the second order term: we prove in Appendix \ref{sec:app_integral} that, for $k \in \{0,1,2\}$, as $\varepsilon \to 0$, 
	\begin{align}
	\int_0^\infty x^k \e^{x^\gamma-\varepsilon x} \diff x 
	& = \left( \frac{2\pi}{1-\gamma} \right)^{1/2} 
	\gamma^{\frac{1}{2(1-\gamma)}}
	\varepsilon^{-\frac{2-\gamma}{2(1-\gamma)}} 
	\exp \left( (1-\gamma) \left( \frac{\gamma}{\varepsilon} \right)^{\frac{\gamma}{1-\gamma}} \right) \nonumber \\
	& \quad {} \times 
	\left( \frac{\gamma}{\varepsilon} \right)^{\frac{k}{1-\gamma}} 
	\left( 1 + c_k \varepsilon^{\frac{\gamma}{1-\gamma}} 
	+ O \left( \varepsilon^{\frac{3\gamma}{2(1-\gamma)}} \right) \right),
	\label{eq:expansion_integral}
	\end{align}
	where the constants $c_0,c_1,c_2$ explicitly depend on $\gamma$ and satisfy $c_2+c_0-2c_1 = \frac{1}{1-\gamma} \gamma^{-\frac{1}{1-\gamma}}$.
	Coming back to Equation \eqref{eq:formula_in_terms_of_expectations}, we get 
	\begin{align*}
	\Var_\beta \left( \frac{S'_\beta}{S_\beta} \right) 
	& = \left( \frac{\beta}{2-\beta} \right)^{2/\gamma}
	\left( \frac{\gamma}{\varepsilon} \right)^{\frac{2}{1-\gamma}} 
	\left( (c_2+c_0-2c_1) \varepsilon^{\frac{\gamma}{1-\gamma}} 
	+ O \left( \varepsilon^{\frac{3\gamma}{2(1-\gamma)}} \right) \right)
	\sim \frac{\gamma^{\frac{1}{1-\gamma}}}{(1-\gamma)} 
	(\beta-1)^{-\frac{2-\gamma}{1-\gamma}},
	\end{align*}
	using the definition of $\varepsilon$.
	We now aim at bounding 
	\begin{equation} \label{eq:var_log_2nd}
	\Var_\beta(\log S_\beta) 
	= \frac{\Ec{S_\beta^{1/\beta} \log^2 S_\beta}}{\Ec{S_\beta^{1/\beta}}} 
	- \left( \frac{\Ec{S_\beta^{1/\beta} \log S_\beta}}{\Ec{S_\beta^{1/\beta}}}  \right)^2.
	\end{equation}
	For $k \in \{0,1,2\}$, proceeding in the first equality as in Equations \eqref{eq:various_expectations} and \eqref{eq:begin_2nd_family}, we get
	\begin{align}
	\Ec{S_\beta^{1/\beta} \log^k S_\beta} 
	& = \left( \frac{\beta}{2-\beta} \right)^{1/\gamma} \int_0^\infty 
	\log^k \left( 1+\e^{\frac{2\beta}{2-\beta} x^\gamma-\varepsilon \beta x} \right) \e^{x^\gamma-\varepsilon x} \diff x + O(1) \nonumber \\
	& = \left( \frac{\beta}{2-\beta} \right)^{1/\gamma} 
	\left( \frac{2\pi}{1-\gamma} \right)^{1/2} 
	\gamma^{\frac{1}{2(1-\gamma)}}
	\varepsilon^{-\frac{2-\gamma}{2(1-\gamma)}}  
	\exp \left( (1-\gamma) \left( \frac{\gamma}{\varepsilon} \right)^{\frac{\gamma}{1-\gamma}} \right) \nonumber \\
	& \quad {} \times 
	\left( \left( \frac{2\beta}{2-\beta} -\beta\gamma \right) \left( \frac{\gamma}{\varepsilon} \right)^{\frac{\gamma}{1-\gamma}} \right)^k
	\left( 1 + O \left( \varepsilon^{\frac{\gamma}{1-\gamma}} \right) \right),
	\label{eq:equiv_integral}
	\end{align}
	where the last equality is proved in Appendix \ref{sec:app_integral}.
	Plugging this in \eqref{eq:var_log_2nd}, we get $\Var_\beta(\log S_\beta) = O(\varepsilon^{-2\gamma/(1-\gamma)} \cdot \varepsilon^{\gamma/(1-\gamma)})
	= o(\Var_\beta(S'_\beta/S_\beta))$ because $\gamma < 2-\gamma$.
	Controlling $\Cov_\beta( \log S_\beta,S'_\beta/S_\beta)$ by Cauchy-Schwarz inequality and coming back to \eqref{eq:suscep}, we get
	\[
	\kappa_{\mathrm{d}}(\beta) - \kappa(\beta)
	\sim \frac{3\gamma^{\frac{1}{1-\gamma}}}{2\pi^2(1-\gamma)} 
	(\beta-1)^{-\frac{2-\gamma}{1-\gamma}}.
	\]	
	Note that the exponent $\frac{2-\gamma}{1-\gamma}$ is always larger than 2, so that $\kappa_{\mathrm{d}}(\beta)$ grows faster than $\kappa(\beta)$, and this exponent can be made arbitrarily large by taking $\gamma$ close to 1.
\end{ex}

\medskip

\subsection{The BBM case}

\medskip

In this section, our goal is to prove Theorem \ref{thm:THM2_DECORATED} which concerns the BBM.

\medskip

\begin{lem}[BBM case]
\label{lem:moment1_Sbeta''}
	For the decoration arising in the BBM and the constant $C_\star$ appearing in Equation \eqref{eq:equiv_C_x}, we have, as $\beta \downarrow 1$,
	\[
	\es \left[S_\beta'\right] \sim -\frac{C_\star}{(\beta-1)^2} 
	\qquad \textrm{and} \qquad 
	\es  \left[S_\beta''\right] \sim \frac{2C_\star}{(\beta-1)^3} \, .
	\]
\end{lem}

\medskip

\begin{proof}
	This follows from Equation \eqref{eq:equiv_C_x} in a similar way as the proof of Proposition \ref{prop:moments_points_extremaux}.
\end{proof}

\medskip

\begin{proof}[Proof of Theorem \ref{thm:THM2_DECORATED}]
	We first prove that $\kappa_{\mathrm{d}}(\beta) > \kappa(\beta)$,  when $\beta >1$ is fixed.
	By Proposition~\ref{prop:susceptibility_decorated}, it is sufficient to show that 
	\begin{equation} \label{eq:variance_goal}
	\Var_\beta \left( \frac{1}{\beta}\log S_\beta - \frac{S'_\beta}{S_\beta} \right)
	> 0.
	\end{equation}
	We proceed by contradiction and assume that this variance equals zero.
	Then, there exists $c \in \R$ such that, $\P_\beta$-a.s.,
	\begin{equation} \label{eq:hyp_fausse}
	\frac{1}{\beta}\log S_\beta = \frac{S_\beta'}{S_\beta} + c, 
	\end{equation}
	But $\P$ is absolutely continuous w.r.t.\@ $\P_\beta$ such that Equation \eqref{eq:hyp_fausse} holds also $\P$-a.s.
	On the one hand, using $S_\beta = \sum_{k \geq 0} \e^{\beta d_k}$, we have
	\[
		S_\beta \log S_\beta 
		> \sum_{k \geq 0} \e^{\beta d_k} \beta d_k
		= \beta S_\beta'.
	\]
	Thus, we necessarily have $c > 0$.
	On the other hand, fix some $\beta' \in (1,\beta)$.
	Then, there exists $C > 0$ such that 
	\begin{equation} \label{eq:ineg}
	\forall x \leq 0, \qquad 
	\left(1 \vee \abs{x}\right)\, \e^{\beta x} \leq C \, \e^{\beta' x}.
	\end{equation}
	Now, for some $\varepsilon >0$, consider the event
	\[
	\left\{\sum_{k \geq 1} \e^{\beta' d_k} \leq \varepsilon \right\}.
	\]
	This event has positive $\P$-probability by \cite[Proposition 3.4]{bonnefont22}.
	So, on this event intersected with the one where Equation \eqref{eq:hyp_fausse} holds (this intersection being non-empty), we have
	\[
		c = \frac{1}{\beta}\log S_\beta - \frac{S'_\beta}{S_\beta}
		\leq \frac{1}{\beta}\log \left( 1 + \sum_{k\geq 1} \e^{\beta d_k} \right) 
		+ \sum_{k\geq 1} d_k \e^{\beta d_k}
		\leq \frac{1}{\beta}\log \left( 1 + C \varepsilon \right) 
		+ C \varepsilon,
	\]
	using Equation \eqref{eq:ineg}.
	Letting $\varepsilon \to 0$ shows $c \leq 0$, which contradicts $c > 0$ and concludes our proof of Equation \eqref{eq:variance_goal}.
	
	\medskip
	
	We study now  the regime $\beta \downarrow 1$.
	Since the asymptotics of $\kappa(\beta)$ are given by Theorem \ref{thm:THM2_REM}, it remains to study the behavior of 
	\begin{equation} \label{eq:decompo_var}
	\Var_\beta \left( \frac{1}{\beta}\log S_\beta - \frac{S'_\beta}{S_\beta} \right) 
	= \frac{1}{\beta^2}\Var_\beta \, (\log S_\beta)
	- \frac{2}{\beta}\Cov_\beta\left(\log S_\beta,\frac{S'_\beta}{S_\beta} \right)
	+ \Var_\beta \left( \frac{S'_\beta}{S_\beta} \right).
	\end{equation}
	As we will see, the main term is the third one, which is of order $(\beta-1)^{-2}$, while the other terms are negligible.
	
	For the first term on the right-hand side of Equation \eqref{eq:decompo_var}, observe that
	\[
	\Var_\beta \, ( \log S_\beta  )
	\leq \E_\beta \left[ \log^2 S_\beta  \right]
	\leq \E_\beta \left[ \log^2 (\e+S_\beta) \right]
	\leq \log^2\left( \e+\E_\beta\left[ S_\beta\right] \right),
	\]
	by Jensen's inequality. 
	Then, by definition of $\E_\beta$, we get
	\[
	\E_\beta\left[S_\beta\right]
	= \frac{\es\left[S_\beta^{1+1/\beta}\right]}
	{\es\left[S_\beta^{1/\beta}\right]} 
	\leq \frac{\es\left[S_\beta^2\right]}{\es\left[S_\beta^{1/\beta}\right]} 
	= O \left( \frac{1}{(\beta-1)^2} \right),
	\]
	as $\beta \downarrow 1$ by Equation \eqref{eq:moment2_Sbeta} and Corollary \ref{cor:moment_1/beta}. 
	This proves
	\begin{equation} \label{eq:first_term}
	\Var_\beta \, (\log S_\beta)
	= O \left(  \log^2 \frac{1}{\beta-1} \right).
	\end{equation}
	
	We now consider the third term on the right-hand side of Equation \eqref{eq:decompo_var}. 
	Using the definition of $\E_\beta$ and then that $S_\beta \geq 1$, we get 
	\begin{equation} \label{eq:var_1}
	\E_\beta \left[\frac{S'_\beta}{S_\beta}\right]
	= \frac{\es\left[ S_\beta^{1/\beta-1} S'_\beta \right]}{\es\left[S_\beta^{1/\beta}\right]}
	\sim -\frac{1}{\beta-1},
	\end{equation}
	as $\beta \downarrow 1$ by Corollary \ref{cor:moment_1/beta} and the fact that $\E\left[ S_\beta^{1/\beta-1} (-S'_\beta)\right] \sim C_\star(\beta-1)^{-2}$ as a consequence of the first part of Equation \eqref{eq:ineq_X_Y} (with $X = S_\beta$ and $Y = -S_\beta'$) together with Lemma \ref{lem:moment1_Sbeta''}. 
	On the other hand, it follows from Cauchy--Schwarz inequality that
	\[
	\frac{S_\beta'^2}{S_\beta}
	= S_\beta \cdot  \left( \sum_{k\geq 0} d_k \frac{\e^{\beta d_k}}{S_\beta} \right)^2
	\leq  S_\beta \cdot \sum_{k\geq 0} d_k^2 \,\frac{\e^{\beta d_k}}{S_\beta} 
	= S_\beta''\, ,
	\]
	and therefore
	\begin{equation} \label{eq:var_2}
	\E_\beta \left[ \left( \frac{S'_\beta}{S_\beta} \right)^2 \right]
	=\frac{\es\left[ S_\beta^{1/\beta-2} S'^2_\beta  \right]}{\es\left[S_\beta^{1/\beta}\right]}
	\leq \frac{\es\left[S'^2_\beta/S_\beta\right]}{\es\left[S_\beta^{1/\beta}\right]}
	\leq \frac{\es\left[S_\beta''\right]}{\es\left[S_\beta^{1/\beta}\right]}
	\sim \frac{2}{(\beta-1)^2},
	\end{equation}
	using Lemma \ref{lem:moment1_Sbeta''} and Corollary \ref{cor:moment_1/beta}.
	For the lower bound, using the first part of Equation \eqref{eq:ineq_X_Y} with $X = S_\beta$ and $Y = (S_\beta')^2/S_\beta$, we get
	\[
	\E_\beta \left[ \left( \frac{S'_\beta}{S_\beta} \right)^2 \right]
	= \frac{\es\left[ S_\beta^{1/\beta-2} S'^2_\beta \right]}{\es\left[S_\beta^{1/\beta}\right]}
	\geq \frac{\Ec{S'^2_\beta/S_\beta}^{2-1/\beta}}{\Ec{S_\beta^{1/\beta}} \Ec{S_\beta}^{1-1/\beta}}.
	\]
	Hence, applying Lemma \ref{lem:lower_bound_susceptibility}, Proposition \ref{prop:moments_points_extremaux} and Corollary \ref{cor:moment_1/beta}, we get
	\begin{equation} \label{eq:var_3}
	\liminf_{\beta \downarrow 1} \,(\beta-1)^2 \,
	\E_\beta \,\left[ \left( \frac{S'_\beta}{S_\beta} \right)^2 \right]
	> 1.
	\end{equation}
	Combining Equations \eqref{eq:var_1}, \eqref{eq:var_2} and \eqref{eq:var_3} yields
	\begin{equation} \label{eq:var_final}
	0 <
	\liminf_{\beta \downarrow 1} 
	(\beta-1)^2 \Var_\beta \left( \frac{S'_\beta}{S_\beta} \right)
	\leq \limsup_{\beta \downarrow 1} 
	(\beta-1)^2 \Var_\beta \left( \frac{S'_\beta}{S_\beta} \right)
	\leq 1.
	\end{equation}
	
	Finally, for the second term on the right-hand side of Equation \eqref{eq:decompo_var}, by Cauchy--Schwarz inequality and then Equations \eqref{eq:first_term} and \eqref{eq:var_final}, we get
	\[
	\abs{\Cov_\beta\left(\log S_\beta,\frac{S'_\beta}{S_\beta} \right)}
	\leq \left( 
	\Var_\beta \left( \log S_\beta \right) 
	\Var_\beta \left( \frac{S'_\beta}{S_\beta} \right)
	\right)^{1/2}
	= O \left( \frac{\log \,\frac{1}{\beta-1} }{\beta-1} \right),
	\]
	which proves that this term is negligible in Equation \eqref{eq:decompo_var} and concludes the proof.
\end{proof}

\medskip

\begin{lem}[BBM case]
\label{lem:lower_bound_susceptibility}
	We have
	\[
	\liminf_{\beta \downarrow 1}\, 
	(\beta-1)^3 \,\Ec{\frac{(S'_\beta)^2}{S_\beta}}
	> C_\star.
	\]
\end{lem}

\medskip

\begin{rem}
	A lower bound with a weak inequality could be easily obtained via Cauchy--Schwarz inequality:
	\[
	\Ec{\frac{(S'_\beta)^2}{S_\beta}}
	\geq \frac{\E[S_\beta']^2}{\Ec{S_\beta}}
	\sim \frac{C_\star}{(\beta-1)^3},
	\]
	using Lemma \ref{lem:moment1_Sbeta''} and Proposition \ref{prop:moments_points_extremaux}. Equality at the first order in this inequality would suggest that $S'_{\beta}$ and $S_\beta$ are colinear at first order (on the events that are significant for the first moment).
	Therefore, the idea of the proof below is to find an event such that the first moments of $S'_{\beta}$ and $S_\beta$ given this event have a different ratio than the one for the non-conditional first moments.
\end{rem}

\medskip

\begin{proof}
	\textit{\underline{Step 1: Working at finite $t$.}} Fix some $\beta > 1$. 
	Setting $f_\beta(x) \coloneqq \e^{\beta \sqrt{2} x}$ and $\partial_\beta f_\beta(x) \coloneqq \sqrt{2} x\e^{\beta \sqrt{2} x}$, note that
	$S_\beta = \cC(f_\beta)$ and $S_\beta' = \cC(\partial_\beta f_\beta)$.
	Our first aim in this step is to show
	\begin{equation} \label{eq:step_1}
	\Ec{\frac{(S'_\beta)^2}{S_\beta}}
	= \lim_{t \to \infty} \hEc{\frac{(\cC_{t,r_t}^*(\partial_\beta f_\beta))^2}{\cC_{t,r_t}^*(f_\beta)}}.
	\end{equation}
	For $K > 0$, let $\chi_K$ denote a continuous function such that $\1_{[-K,K]} \leq \chi_K \leq \1_{[-(K+1),K+1]}$.
	It follows from the vague convergence stated in Equation \eqref{eq:description_deco} that
	\[
	\left(\cC_{t,r_t}^*(\chi_K f_\beta),\cC_{t,r_t}^*(\chi_K \partial_\beta f_\beta)\right) 
	\text{ under } \widetilde{\P}_t
	\quad \xrightarrow[t\to\infty]{\text{(d)}} \quad 
	\left(\cC(\chi_K f_\beta),\cC(\chi_K \partial_\beta f_\beta)\right)  \,
	\text{ under } \P.
	\]
	Moreover, recalling $\cC_{t,r_t}^*$ is supported on $(-\infty,0]$ under $\widetilde{\P}_t$ we get
	\[
	\abs{ \frac{\cC_{t,r_t}^*(\chi_K \partial_\beta f_\beta)^2}
	{\cC_{t,r_t}^*(\chi_K f_\beta)} }
	\leq 2(K+1)^2 \cC_{t,r_t}^*(\chi_K f_\beta)
	\leq 2(K+1)^2 \cC_{t,r_t}^*([-(K+1),0])\,,
	\]
	which is bounded in $L^2$ by Lemma \ref{lem:moment_C_t_r_t}. Hence, $(\cC_{t,r_t}^*(\chi_K \partial_\beta f_\beta)^2 / \cC_{t,r_t}^*(\chi_K f_\beta))_t$ is uniformly integrable and we get
	\begin{equation} \label{eq:step_1_bis}
	\Ec{\frac{\cC(\chi_K \partial_\beta f_\beta)^2}{\cC(\chi_K f_\beta)}}
	= \lim_{t \to \infty} 
	\hEc{\frac{\cC_{t,r_t}^*(\chi_K \partial_\beta f_\beta)^2}
		{\cC_{t,r_t}^*(\chi_K f_\beta)}}.
	\end{equation}
	We now control the differences between the expectations in Equation \eqref{eq:step_1} and Equation \eqref{eq:step_1_bis} and show they are small when $K$ is large.
	We have
	\begin{align}
	\abs{\frac{\cC(\partial_\beta f_\beta)^2}{\cC(f_\beta)} 
		- \frac{\cC(\chi_K \partial_\beta f_\beta)^2}{\cC(\chi_K f_\beta)}}
	& \leq \frac{\abs{\cC(\partial_\beta f_\beta)^2 - \cC(\chi_K \partial_\beta f_\beta)^2}}{\cC(f_\beta)}
	+ \cC(\chi_K \partial_\beta f_\beta)^2 
	\abs{\frac{1}{\cC(f_\beta)} - \frac{1}{\cC(\chi_K f_\beta)}} \nonumber \\
	& \leq 2\, \cC(\partial_\beta f_\beta)\, \cC\left(\1_{(-\infty,-K]} \,\partial_\beta f_\beta\right)
	+ (K+1)^2 \cC\left(\1_{(-\infty,-K]} \,\partial_\beta f_\beta\right),
	\label{eq:step_1_ter}
	\end{align}
	where, for the first term, we used $\cC(f_\beta) \geq 1$ and the fact that $\partial_\beta f_\beta$ is of constant sign on $(-\infty,0]$ which is the support of $\cC$ and, for the second term, we used $\abs{\cC(\chi_K \partial_\beta f_\beta)}/\cC(f_\beta) \leq \abs{\cC(\chi_K \partial_\beta f_\beta)}/\cC(\chi_K f_\beta) \leq K+1$.
	Then, writing $\partial_\beta f_\beta (x) = \int_{-\infty}^x \sqrt{2}(\beta\sqrt{2}y+1)\e^{\beta\sqrt{2} y} \diff y$ and using Fubini's theorem, we have
	\begin{align*}
	& \Ec{ \cC(\partial_\beta f_\beta) \cC(\1_{(-\infty,-K]} \partial_\beta f_\beta) } \\
	& = \int_{-\infty}^0 \int_{-\infty}^{-K}  
	\Ec{\cC([y,0])\, \cC([z,-K])} \,2 (\beta\sqrt{2}z+1)\,\e^{\beta\sqrt{2} z} (\beta\sqrt{2}y+1)\e^{\beta\sqrt{2} y} \diff z \diff y \\
	& \leq C(\beta) (K+1)^{3/2} \e^{-(\beta-1)\sqrt{2}K},
	\end{align*}
	where we bounded the last expectation by $(\abs{y}+1)^{1/2} (\abs{z}+1)^{1/2} \e^{-\sqrt{2} (y+z)}$ using Cauchy--Schwarz inequality and Equation \eqref{eq:2nd_moment_C_x}, and where $C(\beta)$ denotes a constant depending only on $\beta$ and which can change from line to line. 
	Similarly, we have $\E[\cC(\1_{(-\infty,-K]} \partial_\beta f_\beta)] \leq C(\beta) (K+1) \e^{-(\beta-1)\sqrt{2}K}$ by using Equation \eqref{eq:bound_C_x}.
	Coming back to Equation \eqref{eq:step_1_ter}, we get 
	\[
	\abs{ \Ec{\frac{\cC(\partial_\beta f_\beta)^2}{\cC(f_\beta)}}
		- \Ec{\frac{\cC(\chi_K \partial_\beta f_\beta)^2}{\cC(\chi_K f_\beta)}}}
	\leq C(\beta) (K+1)^3 \e^{-(\beta-1)\sqrt{2}K},
	\]
	and the same holds true for $\cC_{t,r_t}^*$ instead of $\cC$, uniformly in $t \geq 1$, by replacing Equations \eqref{eq:2nd_moment_C_x} and \eqref{eq:bound_C_x} by Lemma \ref{lem:moment_C_t_r_t} in the proof.
	Therefore, letting $K \to \infty$ in Equation \eqref{eq:step_1_bis} yields Equation \eqref{eq:step_1}.
	
\medskip
	
	\textit{\underline{Step 2: Using Cauchy--Schwarz inequality conditionally on a well-chosen event.}} We fix some parameters $a > 0$ and $0 < b < B$. 
	For any $\beta >1$, letting $r = r(\beta) \coloneqq a (\beta-1)^{-2}$, and for any $t > r$, we introduce the event
	\begin{equation}
	\label{eq:B_{r,t}}
	B_{r,t} = \left\{ h_{t-r}(X_{t-r}) -m_t + m_r \in [-B\sqrt{r},-b\sqrt{r}] \right\}.
	\end{equation}
	Then, we use Cauchy--Schwarz inequality given $B_{r,t}$ or given $B_{r,t}^c$ to get
	\begin{align}
	\hEc{\frac{(\cC_{t,r_t}^*(\partial_\beta f_\beta))^2}{\cC_{t,r_t}^*(f_\beta)}}
	& = \hEcsq{\frac{(\cC_{t,r_t}^*(\partial_\beta f_\beta))^2}{\cC_{t,r_t}^*(f_\beta)}}{B_{r,t}} \hPp{B_{r,t}}
	+ \hEcsq{\frac{(\cC_{t,r_t}^*(\partial_\beta f_\beta))^2}{\cC_{t,r_t}^*(f_\beta)}}{B_{r,t}^c} \hPp{B_{r,t}^c} \nonumber \\
	& \geq \frac{\hEcsq{\cC_{t,r_t}^*(\partial_\beta f_\beta)}{B_{r,t}}^2}{\hEcsq{\cC_{t,r_t}^*(f_\beta)}{B_{r,t}}} \,\hPp{B_{r,t}}
	+ \frac{\hEcsq{\cC_{t,r_t}^*(\partial_\beta f_\beta)}{B_{r,t}^c}^2}{\hEcsq{\cC_{t,r_t}^*(f_\beta)}{B_{r,t}^c}} \,\hPp{B_{r,t}^c} \nonumber \\
	& = \frac{\hEc{\cC_{t,r_t}^*(\partial_\beta f_\beta) \1_{B_{r,t}}}^2}{\hEc{\cC_{t,r_t}^*(f_\beta) \1_{B_{r,t}}}}
	+ \frac{\hEc{\cC_{t,r_t}^*(\partial_\beta f_\beta) \1_{B_{r,t}^c}}^2}{\hEc{\cC_{t,r_t}^*(f_\beta) \1_{B_{r,t}^c}}}, \label{eq:step_2}
	\end{align}
	which yields a lower bound for the right-hand side of Equation \eqref{eq:step_1}, that we now have to estimate.

\medskip

	\textit{\underline{Step 3: Estimating the expectations.}} Let $\varepsilon > 0$. Our aim in this step consists in proving that, for $\beta > 1$ close enough to 1, there exists $t_0 > r$ such that, for any $t \geq t_0$,
	\begin{align}
	\abs{(\beta-1) \hEc{\cC_{t,r_t}^*(f_\beta) \1_{B_{r,t}}} - C_\star \kappa}
	\leq \varepsilon, \nonumber \\
	\abs{(\beta-1)^2 \hEc{\cC_{t,r_t}^*(\partial_\beta f_\beta) \1_{B_{r,t}}} + C_\star \kappa'}
	\leq \varepsilon, \label{eq:step_3}
	\end{align}
	where $\kappa, \kappa'$ are constants depending only on the parameters $a,b,B$ defined as follows
	\begin{align}
	\kappa & \coloneqq \frac{1}{2 \sqrt{\pi}} \int_0^\infty 
	\varphi_{b,B} \left( \frac{a\vee w}{\abs{a-w}} \right)
	\left( \int_0^\infty u \e^{-u-u^2/(4w)} \diff u \right)
	\frac{\diff w}{w^{3/2}}\,, \nonumber \\
	\kappa' & \coloneqq \frac{1}{2 \sqrt{\pi}} \int_0^\infty 
	\varphi_{b,B} \left( \frac{a\vee w}{\abs{a-w}} \right)
	\left( \int_0^\infty u^2 \e^{-u-u^2/(4w)} \diff u \right)
	\frac{\diff w}{w^{3/2}}\,, \label{eq:def_kappa}
	\end{align}
	with $\varphi_{b,B}$ defined by
	\begin{equation} \label{eq:def_varphi}
	\varphi_{b,B}(v) 
	\coloneqq \int_{b \sqrt{v}}^{B\sqrt{v}} \sqrt{\frac{2}{\pi}}\, y^2 \e^{-y^2/2} \diff y\,,
	\qquad v \geq 0.
	\end{equation}
	Note that $\kappa,\kappa' \in (0,1)$ and they tend to $1$ as $b \to 0$ and $B \to \infty$ (which means $\1_{B_{r,t}} \to 1$).
	To prove the first inequality in Equation \eqref{eq:step_3}, we write
	\[
		\hEc{ \cC_{t,r_t}^*(f_\beta) \1_{B_{r,t}} }
		= \int_0^\infty \hEc{\cC_{t,r_t}^*([-x,0]) \1_{B_{r,t}}} \beta \sqrt{2}\, \e^{-\beta\sqrt{2}x} \diff x.
	\]
	Then, we apply Lemma \ref{lem:level_sets_on_B} for some $\theta \in (0,1)$ and with $\varepsilon > 0$  introduced earlier to estimate the part $x \in [\theta \sqrt{r},\theta^{-1} \sqrt{r}]$ of the integral
	and use Lemma \ref{lem:moment_C_t_r_t} to bound the remaining part.
	Note also that $\varphi_{b,B} \leq 1$ and $\int_0^\infty 
	\frac{x\e^{-x^2/(2s)}}{\sqrt{2\pi}s^{3/2}} \diff s = 1$.
	Therefore, for $\beta > 1$ close enough to 1 (equivalently $r$ large enough), there exists $t_0 > r$ such that, for any $t \geq t_0$,
	\begin{align}
	& \abs{\hEc{ \cC_{t,r_t}^*(f_\beta) \1_{B_{r,t}} }
		- C_\star \int_0^\infty \left(\int_0^\infty 
		\varphi_{b,B} \left( \frac{r\vee s}{\abs{r-s}} \right)
		\frac{x\e^{-x^2/(2s)}}{\sqrt{2\pi}s^{3/2}} \diff s \right)
		\beta \sqrt{2}\, \e^{-(\beta-1)\sqrt{2}x} \diff x} 
	\nonumber \\
	& \leq \varepsilon \int_0^\infty
	\beta \sqrt{2} \e^{-(\beta-1)\sqrt{2}x} \diff x
	+ C \int_{[\theta \sqrt{r},\theta^{-1} \sqrt{r}]^c} \beta \sqrt{2}\, \e^{-(\beta-1)\sqrt{2}x} \diff x 
	\nonumber \\
	& = \frac{\beta}{\beta-1} \left(
	\varepsilon
	+ C \left( 1-\e^{-\theta \sqrt{a}} + \e^{-\theta^{-1} \sqrt{a}} \right)
	\right), \label{eq:step_3_bis}
	\end{align}
	recalling that $r = a(\beta-1)^{-2}$.
	Choosing $\theta$ small enough and considering $\beta < 2$, the right-hand side of Equation \eqref{eq:step_3_bis} is smaller that $3 \varepsilon/(\beta-1)$.
	Then, we rewrite the double integral on the left-hand side of Equation \eqref{eq:step_3_bis} by using Fubini's theorem and changing variables with $u= (\beta-1)\sqrt{2}x$ and $w=s(\beta-1)^2$, which shows that this double integral equals $\beta \kappa/(\beta-1)$.
	This proves the first inequality in Equation \eqref{eq:step_3} (with $4\varepsilon$ instead of $\varepsilon$).
	The second inequality is proved by writing 
	\[
	\hEc{ \cC_{t,r_t}^*(f_\beta) \1_{B_{r,t}} }
	= - \int_0^\infty \hEc{\cC_{t,r_t}^*([-x,0]) \1_{B_{r,t}}} \sqrt{2} (\beta \sqrt{2} x-1) \e^{-\beta\sqrt{2}x} \diff x
	\]
	and then proceeding similarly (note that $\beta \sqrt{2} x-1$ can be replaced by $\beta \sqrt{2} x$ up to a negligible error as $\beta \downarrow 1$).

\medskip
	
	\textit{\underline{Step 4: Conclusion.}}
	We first claim that for $\theta \in (0,1)$ and $\epsilon > 0$, there exists $r_0 > 0$ such that, for any $r \geq r_0$, there exists $t_0 > 0$ such that, for any $t \geq t_0$ and any $v \in [\theta \sqrt{r},\theta^{-1} \sqrt{r}]$,
	\[
	\abs{ \hEc{\cC_{t,r_t}^*([-v,0])} - C_\star \e^{\sqrt{2} v}} 
	\leq \varepsilon \, \e^{\sqrt{2} v}.
	\]
	This is a slightly stronger version of \cite[Lemma 5.2]{cortineshartunglouidor2019} where no uniformity for $v \in [\theta \sqrt{r},\theta^{-1} \sqrt{r}]$ is stated: however, the aforementioned claim follows from their proof (more precisely, it follows from \cite[Lemma 5.6]{cortineshartunglouidor2019} in the same way as Lemma \ref{lem:level_sets_on_B} below follows from Lemma \ref{lem:j}).
	Then, we deduce from this claim that, for $\beta> 1$ close enough to 1, up to a modification of $t_0$, we also have, for any $t \geq t_0$,
	\begin{align*}
	\abs{(\beta-1) \hEc{\cC_{t,r_t}^*(f_\beta)} - C_\star}
	\leq \varepsilon, \\
	\abs{(\beta-1)^2 \hEc{\cC_{t,r_t}^*(\partial_\beta f_\beta)} + C_\star}
	\leq \varepsilon,
	\end{align*}
	where these inequalities are obtained in the same way as Equations \eqref{eq:step_3} have been obtained from Lemma \ref{lem:level_sets_on_B} in Step 3.
	Combining this with Equations \eqref{eq:step_1}, \eqref{eq:step_2} and \eqref{eq:step_3}, we get 
	\[
	\Ec{\frac{(S'_\beta)^2}{S_\beta}}
	\geq \frac{1}{(\beta-1)^3}
	\left( 
	\frac{(C_\star \kappa' - \varepsilon)^2}{C_\star \kappa + \varepsilon}
	+ \frac{(C_\star (1-\kappa') - \varepsilon)^2}{C_\star (1-\kappa) + \varepsilon}
	\right).
	\]
	Letting $\varepsilon \to 0$, this proves 
	\begin{equation} \label{eq:step_4}
	\liminf_{\beta \downarrow 1} 
	(\beta-1)^3 \,\Ec{\frac{(S'_\beta)^2}{S_\beta}}
	\geq C_\star \left( 
	\frac{(\kappa')^2}{\kappa}
	+ \frac{(1-\kappa')^2}{1-\kappa}
	\right).
	\end{equation}
	By Lemma \ref{lem:kappa}, we can choose $a,b,B$ such that $\kappa \neq \kappa'$ and, together with the fact that $\kappa,\kappa' \in (0,1)$, this implies that the right-hand side of Equation \eqref{eq:step_4} is larger than $C_\star$.
\end{proof}

The proof of the following lemma is postponed to Subsection \ref{sec:first_moment_event}.

\medskip

\begin{lem}[BBM case]
 \label{lem:level_sets_on_B}
	Let $\theta \in (0,1)$, $\varepsilon,a > 0$ and $0 < b < B$. There exists $r_0 > 0$ such that, for any $r \geq r_0$, there exists $t_0 > 0$ such that, for any $t \geq t_0$ and any $v \in [\theta \sqrt{r},\theta^{-1} \sqrt{r}]$,
	\[
	\abs{ \hEc{\cC_{t,r_t}^*([-v,0]) \1_{B_{r,t}}} 
		- C_\star \e^{\sqrt{2} v}  \int_0^\infty 
		\varphi_{b,B} \left( \frac{r\vee s}{\abs{r-s}} \right)
		\frac{v\e^{-v^2/(2s)}}{\sqrt{2\pi}s^{3/2}} \diff s} 
	\leq \varepsilon \e^{\sqrt{2} v}.
	\]
\end{lem}

\medskip

\begin{lem}[BBM case]
 \label{lem:kappa}
	Recall the definition of $\kappa$ and $\kappa'$ in Equation \eqref{eq:def_kappa}. There exist $a > 0$ and $0 < b <B$ such that $\kappa \neq \kappa'$.
\end{lem}

\medskip

\begin{proof}
	We proceed by contradiction: assume that, for any $a > 0$ and $0 < b <B$, $\kappa = \kappa'$.
	Then, differentiating the relation $\kappa = \kappa'$ w.r.t.\@ $B$, we get, for any $a,B > 0$,
	the following quantity is the same for $k=1$ and $k=2$:
	\begin{align*}
	\int_0^\infty 
	\frac{a\vee w}{\abs{a-w}} \exp \left( - \frac{B^2}{2} \frac{a\vee w}{\abs{a-w}} \right)
	\left( \int_0^\infty u^k \e^{-u-u^2/(4w)} \diff u \right)\,
	\frac{\diff w}{w^{3/2}}.
	\end{align*}
	On the other hand, a direct calculation shows that, for $k \in \{1,2\}$,
	\begin{align*}
	\int_0^\infty 
	\left( \int_0^\infty u^k \e^{-u-u^2/(4w)} \diff u \right)\,
	\frac{\diff w}{w^{3/2}} = 2 \sqrt{\pi}
	\end{align*}
	and therefore is the same for $k=1$ and $k=2$.
	Therefore, we deduce that the following quantity is the same for $k=1$ and $k=2$:
	\begin{equation} \label{eq:int_2}
	\int_0^\infty 
	\left( \exp \left( - \frac{B^2}{2} \frac{a\vee w}{\abs{a-w}} \right) - \e^{-B^2/2} \right)
	\left( \int_0^\infty u^k \e^{-u-u^2/(4w)} \diff u \right)\,
	\frac{\diff w}{w^{3/2}}.
	\end{equation}
	Our goal is now to study the behavior as $a \to 0$ of this quantity, to find a contradiction.
	We decompose the main integral in Equation \eqref{eq:int_2} into a part $w \in [0,a)$ in which we change $w$ to $x = a/(a-w)$ and $u$ to $v = u (x/(a(x-1)))^{1/2}$, and a part $w \in (a,\infty)$ in which we change $w$ to $x = w/(w-a)$ and $u$ to $v = u ((x-1)/(ax))^{1/2}$. This yields that the following quantity is the same for $k=1$ and $k=2$:
	\begin{equation} \label{eq:int_3}
	\begin{split}
	& a^{k/2} \int_1^\infty \left( x \e^{-xB^2/2} - \e^{-B^2/2} \right)
	\Biggl(
	\frac{(x-1)^{k/2-1}}{x^{k/2+1}}
	\int_0^\infty v^k \e^{-v^2/4-v(\frac{a(x-1)}{x})^{1/2}} \diff v \\
	& \hspace{5.5cm} {} + \frac{x^{k/2-1}}{(x-1)^{k/2+1}}
	\int_0^\infty v^k \e^{-v^2/4-v(\frac{ax}{x-1})^{1/2}} \diff v
	\Biggr)
	\diff x. 
	\end{split}
	\end{equation}
	For $k=1$, both integrals w.r.t.\@ $v$ in \eqref{eq:int_3} converge as $a \to 0$ towards $\int_0^\infty v \e^{-v^2/4} \diff v = 2$ by dominated convergence.
	Hence, by dominated convergence again but in the integral w.r.t.\@~$x$, \eqref{eq:int_3} equals
	\[
	2 a^{1/2} \int_1^\infty \left( x \e^{-xB^2/2} - \e^{-B^2/2} \right)
	\Biggl(
	\frac{(x-1)^{-1/2}}{x^{3/2}}
	 + \frac{x^{-1/2}}{(x-1)^{3/2}}
	\Biggr)
	\diff x
	+ o(a^{1/2}),
	\]
	as $a \to 0$.
	For $k=2$, we cannot use the same argument: the second domination cannot be justified.
	Instead, we bound the first integral w.r.t.\@ $v$ by a constant and, for the second one, for some $\ep> 0$, we write
	\[
	\int_0^\infty v^2 \e^{-v^2/4-v(\frac{ax}{x-1})^{1/2}} \diff v
	= \left(\frac{x-1}{ax}\right)^{3\ep} \int_0^\infty u^2 \e^{-u^2(\frac{x-1}{ax})^{2\ep}/4-u(\frac{ax}{x-1})^{1/2-\ep}} \diff u
	\leq C \left(\frac{x-1}{ax}\right)^{3\ep},
	\]
	by bounding the last integral by $\int_0^\infty u^2 \e^{-u^2/4} \diff u$ if $(x-1)/ax \geq 1$ and by $\int_0^\infty u^2 \e^{-u} \diff u$ otherwise.
	With $\varepsilon < 1/6$, this proves that \eqref{eq:int_3} is a $o\left(a^{1/2}\right)$ for $k=2$.
	Since \eqref{eq:int_3} is the same for $k=1$ and $k=2$, this implies
	\begin{equation} \label{eq:int_4}
	\int_1^\infty \left( x \e^{-xB^2/2} - \e^{-B^2/2} \right)
	\Biggl(
	\frac{(x-1)^{-1/2}}{x^{3/2}}
	+ \frac{x^{-1/2}}{(x-1)^{3/2}}
	\Biggr)
	\diff x = 0,
	\end{equation}
	for any $B>0$. But the left-hand side of Equation \eqref{eq:int_4} tends to infinity as $B \to \infty$, so this is a contradiction and concludes the proof. 
\end{proof}

\addcontentsline{toc}{section}{Appendix}
\appendix
\begin{appendix}\label{app}

\section{Proof of technical results on the decoration of the BBM}
\label{sec:app_BBM}

This section is dedicated to the proof of several technical results concerning the decoration of the branching Brownian motion, mostly based on ideas introduced in \cite{cortineshartunglouidor2019}. Therefore, we explain how to adapt their argument and use their notation without introducing it.

\subsection{Uniform bounds for moments of level sets}
\label{sec:unif_bound}

\begin{proof}[Proof of Lemma \ref{lem:moment_C_t_r_t}]
	The bound on the second moment is a direct consequence of Lemma \ref{lem:new_5.3}, so we focus here on the first moment. However, it could also be deduced from the proof of  \cite[Lemma 5.3]{cortineshartunglouidor2019}, in a similar way as what is done below for the first moment.
	
	To bound uniformly $\widetilde{\E}_t[\cC_{t,r_t}^*([-v,0])]$, we follow the proof of  \cite[Lemma 5.2]{cortineshartunglouidor2019}, 
	which establishes an asymptotic equivalent for the first moment of $\cC_{t,r_t}^*([-v,0])$ as $t \to \infty$ and then $v \to \infty$.
	The proof begins by writing
	\begin{equation} \label{eq:decond_1}
	\hEc{\cC_{t,r_t}^*([-v,0])} 
	= \frac{ \widetilde{\E} \left[ \cC_{t,r_t}^*([-v,0]) \1_{\{\max_{x\in L_t} h_t(x) \leq m_t\}} \mathrel{}\middle|\mathrel{} 
		h_t(X_t) = m_t \right] }
	{\widetilde{\P} \left( \,\max_{x\in L_t} h_t(x) \leq m_t \mathrel{}\middle|\mathrel{} h_t(X_t) = m_t \right)},
	\end{equation}
	where the denominator satisfies, for some constant $C_1 > 0$, as $t \to\infty$,
	\begin{equation} \label{eq:proba_denom}
	\widetilde{\P} \left( \max_{x\in L_t} h_t(x) \leq m_t \mathrel{}\middle|\mathrel{} h_t(X_t) = m_t \right)
	\sim \frac{C_1}{t},
	\end{equation}
	by  \cite[Lemmata 3.1 and 3.4]{cortineshartunglouidor2019} (the constant $C_1$ equals $2f^{(0)}(0)g(0)$ with notation of \cite[Lemma 3.4]{cortineshartunglouidor2019}).
	In particular, there exists a constant $c > 0$ such that the probability in Equation \eqref{eq:proba_denom} is at least $c/t$ for any $t \geq 1$.
	Therefore, it remains to prove that there exist $C > 0$ and $t_0 \geq 0$ such that, for any $t \geq t_0$ and $v \geq 0$,
	\begin{equation} \label{eq:moment_1}
	\widetilde{\E} \left[ \cC_{t,r_t}^*([-v,0]) \1_{\{\max_{x\in L_t} h_t(x) \leq m_t\}} \mathrel{}\middle|\mathrel{} 
	h_t(X_t) = m_t \right] 
	\leq \frac{C}{t} \e^{\sqrt{2} v}.
	\end{equation}
	Indeed, for $t \leq t_0$, bounding the indicator function by 1 and $\cC_{t,r_t}^*([-v,0]) \leq \#L_{t_0}$, the left-hand side of Equation \eqref{eq:moment_1} is at most
	\begin{equation} \label{eq:moment_1_bis}
	\widetilde{\E} \left[ \#L_{t_0} \mathrel{}\middle|\mathrel{} 
	h_t(X_t) = m_t \right] 
	= \widetilde{\E} \left[ \#L_{t_0} \right] 
	= \int_0^{t_0} \e^{t_0-s} \cdot 2 \diff s
	= C(t_0),
	\end{equation}
	where the first equality follows from the fact that displacement of the spine and branching of the BBM are independent and the second equality uses the facts that the spine branches at rate~2, giving birth to standard BBMs and that a standard BBM has in mean $\e^r$ particles at time $r$.
	
	We now focus on proving Equation \eqref{eq:moment_1}.
	Applying successively  \cite[Lemmata 5.4 and 5.5]{cortineshartunglouidor2019} (note that $j_{t,v}(s) = j_{t,v}^{\geq 0}(s)$), we get for any $t$ such that $r_t \leq t/2$ and any $v \geq 0$,
	\begin{align*}
	\widetilde{\E} \left[ \cC_{t,r_t}^*([-v,0]) \1_{\{\max_{x\in L_t} h_t(x) \leq m_t\}} \mathrel{}\middle|\mathrel{} 
	h_t(X_t) = m_t \right] 
	& = 2 \int_0^{r_t} j_{t,v}(s) \diff s \\
	& \leq \frac{C}{t} (v+1)\, \e^{\sqrt{2} v} \int_0^\infty \frac{\e^{-v^2/(16s)} + \e^{-v/2}}{(s+1)\sqrt{s}} \diff s.
	\end{align*}
	This last integral is smaller than $C((v+1)^{-1}+\e^{-v/2})$ for any $v \geq 0$ (see \cite[Equations (5.44)-(5.45)]{cortineshartunglouidor2019} for details) and so we get Equation \eqref{eq:moment_1}. 
\end{proof}

\subsection{First moment of level sets on a particular event}
\label{sec:first_moment_event}

We prove in this subsection Lemma~\ref{lem:level_sets_on_B}. This is a (non-trivial) refinement of the proof of  \cite[Lemma 5.2]{cortineshartunglouidor2019}.

\begin{proof}[Proof of Lemma \ref{lem:level_sets_on_B}]
	Recalling the definition of $\widetilde{\P}_t$ in Equation \eqref{eq:def_widetilde_P_t}, we first have
	\begin{equation} \label{eq:decond}
	\hEc{\cC_{t,r_t}^*([-v,0]) \1_{B_{r,t}}} 
	= \frac{ \widetilde{\E} \left[ \cC_{t,r_t}^*([-v,0]) \1_{B_{r,t}} \1_{\{\max_{x\in L_t} h_t(x) \leq m_t\}} \mathrel{}\middle|\mathrel{} 
		h_t(X_t) = m_t \right] }
	{\widetilde{\P} \left( \max_{x\in L_t} h_t(x) \leq m_t \mathrel{}\middle|\mathrel{} h_t(X_t) = m_t \right)}.
	\end{equation}
	Recall from Equation \eqref{eq:proba_denom} that the denominator in Equation \eqref{eq:decond} is asymptotically equivalent to $C_1/t$, so we now focus on the numerator.
	
	We introduce the event
	\[
	\cB_{r,t} \coloneqq \left\{ \widehat{W}_{t,r} \in [-B\sqrt{r},-b\sqrt{r}] \right\},
	\]
	which is analog to $B_{r,t}$ (see Equation \eqref{eq:B_{r,t}}) but for the process $(\widehat{W}_{t,s})_{s\in[0,t]}$ defined in \cite[Equation (3.1)]{cortineshartunglouidor2019} as
	\begin{equation} \label{eq:def_hatW}
	\widehat{W}_{t,s} \coloneqq W_s - \gamma_{t,s},
	\quad \text{with }
	\gamma_{t,s} \coloneqq \frac{3}{2 \sqrt{2}} \left( \log_+ s - \frac{s}{t} \log_+ t \right),
	\end{equation}
	for any $0 \leq s \leq t$, with $W$ a standard Brownian motion under $\P$.
	Defining, for $v \geq 0$ and $0 \leq s,r \leq t$,
	\[
	j_{t,v,r}(s) 
	\coloneqq \Ecsq{J_{t,v}(s) \1_{\cB_{r,t}}}{\widehat{W}_{t,0}=\widehat{W}_{t,t}=0},
	\]
	where $J_{t,v}(s)$ is introduced in \cite[Equation (5.9)]{cortineshartunglouidor2019},
	it follows from the proof of  \cite[Lemma 5.4]{cortineshartunglouidor2019}, that
	\begin{equation} \label{eq:esp_as_int}
	\widetilde{\E} \left[ \cC_{t,r_t}^*([-v,0]) \1_{B_{r,t}} \1_{\{\max_{x\in L_t} h_t(x) \leq m_t\}} \mathrel{}\middle|\mathrel{} 
	h_t(X_t) = m_t \right]
	= 2 \int_0^{r_t} j_{t,v,r}(s) \diff s.
	\end{equation}
	As in \cite[Equation (5.17)]{cortineshartunglouidor2019}, for any $M \geq 0$, we split $j_{t,v,r}(s)$ into
	\begin{align*}
	j_{t,v,r}^{<M}(s) 
	& \coloneqq \Ecsq{J_{t,v}(s) \1_{\cB_{r,t}} \1_{\{\lvert \widehat{W}_{t,s} \rvert < M\}}}{\widehat{W}_{t,0}=\widehat{W}_{t,t}=0}, \\
	j_{t,v,r}^{\geq M}(s) 
	& \coloneqq \Ecsq{J_{t,v}(s) \1_{\cB_{r,t}} \1_{\{\lvert \widehat{W}_{t,s} \rvert \geq M\}}}{\widehat{W}_{t,0}=\widehat{W}_{t,t}=0}.
	\end{align*}
	We postpone the estimate of $j_{t,v,r}^{<M}(s)$ to Lemma \ref{lem:j} (which replace Lemma 5.6 in \cite{cortineshartunglouidor2019}).
	With this lemma in hand, we can conclude the proof.
	Let $\delta \in (0,\theta \wedge \frac{1}{2})$ and $M> 0$.
	Considering $t$ large enough such that $\delta^{-1}r \leq r_t$, we decompose the right-hand side of Equation \eqref{eq:esp_as_int} as
	\begin{align*}
	2 \int_{[\delta r,\delta^{-1}r]} j_{t,v,r}^{<M}(s) \diff s
	+2 \int_{[\delta r,\delta^{-1}r]} j_{t,v,r}^{\geq M}(s) \diff s
	+2 \int_{[\delta r,\delta^{-1}r]^c} j_{t,v,r}(s) \diff s,
	\end{align*}
	where the two last terms are negligible.
	Indeed, it is proved in \cite[Equation (5.41)]{cortineshartunglouidor2019} (and the paragraph around) that
	\[
	\limsup_{\eta\to0}
	\limsup_{M\to\infty}
	\limsup_{v\to\infty}
	\limsup_{t\to\infty}
	\frac{t}{\e^{\sqrt{2}v}} \cdot \left(
	2 \int_{[\eta v^2,\eta^{-1}v^2]} j_{t,v}^{\geq M}(s) \diff s
	+2 \int_{[\eta v^2,\eta^{-1}v^2]^c} j_{t,v}(s) \diff s
	\right)
	= 0,
	\]
	and therefore
	\begin{equation} \label{eq:negligible}
	\limsup_{\delta\to0}
	\limsup_{M\to\infty}
	\limsup_{r\to\infty}
	\limsup_{t\to\infty}
	\sup_{v \in [\theta \sqrt{r},\theta^{-1} \sqrt{r}]}
	\left(
	2 \int_{[\delta r,\delta^{-1}r]} j_{t,v,r}^{\geq M}(s) \diff s
	+2 \int_{[\delta r,\delta^{-1}r]^c} j_{t,v,r}(s) \diff s
	\right)
	=0.
	\end{equation}
	Now, setting $S_\delta \coloneqq [\delta r,(1-\delta)r] \cup [(1+\delta)r,\delta^{-1}r]$, we have
	\[
	2 \int_{[\delta r,\delta^{-1}r]} j_{t,v,r}^{<M}(s) \diff s
	= 2 \int_{S_\delta} j_{t,v,r}^{<M}(s) \diff s
	+2 \int_{(1-\delta)r}^{(1+\delta)r} j_{t,v,r}^{<M}(s) \diff s.
	\]
	This last integral is also negligible for $\delta$ small enough (in the same sense as in Equation \eqref{eq:negligible}), because
	\[
	j_{t,v,r}^{<M}(s) 
	\leq j_{t,v}(s) 
	= j_{t,v}^{\geq 0}(s) 
	\leq C\, \frac{(v+1) \e^{\sqrt{2}v}}{s^{3/2}t}
	\leq C(\theta)\, \frac{\e^{\sqrt{2}v}}{rt} \, ,
	\]
	for any $M > 0$, $0 \leq r \leq t/4$, $s \in [(1-\delta)r,(1+\delta)r]$ and $v \in [\theta \sqrt{r},\theta^{-1} \sqrt{r}]$ by \cite[Lemma~5.5]{cortineshartunglouidor2019}.
	Finally, using Lemma \ref{lem:j}, we have, as $t \to\infty$, then $r \to \infty$ and then $M\to \infty$, uniformly in $v \in [\theta \sqrt{r},\theta^{-1} \sqrt{r}]$,
	\begin{align*}
	2 \int_{S_\delta} j_{t,v,r}^{<M}(s) \diff s
	\sim \frac{2 C_2}{t}\, v \e^{\sqrt{2} v} 
	\int_{S_\delta} \frac{\e^{-v^2/(2s)}}{s^{3/2}}\,
	\varphi_{b,B} \left( \frac{r\vee s}{\abs{r-s}} \right) \diff s.
	\end{align*}
	Coming back to Equation \eqref{eq:decond} and letting $\delta \to 0$, this proves that there exists $r_0 > 0$ such that, for any $r \geq r_0$, there exists $t_0 > 0$ such that, for any $t \geq t_0$ and any $v \in [\theta \sqrt{r},\theta^{-1} \sqrt{r}]$,
	\begin{equation} \label{eq:C_on_B}
	\abs{ \hEc{\cC_{t,r_t}^*([-v,0]) \1_{B_{r,t}}} 
		- \frac{2 C_2}{C_1} v\e^{\sqrt{2} v}  \int_0^\infty 
		\varphi_{b,B} \left( \frac{r\vee s}{\abs{r-s}} \right)
		\frac{\e^{-v^2/(2s)}}{s^{3/2}} \diff s} 
	\leq \varepsilon \, \e^{\sqrt{2} v}.
	\end{equation}
	Using the relation%
	\footnote{
		The authors of \cite{cortineshartunglouidor2019} do not keep precisely track of constants but this relation can be deduced from a careful reading of their paper. Alternatively letting $b \to 0$ and $B \to \infty$, we have $\1_{B_{r,t}} \to 1$ and $\varphi_{b,B} \to 1$ and Equation \eqref{eq:C_on_B} gives an alternative proof of their \cite[Lemma 5.2]{cortineshartunglouidor2019}, showing that the constant $C$ appearing there equals $2 \sqrt{2\pi}C_2/C_1$. Then, a quick look at the proof of  \cite[Proposition 1.5]{cortineshartunglouidor2019} ensures that this constant $C$ is actually $C_\star$.
	} 
	$C_\star = 2 \sqrt{2\pi}\,C_2/C_1$ gives the desired result. 
\end{proof}

\medskip

\begin{lem}[BBM case]
 \label{lem:j}
	Let $a > 0$, $0 < b < B$ and $\delta \in (0,1/2)$. 
	As $t \to\infty$, then $r \to \infty$ and then $M\to \infty$, we have 	
	\[
	j_{t,v,r}^{<M}(s)
	\sim \frac{C_2}{t s^{3/2}}\, v \e^{\sqrt{2}v - v^2/(2s)} \,\varphi_{b,B} \left( \frac{r\vee s}{\abs{r-s}} \right),
	\]
	uniformly $s \in [\delta r,(1-\delta)r] \cup [(1+\delta)r,\delta^{-1}r]$ and $v \in [\delta \sqrt{r},\delta^{-1} \sqrt{r}]$, with
	\[
	C_2 \coloneqq \frac{2\sqrt{2}}{\pi} f^{(0)}(0)g(0)
	\int_\R f(z) g(z)^2 \e^{\sqrt{2} z} \diff z \in (0,\infty),
	\]
	where $f^{(0)},f,g$ are positive functions introduced in \cite[Lemma~3.4]{cortineshartunglouidor2019}.
\end{lem}

\medskip

\begin{proof}
	We follow ideas from the proof of  \cite[Lemma 5.6]{cortineshartunglouidor2019}, but instead of only distinguishing according to the value of $\widehat{W}_{t,s}$, we also distinguish according to the value of $\widehat{W}_{t,r}$.
	For comparison, the constant $C$ appearing in the statement of  \cite[Lemma 5.6]{cortineshartunglouidor2019} equals $C_2$ introduced here.
	
	We start with the case $r \leq s$, that is $s \in [(1+\delta)r,\delta^{-1}r]$. Then, we have, with the notations from~\cite{cortineshartunglouidor2019},
	\begin{align*}
	j_{t,v,r}^{<M}(s)
	= \int_{-B\sqrt{r}}^{-b\sqrt{r}} \int_{-M}^M 
	p_t((r,y){;}(s,z))
	q_t((0,0){;}(r,y)) q_t((r,y){;}(s,z)) q_t((s,z){;}(t,0)) 
	e_{s,v}(z) \diff z \diff y.
	\end{align*}
	The function $(y,z) \mapsto p_t((r,y){;}(s,z))$ is the density of $(\widehat{W}_{t,r},\widehat{W}_{t,s})$ given $\widehat{W}_{t,0}=\widehat{W}_{t,t}=0$. It is explicitly given by (recall the definition of $\gamma_{t,s}$ in Equation \eqref{eq:def_hatW})
	\begin{align*}
	p_t((r,y){;}(s,z))
	& = \frac{1}{2 \pi} \sqrt{\frac{t}{r(s-r)(t-s)}}
	\,\exp \left(
	- \frac{(s(y+\gamma_{t,r})-r(z+\gamma_{t,s}))^2}{2rs(s-r)}
	- \frac{t(z+\gamma_{t,s})^2}{2s(t-s)} 
	\right) \\
	& \sim \frac{1}{2 \pi \sqrt{r(s-r)}} 
	\,\exp \left(	- \frac{sy^2}{2r(s-r)} \right),
	\end{align*}
	as $t \to\infty$ and then $r \to \infty$,
	uniformly in $s \in [(1+\delta)r,\delta^{-1}r]$, $y \in [-B\sqrt{r},-b\sqrt{r}]$ and $z \in [-M,M]$.
	Furthermore, it follows from \cite[Lemma 3.4]{cortineshartunglouidor2019} that
	\begin{align*}
	q_t((0,0){;}(r,y)) & = q_r((0,0){;}(r,y)) \sim \frac{2(-y) f^{(0)}(0)}{r}, \\
	q_t((r,y){;}(s,z)) & = q_s((r,y){;}(s,z)) \sim \frac{2(-y)g(z)}{s-r}, \\
	q_t((s,z){;}(t,0)) & \sim \frac{2f^{(s)}(z) g(0)}{t-s} \sim \frac{2f(z) g(0)}{t}, 
	\end{align*}
	as $t \to\infty$ and then $r \to \infty$,
	uniformly in $s$, $y$ and $z$ as before.
	Finally, it follows from \cite[Lemma 4.2]{cortineshartunglouidor2019} (see also \cite[Equation (5.37)]{cortineshartunglouidor2019}) that
	\[
	e_{s,v}(z) \sim v \e^{\sqrt{2} v - v^2/(2s)} \,\frac{g(z)}{\sqrt{\pi}} \,\e^{\sqrt{2} z},
	\]
	as $r \to \infty$, uniformly in $v \in [\delta \sqrt{r},\delta^{-1} \sqrt{r}]$ and $s,z$ as before.
	Altogether, this proves
	\begin{align*}
	j_{t,v,r}^{<M}(s)
	\sim 
	\frac{4}{\pi^{3/2}}\, 
	\frac{f^{(0)}(0)g(0)}{t r^{3/2}(s-r)^{3/2}}\,
	v \e^{\sqrt{2} v - v^2/(2s)} 
	\int_{-B\sqrt{r}}^{-b\sqrt{r}} 
	y^2 \e^{-sy^2/(2r(s-r))}
	\diff y
	\int_{-M}^M f(z) g(z)^2 \e^{\sqrt{2} z} \diff z,
	\end{align*}
	as $t \to\infty$ and then $r \to \infty$,
	uniformly in $s$ and $v$ as before.
	A change of variable in the first integral and letting $M \to \infty$ in the second integral (the fact that this integral converges to a finite limit is justified at the end of the proof of \cite[Lemma 5.6]{cortineshartunglouidor2019}) yields the result.
	
	The case $s \leq r$, that is $s \in [\delta r,(1-\delta)r]$ is similar: we write
	\begin{align*}
	j_{t,v,r}^{<M}(s)
	= \int_{-B\sqrt{r}}^{-b\sqrt{r}} \int_{-M}^M 
	p_t((s,z){;}(r,y))
	q_t((0,0){;}(s,z)) q_t((s,z){;}(r,y)) q_t((r,y){;}(t,0)) 
	e_{s,v}(z) \diff z \diff y
	\end{align*}
	and use the same asymptotics as before, the main difference being
	\begin{align*}
	p_t((s,z){;}(r,y))
	& = \frac{1}{2 \pi} \,\sqrt{\frac{t}{s(r-s)(t-r)}}\,
	\exp \left(
	- \frac{(r(z+\gamma_{t,s})-s(y+\gamma_{t,r}))^2}{2rs(r-s)}
	- \frac{t(y+\gamma_{t,r})^2}{2r(t-r)} 
	\right) \\
	& \sim \frac{1}{2 \pi \sqrt{s(r-s)}} \,
	\exp \left(	- \frac{y^2}{2(s-r)} \right).
	\end{align*}
	This yields the result in that case.
\end{proof}

\subsection{Cross-moments of level sets}
\label{sec:cross_moment}

Our aim in this section is to prove Proposition \ref{prop:crossed}. For this, we follow the proof of \cite[Proposition 1.5]{cortineshartunglouidor2019}, which bounds $\E[\cC([-v,0])^2]$. This proof is based on a series of five lemmas, that we re-state here in a new version tuned for our purpose of dealing with two level sets of different levels $v$ and $v'$.

In the following lemma, as in \cite[Lemma 4.3]{cortineshartunglouidor2019}, we work with a 2-spine BBM defined under the probability measure $\widetilde{\P}^{(2)}$ as follows. It starts with one particle at 0 at time 0 which is part of the spines 1 and 2. Particles belonging to $m$ spines branch at rate $2^m$ and move according to a standard Brownian motion. 
At a branching point, for each $k \in \{1,2\}$, if the parent was part of spine $k$, then one of both children is chosen uniformly at random to be part of spine $k$.
We denote $X_t(k)$ the particle at time~$t$ which is part of the spine $k$.
For $x,y \in L_t$, we write $d(x,y) = r$ if the most recent common ancestor of $x$ and $y$ died at time $t-r$.

\medskip

\begin{lem}[BBM case]
 \label{lem:new_4.3}
	There exists $C > 0$ such that, for any $0 \leq r \leq t$ and $v \leq v' \leq u$,
	\begin{align*}
	& \widetilde{\P}^{(2)} \left( 
	\widehat{h}_t(X_t(1)) \geq v \, , \, 
	\widehat{h}_t(X_t(2)) \geq v'\, , \, 
	\widehat{h}_t^* \leq u
	\mathrel{}\middle|\mathrel{}
	d(X_t(1),X_t(2)) = r 
	\right) \\
	& \leq \frac{C \e^{-t-r}}{1+(r\wedge(t-r))^{3/2}}
	(u_++1) \e^{\sqrt{2}u}
	(u-v+1) (u-v'+1) \e^{-\sqrt{2}(v+v')}
	\left( \e^{-(u-v)^2/(4t)} + \e^{-(u-v)/2} \right).
	\end{align*}
\end{lem}

\medskip

\begin{proof}
	This is a new version of \cite[Lemma 4.3]{cortineshartunglouidor2019} and we explain how to adapt its proof. Similarly as \cite[Equation (4.15)]{cortineshartunglouidor2019}, the probability in the statement equals
	\begin{align}
	& \int_\R 
	\widetilde{\P}\left( \widehat{h}_r(X_r) \geq v-z, \widehat{h}^*_r \leq u-z \right)
	\widetilde{\P}\left( \widehat{h}_r(X_r) \geq v-z, \widehat{h}^*_r \leq u-z \right) \nonumber \\
	& \qquad {} \times 
	\widetilde{\P}\left( \widehat{h}_{t-r}(X_{t-r}) - m_{t,r} \in \diff z, \widehat{h}^*_t(\mathrm{B}(X_t)^c) \leq u \right).
	\label{eq:4.15}
	\end{align}
	Following \cite{cortineshartunglouidor2019}, we split the integral according to $z \leq u$ and $z > u$.
	
	For $z > u$, using \cite[Equation (4.2)]{cortineshartunglouidor2019} for the first one and \cite[Equation (4.3)]{cortineshartunglouidor2019} for the second one, we bound the product of the two first probabilities in Equation \eqref{eq:4.15} by
	\[
	C \e^{-2r} (u-v+1) (u-v'+1) \e^{-\sqrt{2}(v+v')}
	\e^{2\sqrt{2} z - \frac{3}{2}(z-u)}
	\left( \e^{-(v-z)^2/(4t)} + \e^{(v-z)/2} \right).	
	\]
	This is exactly the same as \cite[Equation (4.20)]{cortineshartunglouidor2019} up to the factor $(u-v'+1) \e^{-\sqrt{2}v'}$ where primes have been added.
	Therefore, this part of the integral is dealt with exactly the same way as in~\cite{cortineshartunglouidor2019}.
	
	For $z \leq u$, we use \cite[Equation (4.2)]{cortineshartunglouidor2019} for both first probabilities in Equation \eqref{eq:4.15} and note that, because $v \leq v'$,
	\begin{align*}
	\left( \e^{-(v-z)^2/(4t)} + \e^{(v-z)/2} \right)
	\left( \e^{-(v'-z)^2/(4t)} + \e^{(v'-z)/2} \right)
	& \leq \left( \e^{-(v-z)^2/(4t)} + \e^{(v-z)/2} \right)
	\left( 1 + \e^{(v-z)/2} \right) \\
	& \leq \left( \e^{-(v-z)^2/(4t)} + \e^{v-z} + 2 \e^{(v-z)/2} \right).
	\end{align*}
	This shows the product of the two first probabilities in Equation \eqref{eq:4.15} is at most
	\[
	C \e^{-2r} (u-v+1) (u-v'+1) \e^{-\sqrt{2}(v+v')}
	\e^{2\sqrt{2} z} (u-z+1)^2
	\left( \e^{-(v-z)^2/(4t)} + \e^{(v-z)/2} + \e^{(v-z)/2} \right).	
	\]
	There are two differences with \cite[Equation (4.19)]{cortineshartunglouidor2019}: the factor $(u-v'+1) \e^{-\sqrt{2}v'}$ where primes have been added (but this adds no new difficulty), and the additional term $\e^{(v-z)/2}$ in the last parentheses.
	This latter gives rise to the following new term, which should be added to the integral in \cite[Equation (4.21)]{cortineshartunglouidor2019},
	\[
	\int_{-\infty}^u \e^{\frac{v}{2} + (\sqrt{2} - \frac{1}{2})z)} (u-z+1)^3 \diff z
	\leq C \e^{\sqrt{2} u} \e^{-(u-v)/2},
	\]
	which can be included in the upper bound of the statement of the lemma after taking care of the other factors (see \cite[Equation (4.22)]{cortineshartunglouidor2019}: this is exactly how another part of the integral in \cite[Equation (4.21)]{cortineshartunglouidor2019} is bounded).
\end{proof}

\medskip

\begin{lem}[BBM case]
 \label{lem:new_4.4}
	There exists $C > 0$ such that, for any $t \geq 0$ and $v \leq v' \leq u$,
	\begin{align*}
	& \Ec{\cE_t([v,\infty)) \cE_t([v',\infty)) ; 
		\widehat{h}_t^* \leq u } \\
	& \leq C (u_++1) \e^{\sqrt{2}u}
	(u-v+1) (u-v'+1) \e^{-\sqrt{2}(v+v')}
	\left( \e^{-(u-v)^2/(4t)} + \e^{-(u-v)/2} \right).
	\end{align*}
\end{lem}

\medskip

\begin{proof}
	This follows from Lemma \ref{lem:new_4.3} in the same way as \cite[Lemma 4.4]{cortineshartunglouidor2019} follows from \cite[Lemma 4.3]{cortineshartunglouidor2019}.
\end{proof}

\medskip

\begin{lem}[BBM case]
 \label{lem:new_5.4}
	For any $v, v' \geq 0$,
	\begin{align*}
	& \widetilde{\E} \left[
	\cC_{t,r_t}^*([-v,0]) \, \cC_{t,r_t}^*([-v',0]) \, ; \, \widehat{h}_t^* \leq 0
	\mathrel{}\middle|\mathrel{} \widehat{h}_t(X_t) = 0 \right] \\
	& = 4 \int_0^{r_t} \int_0^{r_t} j_{t,v,v'}(s,s') \diff s' \diff s
	+ 2 \int_0^{r_t} j_{t,v,v'}(s,s) \diff s,
	\end{align*}
	where $j_{t,v,v'}(s,s') \coloneqq \widehat{\E}_t[J_{t,v}(s)J_{t,v'}(s')]$.
\end{lem}

\medskip

\begin{proof}
	This follows directly from the proof of  \cite[Lemma 5.4]{cortineshartunglouidor2019}.
\end{proof}

\medskip

\begin{lem}[BBM case]
 \label{lem:new_5.5}
	There exists $C > 0$ such that, for any $t \geq 0$, $s,s' \in [0,t/2]$ and $v \geq v' \geq 0$,
	\begin{align*}
	j_{t,v,v'}(s,s')
	& \leq \frac{C (v+1) (v'+1) \e^{\sqrt{2}(v+v')}}{t(s \wedge s' +1)\sqrt{s \wedge s'} (\abs{s'-s}+1) \sqrt{\abs{s'-s}+\1_{s=s'}}}
	\left( \e^{-v^2/(16s)} + \e^{-v/4} \right).
	\end{align*}
\end{lem}

\medskip

\begin{proof}	
	This is a new version of  \cite[Lemma 5.5]{cortineshartunglouidor2019}.
	
	For the case $s \neq s'$, we first assume that $s < s'$ but do not assume $v \geq v'$. Then it follows directly from the proof in~\cite{cortineshartunglouidor2019} that
	\begin{align*}
	j_{t,v,v'}(s,s')
	& \leq \frac{C (v+1) (v'+1) \e^{\sqrt{2}(v+v')}}{t(s+1) \sqrt{s} (s'-s+1) \sqrt{s'-s+\1_{s=s'}}}
	\left( \e^{-v^2/(16s)} + \e^{-v/4} \right)
	\left( \e^{-(v')^2/(16s')} + \e^{-v'/4} \right).
	\end{align*}
	Then, we use $j_{t,v,v'}(s,s') = j_{t,v',v}(s',s)$ to cover the case $s > s'$
	(this is fine because we removed the assumption $v \geq v'$).
	Finally, for $v \geq v'$, we bound $\e^{-(v')^2/(16s')} + \e^{-v'/4} \leq 2$. This yields the desired result.
	
	The case $s=s'$ is also identical to the proof of \cite[Lemma 5.5]{cortineshartunglouidor2019}, applying here Lemma \ref{lem:new_4.4} instead of \cite[Lemma 4.4]{cortineshartunglouidor2019}.
\end{proof}

\medskip

\begin{lem}[BBM case]
 \label{lem:new_5.3}
	There exists $C > 0$ such that, for any $t \geq 1$ and $v \geq v' \geq 0$,
	\begin{align*}
	\widetilde{\E} \left[
	\cC_{t,r_t}^*([-v,0]) \cC_{t,r_t}^*([-v',0])
	\mathrel{}\middle|\mathrel{} 
	\widehat{h}_t^* = \widehat{h}_t(X_t) = 0 \right] 
	& \leq C (v'+1) \e^{-\sqrt{2}(v+v')}.
	\end{align*}
\end{lem}

\medskip

\begin{proof}
	Proceeding as in Equations \eqref{eq:decond_1}, \eqref{eq:proba_denom},  \eqref{eq:moment_1} and \eqref{eq:moment_1_bis}%
	\footnote{For this step note that we end up here with $\E[(\# L_{t_0})^2]$, which is also a finite constant depending on $t_0$ using similar arguments, including the fact that the second moment of the number of particles at time $r$ in a standard BBM is finite (more precisely it equals $2\e^{2r}-\e^r$).}, 
	it is enough to prove that there exist $C > 0$ and $t_0 \geq 0$ such that, for any $t \geq t_0$ and $v \geq 0$,
	\begin{equation} \label{eq:moment_2}
	\widetilde{\E} \left[ \cC_{t,r_t}^*([-v,0]) \cC_{t,r_t}^*([-v',0]) \1_{\widehat{h}_t^* \leq 0} \mathrel{}\middle|\mathrel{} 
	 \widehat{h}_t(X_t) = 0 \right] 
	\leq \frac{C}{t} (v'+1) \e^{\sqrt{2} v}.
	\end{equation}
	We choose $t_0$ such that, for any $t \geq t_0$, $r_t \leq t/2$.
	Using Lemmas \ref{lem:new_5.4} and \ref{lem:new_5.5}, the left-hand side of Equation \eqref{eq:moment_2} is at most
	\begin{align*}
	& \frac{C}{t} (v+1) (v'+1) \e^{\sqrt{2}(v+v')} \Biggl( 
	\int_0^{r_t} \int_0^{r_t} \frac{(\e^{-v^2/(16s)} + \e^{-v/4})}{(s \wedge s' +1)\sqrt{s \wedge s'} (\abs{s'-s}+1) \sqrt{\abs{s'-s}}} \diff s' \diff s \\
	& \hspace{4.5cm} {} + 
	\int_0^{r_t} \frac{(\e^{-v^2/(16s)} + \e^{-v/4})}{(s+1)\sqrt{s}} \diff s
	\Biggr) \\
	& \leq \frac{C}{t} (v+1) (v'+1) \e^{\sqrt{2}(v+v')} \int_0^{r_t} \frac{(\e^{-v^2/(16s)} + \e^{-v/4})}{(s+1)\sqrt{s}} \diff s,
	\end{align*}
	by integrating w.r.t.\@ $s'$ first. Then, proceeding as in the proof of  \cite[Lemma 5.3]{cortineshartunglouidor2019}, this last integral is at most $C/(v+1)$ and this yields the result.
\end{proof}

\medskip

\begin{proof}[Proof of Proposition \ref{prop:crossed}]
	This follows from Lemma \ref{lem:new_5.3} in the same way as  \cite[Proposition 1.5]{cortineshartunglouidor2019} follows from \cite[Lemma 5.3]{cortineshartunglouidor2019}.
\end{proof}

\subsection{Small moments of \texorpdfstring{$S_\beta$}{Sbeta}}
\label{sec:small_moments}

\begin{lem}[BBM case]
 \label{lem:small_moments_at_t}
	For any $\gamma \in (0,1)$, there exists $C = C(\gamma) > 0$ and $t_0 \geq 0$, such that, for any $t \geq t_0$ and any $\beta \in (1,2]$,
	\[
	\widetilde{\E}_t \left[	\cC_{t,r_t}^*(f_\beta)^\gamma \right]
	\leq \begin{cases}
	C(\beta-1)^{1-2\gamma} & \text{if } \gamma \in (1/2,1), \\
	C \log \frac{1}{\beta-1} & \text{if } \gamma = 1/2, \\
	C & \text{if } \gamma \in (0,1/2),
	\end{cases} 
	\]
	where $f_\beta \colon x \in \R \mapsto \e^{\beta \sqrt{2} x}$.
\end{lem}

\begin{proof}
	Proceeding as in Equations \eqref{eq:decond_1} and \eqref{eq:proba_denom}, it is enough to prove that there exist $C > 0$ such that, for any $t \geq t_0$ and $\beta \in (1,2]$,
	\begin{equation} \label{eq:goal_small_moments}
	\widetilde{\E} \left[ \cC_{t,r_t}^*(f_\beta)^\gamma \1_{\{\widehat{h}_t^* \leq 0\}} \mathrel{}\middle|\mathrel{} 
	\widehat{h}_t(X_t) = 0 \right] 
	\leq \frac{C}{t} \times
	\begin{cases}
	(\beta-1)^{1-2\gamma} & \text{if } \gamma \in (1/2,1), \\
	\log \frac{1}{\beta-1} & \text{if } \gamma = 1/2, \\
	1 & \text{if } \gamma \in (0,1/2),
	\end{cases}
	\end{equation}
	where $t_0$ is chosen such that, for any $t \geq t_0$, $r_t \leq t/2$.
	
	We first decompose $\cC_{t,r_t}^*(f_\beta)$ along the spine as it is done for level sets in \cite[Lemma 5.4]{cortineshartunglouidor2019}.
	Setting (this replaces $J_{t,v}(s)$ defined in \cite[Equation (5.9)]{cortineshartunglouidor2019})
	\[
		K_{t,\beta}(s) \coloneqq \cE_s^s(f_\beta(\,\cdot\,+\widehat{W}_{t,s})) 
		\times \1_{\{\widehat{h}_s^{s*} \leq -\widehat{W}_{t,s}\}} 
		\times \1_{\cA_t},
	\]
	and $\widehat{\P}_t \coloneqq \P( \,\cdot\, | \widehat{W}_{t,0} = \widehat{W}_{t,t} = 0)$,
	we have 
	\[
		\cC_{t,r_t}^*(f_\beta) \1_{\{\widehat{h}_t^* \leq 0\}} \text{ under } \widetilde{\P}( \,\cdot\,|\widehat{h}_t(X_t)=0)  
		\overset{(\mathrm{d})}{=} \int_0^{r_t} K_{t,\beta}(s) \cN(\diff s) \text{ under } \widehat{\P}_t,
	\]
	where $\cN$ is a Poisson point process on $\R_+$ with intensity $2\diff s$.
	Using subadditivity of $x \mapsto x^\gamma$ (note that the integral above is actually a finite sum) and then proceeding as in the proof of \cite[Lemma 5.4]{cortineshartunglouidor2019}, we get
	\begin{equation} \label{eq:link_to_k}
	\widetilde{\E} \left[ \cC_{t,r_t}^*(f_\beta)^\gamma \1_{\{\widehat{h}_t^* \leq 0\}} \mathrel{}\middle|\mathrel{} 
	\widehat{h}_t(X_t) = 0 \right] 
	\leq \widehat{\E}_t \left[ \int_0^{r_t} K_{t,\beta}(s)^\gamma \cN(\diff s) \right]
	= 2 \int_0^{r_t} k_{t,\beta,\gamma}(s) \diff s,
	\end{equation}
	with $k_{t,\beta,\gamma}(s) \coloneqq \widehat{\E}_t [K_{t,v}(s)^\gamma]$.
	
	We now aim at proving that there exists $C > 0$ such that, for any $t \geq 0$, $s \in [0,t/2]$ and $\beta \in (1,2]$, 
	\begin{equation} \label{eq:goal_k}
	k_{t,\beta,\gamma}(s) 
	\leq \frac{C}{t(s+1)\sqrt{s}} \left( (s+1) \wedge(\beta-1)^{-2} \right)^\gamma.
	\end{equation}
	For this, we follow the ideas from the proof of \cite[Lemma 5.5]{cortineshartunglouidor2019} (in the case $M=0$). 
	Conditioning on $\widehat{W}_{t,s}$, we get, similarly as \cite[Equation (5.20)]{cortineshartunglouidor2019},
	\begin{equation} \label{eq:decompo_k}
	k_{t,\beta,\gamma}(s) 
	= \int_\R q_t((0,0){;}(s,z))
	\Ec{\cE_s(f_\beta(\,\cdot\,+z))^\gamma \1_{\{\widehat{h}_s^* \leq -z\}} }
	q_t((s,z){;}(t,0)) p_t(s,z) \diff z.
	\end{equation}
	The single difference with \cite{cortineshartunglouidor2019} is inside the expectation.
	Using Jensen's inequality and then writing $f_\beta(x+z) = \int_0^\infty \beta \sqrt{2} \e^{-\beta\sqrt{2} v} \1_{\{x \geq -v-z\}}\diff v$ for $x \leq -z$, we get
	\begin{align}
	\Ec{\cE_s(f_\beta(\,\cdot\,+z))^\gamma \1_{\{\widehat{h}_s^* \leq -z\}} }^{1/\gamma}
	& \leq \Ec{\cE_s(f_\beta(\,\cdot\,+z)) \1_{\{\widehat{h}_s^* \leq -z\}} } \nonumber \\
	& = \int_0^\infty \beta \sqrt{2} \e^{-\beta\sqrt{2} v} \Ec{\cE_s([-v,0]-z) \1_{\{\widehat{h}_s^* \leq - z\}}} \diff v. \label{eq:new_exp}
	\end{align}
	Using \cite[Lemma 4.2]{cortineshartunglouidor2019}, the right-hand side of Equation \eqref{eq:new_exp} is at most
	\begin{align*}
	& C (z_-+1) \e^{\sqrt{2} z} 
	\int_0^\infty (v+1) \e^{-(\beta-1)\sqrt{2} v} 
	\left( \e^{-(v+z)^2/4s} + \e^{-(v+z)/2} \right) \diff v. 
	\end{align*}
	The part of the last integral due to the term $\e^{-(v+z)/2}$ is bounded by $C\e^{-z/2}$. For the other part, it can be bounded by $\int_0^\infty (v+1) \e^{-(\beta-1)\sqrt{2} v} \diff v \leq C(\beta-1)^{-2}$ or by $\int_\R (v+1) \e^{-(v+z)^2/4s} \diff v \leq C (\abs{z}+1) (s+1)$. 
	Therefore, we proved
	\[
	\Ec{\cE_s(f_\beta(\,\cdot\,+z))^\gamma \1_{\{\widehat{h}_s^* \leq -z\}} }
	\leq C (\abs{z}+1)^{2\gamma} \e^{\gamma \sqrt{2} z} \left( \e^{-z/2} + \left( (s+1) \wedge(\beta-1)^{-2} \right) \right)^\gamma.
	\]
	Coming back to Equation \eqref{eq:decompo_k} and bounding the other factors in the integrand as in \cite[Equations (5.23) and (5.24)]{cortineshartunglouidor2019}, we get
	\[
	k_{t,\beta,\gamma}(s) 
	\leq \frac{C}{t(s+1)\sqrt{s}} 
	\int_\R \left( z_-+\e^{-\frac{3}{2}z_+} \right) 
	(\abs{z}+1)^{2\gamma+1} \e^{\gamma \sqrt{2} z} 
	\left( \e^{-\frac{\gamma}{2}z} + \left( (s+1) \wedge(\beta-1)^{-2} \right)^\gamma \right) \diff z,
	\]
	and Equation \eqref{eq:goal_k} follows.

	Finally, we come back to Equation \eqref{eq:link_to_k} and apply Equation \eqref{eq:goal_k} (we use here $r_t \leq t/2$) to get
	\begin{align*} 
	\widetilde{\E} \left[ \cC_{t,r_t}^*(f_\beta)^\gamma \1_{\widehat{h}_t^* \leq 0} \mathrel{}\middle|\mathrel{} 
	\widehat{h}_t(X_t) = 0 \right] 
	& \leq \frac{C}{t}
	\int_0^\infty \frac{1}{(s+1)\sqrt{s}} 
	\left( (s+1) \wedge(\beta-1)^{-2} \right)^\gamma \diff s \\
	& \leq \frac{C}{t} \left( 
	\int_0^{(\beta-1)^{-2}} \frac{\diff s}{(s+1)^{1-\gamma} \sqrt{s}} 
	+ (\beta-1)^{-2\gamma} \int_{(\beta-1)^{-2}}^\infty \frac{\diff s}{(s+1)\sqrt{s}} 
	\right)
	\end{align*}
	and Equation \eqref{eq:goal_small_moments} follows.
\end{proof}

\section{Asymptotic expansion of an integral}
\label{sec:app_integral}

\begin{proof}[Proof of Equation \eqref{eq:expansion_integral}]
	Write $h_\varepsilon(x) = x^\gamma - \varepsilon x$ for any $x \geq 0$.
	Recall we aim at estimating
	\[
	I_k \coloneqq \int_0^\infty x^k \e^{h_\varepsilon(x)} \diff x,
	\]
	for $k \in \{0,1,2\}$.
	The function $h_\varepsilon$ is maximized at $x_0 \coloneqq (\gamma/\varepsilon)^{1/(1-\gamma)}$, and we can rewrite, for $x \geq -x_0$,
	\[
	h_\varepsilon(x_0+x) - h_\varepsilon(x_0)
	= -x_0^\gamma \left( 1 + \gamma \frac{x}{x_0}
	- \left( 1 + \frac{x}{x_0} \right)^\gamma \right).
	\]
	Let $\delta \in (0,1)$. 
	Setting $f \colon t \in [-1,\infty) \mapsto 1+\gamma t - (1+t)^\gamma$, there exists $c>0$ depending only on $\gamma$ such that $f(\pm\delta) \geq c\delta^2$. Together with convexity of $f$, this proves that $f(t) \geq c\delta \abs{t}$ for any $t \in[-1,\infty) \setminus (-\delta,\delta)$.
	From this inequality, choosing $\delta = x_0^{-\theta}$ for some $\theta \in (0,\gamma/2)$ so that $\delta \to 0$ and $\delta^2 x_0^\gamma \to \infty$ as $\varepsilon \to 0$, we deduce that, up to a modification of the constant $c$,
	\begin{equation} \label{eq:int_exp_1}
	\int_0^\infty x^k \e^{h_\varepsilon(x)} \diff x 
	= \e^{h_\varepsilon(x_0)}
	\left(  \int_{-\delta x_0}^{\delta x_0} (x_0+x)^k \e^{h_\varepsilon(x_0+x) - h_\varepsilon(x_0)} \diff x 
	+ O \left( \e^{-c \delta^2 x_0^\gamma} \right)\right).
	\end{equation}
	Uniformly in $x \in [-\delta x_0,\delta x_0]$, we can expand
	\begin{equation} \label{eq:exp_h}
	h_\varepsilon(x_0+x) - h_\varepsilon(x_0)
	= - \frac{x^2}{2 \sigma^2} + a_3 \frac{x^3}{x_0^{3-\gamma}}
	- a_4 \frac{x^4}{x_0^{4-\gamma}}
	+ O \left( \frac{x^5}{x_0^{5-\gamma}} \right),
	\end{equation}
	where $a_3, a_4 > 0$ are explicit in terms of $\gamma$ and
	\[
		\sigma^2 \coloneqq \frac{x_0^{2-\gamma}}{\gamma(1-\gamma)}
		= \frac{\gamma^{\frac{1}{1-\gamma}}}{1-\gamma} \varepsilon^{-\frac{2-\gamma}{1-\gamma}}.
	\]
	Now we also choose $\theta > \gamma/3$ so that $x^3/x_0^{3-\gamma} \to 0$ uniformly in $x \in [-\delta x_0,\delta x_0]$, and we can expand the $\e^{h_\varepsilon(x_0+x) - h_\varepsilon(x_0)}$ to get that \eqref{eq:int_exp_1} equals
	\[
	\e^{h_\varepsilon(x_0)}
	\left( \int_\R (x_0+x)^k \e^{-x^2/2\sigma^2}
	\left( 1 + a_3 \frac{x^3}{x_0^{3-\gamma}}
	- a_4 \frac{x^4}{x_0^{4-\gamma}}
	+ O \left( \frac{x^5}{x_0^{5-\gamma}} \right)\right) \diff x 
	+ O \left( \e^{-c \delta^2 x_0^\gamma} \right)\right),
	\]
	where we changed the domain of integration up to an error which enters the last $O$-term.
	Expanding $(x_0+x)^k$ depending on the value of $k$, changing variable with $y = x/\sigma$, computing explicitly the resulting Gaussian integrals and using that $\sigma^2/x_0^{2-\gamma} = 1/[\gamma(1-\gamma)]$, we get
	\begin{align*}
	I_0 & = \sqrt{2\pi} \sigma \e^{h_\varepsilon(x_0)}
	\left( 1 - \frac{3a_4}{\gamma(\gamma-1)} \frac{\sigma^2}{x_0^2} 
	+ O \left( \frac{\sigma^3}{x_0^3}  \right) \right), \\
	I_1 & = \sqrt{2\pi} \sigma x_0 \e^{h_\varepsilon(x_0)}
	\left( 1 + \frac{3(a_3-a_4)}{\gamma(\gamma-1)} \frac{\sigma^2}{x_0^2} 
	+ O \left( \frac{\sigma^3}{x_0^3}  \right) \right), \\
	I_2 & = \sqrt{2\pi} \sigma x_0^2 \e^{h_\varepsilon(x_0)}
	\left( 1 + \left( \frac{3(2a_3-a_4)}{\gamma(\gamma-1)} + 1 \right) \frac{\sigma^2}{x_0^2} 
	+ O \left( \frac{\sigma^3}{x_0^3}  \right) \right),
	\end{align*}
	and this proves \eqref{eq:expansion_integral}.
\end{proof}

\begin{proof}[Proof of Equation \eqref{eq:equiv_integral}]
	We follow the same argument as in the previous proof.
	Recall we are in the regime $\beta \to 1^+$ and that $\varepsilon \sim (\beta-1)$.
	Uniformly in $x \in [-\delta x_0,\delta x_0]$, we have, for some $c>0$,
	\begin{align*}
	\log \left( 1+\e^{\frac{2\beta}{2-\beta} (x_0+x)^\gamma-\varepsilon \beta (x_0+x)} \right)
	& = \left( \frac{2\beta}{2-\beta} (x_0+x)^\gamma-\varepsilon \beta (x_0+x) \right) 
	+ O \left( e^{-c x_0^\gamma} \right) \\
	& = \left( \frac{2\beta}{2-\beta} -\beta\gamma \right) x_0^\gamma 
	+ 
	\frac{\beta^2\gamma}{2-\beta} \frac{x}{x_0^{1-\gamma}} 
	+ O\left( \frac{x^2}{x_0^{2-\gamma}} \right)
	+ O \left( e^{-c x_0^\gamma} \right),
	\end{align*}
	using that $\varepsilon = \gamma x_0^{\gamma-1}$.
	Combining this with the argument of the previous proof (using in particular \eqref{eq:exp_h} up to the order $x^4$), we get, with $b = b(\gamma,\beta) \coloneqq \frac{\beta^2\gamma}{2-\beta} (\frac{2\beta}{2-\beta} -\beta\gamma)^{-1}$
	\begin{align*}
	& \int_0^\infty 
	\log^k \left( 1+\e^{\frac{2\beta}{2-\beta} x^\gamma-\varepsilon \beta x} \right) \e^{x^\gamma-\varepsilon x} \diff x \\
	& = \e^{h_\varepsilon(x_0)} 
	\left( \frac{2\beta}{2-\beta} -\beta\gamma \right)^k x_0^{k\gamma} \\
	& \quad {} \times
	\left( \int_\R 
	\left( 1 + k b \frac{x}{x_0} 
	+ O\left( \frac{x^2}{x_0^2} \right) \right)
	\e^{-x^2/2\sigma^2} 
	\left( 1+ a_3 \frac{x^3}{x_0^{3-\gamma}}
	+ O \left( \frac{x^4}{x_0^{4-\gamma}} \right) \right)
	\diff x 
	+ O \left( \e^{-c \delta^2 x_0^\gamma} \right) \right) \\
	& = \sqrt{2\pi} \sigma (2x_0^\gamma-\varepsilon x_0)^k \e^{h_\varepsilon(x_0)} \left( 1 + O \left( \frac{\sigma^2}{x_0^2} \right) \right),
	\end{align*}
	noting that the integrals involving an odd power of $x$ equal zero and using that $\sigma^2/x_0^{2-\gamma} = O(1)$.
	Using the definition of $x_0$ and $\sigma$, \eqref{eq:equiv_integral} follows.
\end{proof}

\end{appendix}

\medskip

\section*{Acknowledgements}

The authors warmly thank Bernard Derrida for very stimulating discussions. Michel Pain keenly thanks Jean-Philippe Bouchaud for presenting him the reference \cite{salesbouchaud01}. We are grateful to Xinxin Chen for suggesting us the content of Remark \ref{rem:1-stable_2}. We also thank an anonymous referee for several interesting questions and suggestions which helped us to improve the readability of the paper.

\medskip

\bibliographystyle{abbrv}

\end{document}